\documentclass[12pt]{article}

\usepackage[margin=2.5cm, top=2.5cm, bottom=2.5cm]{geometry}

\usepackage[utf8]{inputenc}
\usepackage[T1]{fontenc}
\usepackage{soul}
\usepackage{microtype}
\usepackage{amsmath,amssymb,amsfonts,amsthm}
\usepackage{nicefrac}
\usepackage{array,booktabs}
\usepackage{graphicx}
\usepackage{subcaption}
\usepackage{enumitem}
\usepackage{changepage}
\usepackage{stmaryrd}
\usepackage{url}
\usepackage{thmtools}
\usepackage{xspace}
\usepackage[hidelinks,hypertexnames=false]{hyperref}
\usepackage[capitalise]{cleveref}
\crefname{equation}{}{}
\crefname{enumi}{}{}
\Crefname{enumi}{}{}
\crefformat{enumi}{#2#1#3}
\Crefformat{enumi}{#2#1#3}
\allowdisplaybreaks
\newcommand{\RIP}{\mathrm{RIP}}
\newtheorem{theorem}{Theorem}[section]
\declaretheorem[
  name=Fact,          
  sibling=theorem,      
  refname={Fact,Facts}   
]{fact}
\declaretheorem[name=Lemma,       sibling=theorem, refname={Lemma,Lemmas}]{lemma}
\declaretheorem[name=Corollary,   sibling=theorem, refname={Corollary,Corollaries}]{corollary}

\declaretheorem[name=Proposition, sibling=theorem, refname={Proposition,Propositions}]{proposition}

\declaretheorem[name=Example,     sibling=theorem, refname={Example,Examples}]{example}

 \declaretheorem[
  name=Definition,
  sibling=theorem,
  refname={Definition,Definitions},
  style=definition
]{definition}

\newcommand{\N}{\mathbb{N}}
\newcommand{\R}{\mathbb{R}}
\newcommand{\eR}{\overline{\R}}
\newcommand{\p}{^}

\newcommand{\cP}{\mathcal{P}}
\newcommand{\rO}{\mathrm{O}}
\newcommand{\rA}{\mathrm{A}}
\newcommand{\fg}{\mathfrak{g}}
\newcommand{\rnum}[1]{\textup{(\romannumeral #1)}}
\newcommand{\Sym}{\mathbb{S}} % real symmetric matrices
\newcommand{\GL}{\mathrm{GL}}
\newcommand{\cL}{\mathcal{L}}
\newcommand{\Orth}{\mathrm{O}}
\newcommand{\cS}{\mathcal{S}}

\newcommand{\cp}{\overline{\partial}}
\newcommand{\bS}{\mathbb{S}}
\newcommand{\Alpha}{\mathrm{A}}
\newcommand{\Beta}{\mathrm{B}}
\newcommand{\Mr}{\R^{m\times n}_r}
\newcommand{\orbit}{\mathcal{O}}

\newcommand{\cM}{\mathcal{M}}
\newcommand{\cN}{\mathcal{N}}
\newcommand{\KL}{K\L{}\xspace}%{\L{}ojasiewicz}
\def\ker{\mathrm{Ker}}
\DeclareMathOperator{\rk}{rk}
\DeclareMathOperator{\tr}{tr}
\DeclareMathOperator{\diag}{diag}
\DeclareMathOperator{\gph}{gph}
\DeclareMathOperator{\epi}{epi}
\DeclareMathOperator{\dom}{dom}
\DeclareMathOperator{\im}{Im}
\DeclareMathOperator{\co}{co}

\DeclareMathOperator{\spa}{span}

\title{\LARGE Computing Kurdyka-\L{}ojasiewicz exponents via composition and symmetry}

\allowdisplaybreaks
\begin{document}

\author{\large C\'edric Josz\thanks{\url{cj2638@columbia.edu,wo2205@columbia.edu}, IEOR, Columbia University, New York.} \and Wenqing Ouyang\footnotemark[\value{footnote}]}

\date{}

\maketitle

\begin{center}
    \textbf{Abstract}
    \end{center}
    \vspace*{-2.5mm}
 \begin{adjustwidth}{0.2in}{0.2in}
 ~~~~ We devise calculus rules for the Kurdyka-\L{}ojasiewicz exponent using the rank theorem and Lie group actions. They apply to a wide class of composite and invariant functions, and are particularly suitable for handling nonisolated local minima. Notably, smoothness plays no role, eschewing gradient and Hessian computations. This provides a unified framework for establishing linear convergence of various algorithms in matrix factorization, $\ell_1$-matrix factorization, matrix sensing, and linear neural networks.
\end{adjustwidth} 
\vspace*{3mm}
\vspace*{3mm}


% 14P10 Semialgebraic set and related spaces
% 34A60 Differential inclusions 
% 49-XX Calculus of variations and optimal control, optimization
% 53-XX Differential geometry
% 32B20 Semi-analytic sets and subanalytic sets
% 49J52: Nonsmooth analysis
% 49J53: Set-valued and variational analysis

\tableofcontents

\section{Introduction}

Given a smooth function $f:\R^ n \to \R_+$, Polyak's seminal paper \cite{polyak1963gradient} provides a simple criterion for linear convergence of gradient descent: 
$$\forall x \in \R^ n, ~~~ |\nabla f(x)| \geq c\sqrt{f(x)},$$
where $c>0$ and $|\cdot|:=\sqrt{\langle \cdot , \cdot \rangle }$ denotes the Euclidean norm. The local version of this inequality is, of course, a special case of the \L{}ojasiewicz gradient inequality \cite{lojasiewicz1963propriete}:
$$\forall x \in B_r(\overline{x}), ~~~ |\nabla f(x)| \geq c f(x)^ \alpha,$$
where $\overline{x}\in \R^ n$, $r,c>0$, and $\alpha\in [0,1)$. Originally conceived for real analytic functions, this inequality was later extended to the $C\p 1$ definable \cite{kurdyka1998gradients} and nonsmooth \cite{bolte2007clarke} settings, leading to what is now known as the Kurdyka-\L{}ojasiewicz (\KL) inequality. The \KL exponent $\alpha$ plays a key role in the convergence rate of algorithms. Roughly speaking (details in \cref{subsec:subanalytic-geometry}):
\begin{enumerate}[label=\rm{(\rm{\roman*})},nosep]
    \item $\alpha \in [0,1/2)$: finite/linear convergence;
    \item $\alpha = 1/2$: linear convergence;
    \item $\alpha \in(1/2,1)$: sublinear convergence.
\end{enumerate}

Determining a \KL exponent $\alpha$ is a challenging task however.
%\cite{kuo1974computation, lichtin1981estimation,vui2018computation,son2007effective,bivia2009local,krasinski2009lojasiewicz,d2005explicit,rodak2011lojasiewicz,paunescu1997lojasiewicz,brzostowski2015lojasiewicz}. It has been studied for holomorphic functions \cite{oka2022lojasiewicz}, for Newton polyhedra \cite{brzostowski2023lojasiewicz}, for bivariate real analytic functions \cite{dinh2025computation}.
It has been the object of several recent works in connection with the Stiefel manifold \cite{liu2016quadratic,liu2019quadratic,wang2023linear}, the inf-projection operation \cite{yu2022kurdyka}, the Hadamard parametrization \cite{ouyang2025kurdyka}, and the square reparametrization \cite{ouyang2025square}. Notably, Li and Pong \cite{li2018calculus} designed several useful calculus rules in the context of optimization. 

In particular, suppose one has a composite structure
$f := g \circ F$
where $g:\R^ m \to\eR$, $\eR:=\R\cup\{\infty\}$, is lower semicontinuous (lsc) and $F:\R^ n \to\R^ m$ is $C\p 1$ smooth. If $g$ has \KL exponent $\alpha$ at $F(\overline{x})$ and $F$ is a submersion at $\overline{x}$, then $f$ has \KL exponent $\alpha$ at $\overline{x}$ \cite[Theorem 3.2]{li2018calculus}. Rebjock and Boumal recently proved a related result. They consider $g:\cN\to \R$ and $F:\cM\to\cN$ where $\cM,\cN$ are Riemannian manifolds. If $g$ is $C\p 2$, $\nabla g(F(\overline{x}))=0$, $\nabla\p 2 g(F(\overline x))\succ 0$, and $F$ is $C\p 2$ with constant rank near $\overline{x}$, then $f$ has \KL exponent $1/2$ \cite[Propositions 2.3, 2.4, 2.8]{rebjock2024fast}. 

Another route relies on the original \L{}ojasiewicz inequality \cite{lojasiewicz1959}:
$$\forall x \in B_\rho(\overline{x}), ~~~  f(x) \geq \kappa\hspace{.4mm}d(x,[f=0])^ \beta,$$
where we refer to $\beta\geq 1$ as a growth exponent, and $\rho,\kappa>0$ are constants. This inequality is generally weaker, but Pham showed that isolated local minima of continuous semi-algebraic functions with growth exponent $\beta$ have \KL exponent $\alpha=1-1/\beta$ \cite[Theorem 4.2]{pham2020local}. Also, if $f$ is $C\p 2$ and has quadratic growth (i.e., $\beta=2$) at a local minimum, then it has \KL exponent $1/2$, regardless of whether it is isolated \cite[Propositions 2.4, 2.8]{rebjock2024fast} (see also \cite[Corollary 3.2]{drusvyatskiy2013second}). Quadratic growth is equivalent to the Morse-Bott property \cite[Definition 1.1]{rebjock2024fast}, which posits that the level set $\cM=[f=f(\overline{x})]$ is a $C\p 1$ embedded submanifold near a local minimum $\overline{x}$ and $\ker \nabla\p 2 f(\overline{x})=T_{\overline{x}}\cM$, where $T_{\overline{x}}\cM$ is the tangent space to $\cM$ at $\overline{x}$.

Nevertheless, several important applications fall outside the scope of these calculus rules. 
%these calculus rules are not always applicable in practice. 
For concreteness, consider rank-one matrix factorization
\begin{equation*}
    \begin{array}{cccc}
         f: & \mathbb{R}^{m}\times \mathbb{R}^{n} & \longrightarrow & \mathbb{R} \\
        & (x,y) & \longmapsto & \|xy^ T - M\|_F^ 2
    \end{array}
    \end{equation*}
%$$f(x,y)=\|xy^ T - M\|_F^ 2$$
where % $(x,y) \in \mathbb{R}^{m}\times \mathbb{R}^{n}$, 
$M \in \mathbb{R}^{m\times n}$ has full rank and $\|\cdot\|_F$ is the Frobenius norm. Then the inner mapping $F(x,y):= xy^ T$ is not a submersion at the global minima. While it does have constant rank near any of them, the Riemannian Hessian of the outer function %$g(\cdot):=\|\cdot-M\|_2$ 
$g:\im F\ni A\to \|A-M\|_F\p 2\in \R$
is not 
%locally optimal 
positive definite
at $F(\overline{x},\overline{y})$ if $(\overline{x},\overline{y})$ is a global minimum of $f$. The minima of $f$ are also not isolated since they are invariant under the action $(0,\infty)\times \R\p m\times \R\p n\ni(t,x,y)\to (xt,y/t)$. Moreover, it is not clear how to establish quadratic growth, especially when considering higher-rank factors. Finally, when $M=0$, the solution set is not an embedded submanifold near the global minimum $(0,0)$.
%Indeed, with $F(x,y):= xy^ T$, one has $dF_{(x,y)}(u,v) = xv^ T+uy^ T$, whose image is a subset of the rank-two matrices of size $m\times n$. 

%Most modern applications are however devoid of isolated local minima, usually due to the presence of continuous symmetries. In particular, the above example is invariant under the action $(0,\infty)\times \R\p m\times \R\p n\ni(t,x,y)\to (xt,y/t)$. 

%The issue is that proving quadratic growth remains a challenge in applications, including the above example, although it has been established in some cases. Another 

In this paper, we seek to bridge the gap with applications. In order to do so, we propose two new calculus rules -- a composition rule and symmetry rule -- using tools from differential geometry. %not previously considered in this context, to the best of our knowledge. 
%The former relaxes the submersion requirement to having constant rank locally, or in other words, to being a submersion after restricting the codomain to the image. 
The former simulateneously generalizes Pong and Li's and Rebjock and Boumal's composition rules. We consider the setting in which
%we allow the outer function $g$ to remain lsc while relaxing the inner mapping $F$ to have constant rank. 
$$f:=g\circ F~\text{where}~g:\R\p m\to\eR~\text{is lsc and}~F:\R\p n\to\R\p m~\text{has constant rank near}~\overline{x}.$$
It enables us to convert both the growth and the \KL exponent from $g$ to $f$. The main ingredient in the proof is naturally the rank theorem. Allowing extended real values in the outer function is crucial even for dealing with compositions of smooth functions $g$ and $F$. Indeed, we will often consider another outer function using the identity $g\circ F =(g+\delta_{\im F})\circ F$, where $\delta$ denotes the indicator function.
%%Allowing extended real values for $g$ is crucial even for dealing with compositions of smooth functions $h$ and $F$. Indeed, $h\circ F =g\circ F$ where $g:= h+\delta_{\im F}$ and

The latter rule considers lsc objectives $f$ that are invariant under general Lie group actions. One only needs to check the growth and \KL inequalities on a supplementary subspace $L$ of the tangent space $T_{\overline{x}}G\overline{x}$ at the point of interest $\overline{x}$. The setup is simple:
$$\forall g\in G,~~ f(g\cdot x)=f(x) ~~~\text{and}~~~ T_{\overline{x}}G\overline{x}+L=\R\p n$$
where $G$ is a Lie group. By specializing $L$ to be the normal space $N_{\overline{x}}G\overline{x}$, we generalize Pham's result to nonisolated local minima of lsc subanalytic functions if the level set is locally homogeneous (i.e., a single orbit) and embedded. Since the Morse-Bott property is tantamount to showing quadratic growth on the normal space, our result can be viewed as an extension to general growth exponents $\beta$ relying on invariance instead of smoothness. Even in the smooth case, this avoids computing derivatives, which can be tedious in applications.

% To the best of our knowledge, a \KL exponent has been stated explicitly in only two cases: exactly parametrized symmetric matrix factorization and matrix sensing \cite[Theorem 1]{bi2022local}. Nonetheless, it follows easily in exactly parametrized symmetric \cite[Lemma 5.4]{tu2016low} \cite[Lemma 7]{ma2020implicit} and asymmetric \cite[Theorem 4.6]{charisopoulos2021low} matrix factorization as they have quadratic growth. One can also apply Rebjock and Boumal's composition rule \cite[Section 1.2]{rebjock2024fast}.
% %The Morse-Bott condition for symmetric exact parametrized matrix completion is proved in \cite[Lemma 7]{ma2020implicit}, which also yields $1/2$ \KL exponent. 
% %These cases turn out to be immediate consequences of our composition rule. 
% The \KL exponent in overparametrized matrix factorization, $\ell_1$-matrix factorization, and matrix sensing with full rank $M$ already follows from Li and Pong's calculus rule, even if it hasn't been stated before. 
% On the other hand, while exactly parametrized asymmetric and symmetric $\ell_1$-matrix factorization are sharp (i.e., linear growth, $\beta=1$) by equivalence of norms (even with sparse noise \cite[Theorem 3]{gong2026certifying}), there is no clear way to convert this into a \KL exponent since they are not $C\p 2$ and $\beta\neq 2$. Determining a \KL exponent in these cases is an immediate consequence of our composition rule, for which it is irrelevant whether the objective is smooth or not.

The new rules enable one to compute a \KL exponent in various problems of interest, as summarized in \cref{table:exponent}.
There are essentially two hard instances that remained untouched by previous work, totaling in 8 cases in \cref{table:exponent}: 1) underparametrized matrix factorization, 2) overparametrized $\ell_1$-matrix factorization and matrix sensing with rank deficient data (asymmetric and symmetric in both instances).

\begin{table}[h!]
    \centering
    \caption{\KL exponent at global minima.}
    \label{tab:example}
    \begin{tabular}{|c|c|c|c|}
        \hline
        Parametrization & Matrix fact. & $\ell_1$-matrix fact. & Matrix sensing RIP \\ \hline
        $X\in\R\p{m\times r}, Y\in \R\p {r\times n}, M\in \R\p{m\times n}$ & $\|XY-M\|_F\p 2$ & $\|XY-M\|_1$ & $\sum_i \langle A_i , XY-M\rangle\p 2$ \\
        \hline
        under $r<\rk(M)$ & 1/2 & ? & ? \\
        exact $r=\rk(M)$ & ~~~~~ 1/2 \cite{charisopoulos2021low,rebjock2024fast} & 0 & 1/2  \\
        over $r>\rk(M)$, full rank $M$ & ~~~~~ 1/2 \cite{li2018calculus,rebjock2024fast} & ~~~~~ 0 \cite{li2018calculus} & ~~~~~ 1/2 \cite{li2018calculus} \\
        over $r>\rk(M)$, rank deficient $M$ & 3/4 (a.e.\hspace{.4mm}1/2) & 1/2 (a.e.\hspace{.4mm}0) & 3/4 (a.e.\hspace{.4mm}1/2) \\
        \hline
        $X\in\R\p{n\times r}, M\in \Sym_+\p n$ & $\|XX\p T-M\|_F\p 2$ & $\|XX\p T -M\|_1$ & $\sum_i \langle A_i , XX\p T-M\rangle\p 2$ \\
        \hline
        under $r<\rk(M)$ & 1/2 & ? & ? \\
        exact $r=\rk(M)$ & %~~~~~~~~~~~ 
        1/2 \cite{tu2016low,ma2020implicit,bi2022local,rebjock2024fast} & 0 & ~~~ 
        1/2 \cite{bi2022local} \\
        over $r>\rk(M)$, full rank $M$ & ~~~~~ 1/2 \cite{li2018calculus,rebjock2024fast} & ~~~~~ 
        0 \cite{li2018calculus} & ~~~~ 
        1/2 \cite{li2018calculus} \\
        over $r>\rk(M)$, rank deficient $M$ & 3/4 & 1/2 & 3/4 \\
        \hline
    \end{tabular}
    \label{table:exponent}
    \vspace{2mm}
\begin{minipage}{0.5\linewidth}
\footnotesize
(`a.e.' means for almost every global minimum.) 
\end{minipage}
\end{table}

The first instance is particularly relevant in data science as it allows one to find a best rank-$r$ approximation of a data matrix $M\in \R\p {m\times n}$ where $r\leq \rk(M)$:
%and perform principal component analysis:
$$\min_{A\in \R\p{m\times n}} \|A-M\|_F\p 2 ~~~ \text{subject to}~~~ \rk(A)=r.$$ By the Eckart-Young theorem \cite{eckart1936approximation}, this problem admits a closed-form solution by truncating a singular value decomposition of $M$, only keeping the top $r$ singular values. The \KL exponent 1/2 in underparametrized matrix factorization, together with the absence of spurious second-order stationary points \cite{baldi1989neural,valavi2020revisiting}, implies linear convergence of gradient descent to a global minimum from almost every initial point. This is due to a general global convergence property of gradient descent \cite{josz2023global}.

The second instance reveals an interesting phenomenon. Rank deficiency in overparametrized matrix sensing causes the \KL exponent to increase from 1/2 to 3/4, yielding the sublinear rate $O(1/k\p 2)$. This pathological behavior affects all the global minima in the symmetric case, but only a negligible subset of the global minima in the asymmetric case. This helps to explain why asymmetric parametrization can exponentially speed up convergence \cite{xiong2024over}. 

In asymmetric matrix factorization, we show that the unbalanced initialization \cite{ward2023convergence} $(X_0,Y_0)=(MA,B)$, for almost every $A,B$, suffices to bring the convergence of gradient descent with constant step size back to linear. This has been achieved with high probability \cite[Theorem 4.2]{jiang2023algorithmic} with small random initialization under a nondegeneracy assumption on the singular values of $M$. As for the symmetric case, linear convergence of gradient descent is possible with a preconditioner \cite[Corollary 5]{zhang2021preconditioned} or adaptive step sizes \cite[Sections 7.1, 7.2]{davis2025gradient}.

More generally, we prove that overparametrized linear neural networks
$$f(W)=\|W_\ell\cdots W_1X-Y\|_F\p 2$$
have \KL exponent $1/2$ for almost every input matrix $X$ and almost every full row rank output matrix $Y$. Indeed, the inner mapping $W\mapsto W_\ell\cdots W_1X$ has constant rank near every global minimum, so the result follows from our composition rule (Rebjock and Boumal's rule already applies here). The \KL exponent 1/2 has been established in linear neural network regression \cite{marion2024deep} and regularized deep matrix factorization \cite{chen2025error}.

This paper is organized as follows. \cref{sec:Background} contains background material on variational analysis, subanalytic geometry, and differential geometry. \cref{sec:calculus-rules} proposes calculus rules for \KL exponents in the presence of a composite structure or symmetry. Finally, \cref{sec:Applications} deals with applications.

%we propose a new and shorter proof of the absence of second-order stationary points in matrix factorization. It 
%Lipschitz Examples include neural networks with smooth activation functions like GeLU (used in tranformers) and robust principal component analysis.

\section{Background}
\label{sec:Background}

We will borrow notions from variational analysis \cite{rockafellar2009variational}, subanalytic geometry \cite{shiota2012geometry}, and differential geometry \cite{lee2012smooth}. Let us start with some notations. Let $[m]:=\{1,\dots,m\}$. For a matrix $A\in \R^{m\times n}$ and an SVD $A=U \Sigma(A)V^\top$, the matrix $\Sigma(A)\in \R^{m\times n}$ is the diagonal matrix consisting of all the singular values of $A$. The $i$-th largest singular value of $A$ is denoted by $\sigma_i(A)$. For a symmetric matrix $A\in \bS^n$ and an eigenvalue decomposition $A=U\Lambda(A) U^\top$, we denote the diagonal matrix consisting of all the eigenvalues of $A$ by $\Lambda(A)\in\bS^n$. The $i$-th largest eigenvalue of $A$ is denoted by $\lambda_i(A)$. For matrix norms, $\|A\|_F$ denotes the Frobenius norm of $A$, $\|A\|_2$ denotes the spectral norm of $A$, and $\|A\|_1$ is the entry-wise $\ell_1$ norm of $A$, i.e., the sum of the absolute values of all the entries of $A$. For two matrices $A,B$ of the same dimension, the inner product $\langle A,B \rangle$ is defined as $\tr(A^\top B)$, where $\tr(C)$ is the sum of the diagonal elements of $C$. For any $I\subseteq [m]$ and $J\subseteq [n]$, let $A_{IJ}\in \R^{|I|\times |J|}$ denote the submatrix of $A$ obtained by retaining the rows indexed by $I$ and the columns indexed by $J$. The set of positive matrices is denoted by $\R^{m\times n}_{++}$. 

% and $I\subseteq [m]$ and $J\subseteq [n]$, we define $A_{IJ}\in \R^{|I|\times |J|}$ to be the submatrix of $A$, whose rows are in $I$ and columns are in $J$. 

% For a subset $S\subseteq \R^{m\times n}$ and $A\in \R^{m\times n}$, the distance from $A$ to $S$ is defined as $d(A,S)=\inf_{B\in S}\|A-B\|_F$. The projection of $x$ onto a set $\Omega$ is denoted by $P_\Omega(x)$. 
A map $F:A\to B$ is called open, where $A$ and $B$ are two topological spaces, if $F$ maps open sets in $A$ to open sets in $B$. The map $F$ is said to be an open map near $x\in \mathcal{A}$, if there exists a neighborhood $U$ of $x$, such that $F|_{U}$ is an open map. For a linear map $F:A \to B$, where $A$ and $B$ are Hilbert spaces, the adjoint operator $F^*:B\to A$ is defined as the unique linear map from $B$ to $A$ such that $\langle y, Fx \rangle=\langle F^*y, x \rangle$ for all $x\in A$ and $y\in B$.

For $x\in \R$, we denote its nonnegative part $\max\{x,0\}$ by $x_+$. For a linear subspace $L\subseteq \R^n$, the orthogonal complement of $L$ is denoted by $L^\perp$. 

\subsection{Variational analysis}
\label{subsec:variational-analysis}

Given $x \in \mathbb{R}^n$ and $S \subseteq \R^ n$, let 
    $$d(x,S) := \inf \{|x-y|: y \in S\} ~~~ \text{and} ~~~
    P_S(x) := \arg\min \{ |x-y| : y \in S\}.$$
For matrices, we use the Frobenius norm. Given $f:\R^ n \to \overline{\R}$ and $\ell\in \R$, let $$[f = \ell] := \{x\in \R^ n : f(x) = \ell\}$$ and define other expressions like $[f \leq \ell]$ similarly. Let $\dom f := \{ x\in \mathbb{R}^n : f(x) < \infty\}$, 
$$\gph f:=\{ (x,t)\in \mathbb{R}^n\times \R : f(x)=t\}, ~~~\text{and}~~~ \epi f:=\{ (x,t)\in \mathbb{R}^n\times \R : f(x)\leq t\}.$$
A function $f:\mathbb{R}^n\rightarrow\overline{\mathbb{R}}$ is lsc at $\overline{x} \in \dom f$ if $\liminf_{x\rightarrow \overline{x}} f(x) \geq   f(\overline{x})$ \cite[Definition 1.5]{rockafellar2009variational}. It is lsc if it is so at every point in its domain. This is equivalent to requiring that $\epi f$
is closed \cite[Theorem 1.6]{rockafellar2009variational}.

The regular normal cone and normal cone \cite[Definition 6.3]{rockafellar2009variational} are defined by
\begin{gather*}
    \widehat{N}_C (\overline{x}) := \{ v \in \mathbb{R}^n : \langle v , x - \overline{x}\rangle \leq  o(|x-\overline{x}|)~\text{for}~x\in C~\text{near}~\overline{x} \}, \\ 
    N_C(\overline{x}) := \{ v \in \mathbb{R}^n : \exists x_k \xrightarrow[C]{} \overline{x}~\text{and}~ \exists v_k \rightarrow v~\text{with}~v_k \in \widehat{N}_C (x_k) \},
\end{gather*}
where $x_k \xrightarrow[C]{} \overline{x}$ is a shorthand for $x_k \rightarrow \overline{x}$ and $x_k \in C$. Explicitly, the $o$ means that 
\begin{equation*}
    \limsup_{\scriptsize \begin{array}{c}x \xrightarrow[C]{} \overline{x}\\ x\neq \overline{x}\end{array}} \frac{\langle v , x - \overline{x}\rangle}{|x-\overline{x}|} \leqslant  0.
\end{equation*} 
A set $C \subseteq \mathbb{R}^n$ is regular \cite[Definition 6.4]{rockafellar2009variational} at one of its points $\overline{x}$ if it is locally closed\footnote{A subset $S$ of a topological space $X$ is locally closed if every point $p\in S$ admits a neighborhood such that $S\cap U$ is closed in $U$.} and $\widehat{N}_C(\overline{x}) = N_C(\overline{x})$.

Given $f:\mathbb{R}^n\rightarrow \overline{\mathbb{R}}$ and a point $\overline{x}\in\mathbb{R}^n$ where $f(\overline{x})$ is finite, the regular subdifferential, subdifferential, horizon subdifferential \cite[Definition 8.3]{rockafellar2009variational}, and Clarke subdifferential of $f$ at $\overline{x}$ \cite[Definition 4.1]{drusvyatskiy2015curves} are respectively given by
\begin{gather*}
    \widehat{\partial} f (\overline{x}) := \{ v \in \mathbb{R}^n : f(x) \geq   f(\overline{x}) + \langle v , x - \overline{x} \rangle + o(|x-\overline{x}|) ~\text{near}~ \overline{x} \}, \\
    \partial f(\overline{x}) := \{ v \in \mathbb{R}^n : \exists (x_k,v_k)\in \gph\hspace*{.5mm}\widehat{\partial} f: (x_k, f(x_k), v_k)\rightarrow(\overline{x}, f(\overline{x}), v) \}, \\[1mm]
    \partial^\infty f(\overline{x}) := \{ v \in \mathbb{R}^n : \exists (x_k,v_k)\in \gph\hspace*{.5mm}\widehat{\partial} f: \exists \tau_k \downarrow 0: (x_k, f(x_k), \tau_kv_k)\rightarrow(\overline{x}, f(\overline{x}), v) \}, \\[2mm]
    \cp f(\overline{x}) := \overline{\mathrm{co}} [\partial f(\overline{x}) + \partial^\infty f(\overline{x})],
\end{gather*}
where $\co$ denotes the convex hull, and $\overline\co$ its closure.
The $o$ means that
\begin{equation*}
    \liminf_{\scriptsize \begin{array}{c}x \rightarrow \overline{x} \\ x\neq \overline{x}\end{array}} \frac{f(x) - f(\overline{x}) - \langle v , x - \overline{x} \rangle}{|x-\overline{x}|} \geq   0.
\end{equation*} 
A point $x\in \R^ n$ is critical (resp. Clarke critical) if $0\in \partial f(x)$ (resp. $0\in \cp f(x)$).  

Following in the footsteps of Li and Pong \cite{li2018calculus}, we will use a change of coordinates to devise one of our calculus rules.

\begin{fact}
    \label{fact:change_coordinates}
    Let $f := g \circ F$ where $g:\R^ m \to \overline{\R}$ is lsc and $F:\R^ n \to \R^ m$ is $C^ 1$ near $\overline{x}\in\dom f = F^ {-1}(\dom g)$. Suppose $dF_{\overline{x}}$ is surjective. Then
    $$\partial f(\overline{x}) = (dF_{\overline{x}})^ * \partial g(F(\overline{x})),\quad \cp f(\overline{x}) = (dF_{\overline{x}})^ * \cp g(F(\overline{x})).$$
\end{fact}
\begin{proof}
    The chain rule for $\partial f$ follows from  \cite[Exercise 10.7]{rockafellar2009variational}, which also gives 
\[    \partial ^\infty f(\overline{x})=(dF_{\overline{x}})^*\partial^\infty g(F(\overline{x})).      \]
Utilizing the linearity of the operation of taking convex hull \cite[Chapter 1, Exercise 2.5]{brondsted2012introduction}, we have 
\[   \co[\partial f(\overline{x})+\partial^\infty f(\overline{x})]=(dF_{\overline{x}})^*\co[\partial g(F(\overline{x}))+\partial^\infty g(F(\overline{x}))].       \]
Then, by utilizing the closure criterion in \cite[Theorem 9.1]{rockafellar1970convex}, which implies that if a linear mapping $A$ is injective then $\overline{AC}=A\overline{C}$ for any convex sets, we have 
\begin{equation*}
    \cp f(\overline{x}) = \overline{\co}(\partial f(\overline{x})+\partial^\infty f(\overline{x}))=(dF_{\overline{x}})^*\cp g(F(\overline{x})). \qedhere     
\end{equation*}   
\end{proof}
% \begin{fact}
% \label{fact:hull}
%     If $L:V\to W$ is a linear map between vectors spaces $V,W$ and $X$ is a subset of $V$, then $\co L(X) = L(\co X)$.
% \end{fact}
The outer and inner second-order tangent set to $C$ at $\overline{x}$ in the direction $v$ are defined as follows \cite[Definition 3.28]{bonnans2013perturbation}:
\begin{align*}
     T^2_C(\overline{x},v):=\{\xi\in \R^n:~ \exists t_i\downarrow 0,~ d(\overline{x}+t_iv+\frac{t_i^2}{2}\xi  ,C)= o (t_i^2)\},\\
     T^{i,2}_C(\overline{x},v):=\{\xi \in \R^n:~d(\overline{x}+tv+\frac{t^2}{2}\xi  ,C)= o (t^2),~t\geq 0  \}.
\end{align*}
A set $C$ is said to be outer second-order regular at $\overline{x}$ if the next condition holds, 
\[ \forall v\in T_C(\overline{x}),~y_i=\overline{x}+t_iv+\frac{t_i^2}{2}w_i,~t_i\downarrow 0,~t_i|w_i|\to0,~~ \text{ it holds that } \lim_{i\to\infty}d(w_i,T_C^2(\overline{x},v))=0.  \]
A set $C$ is said to be second-order regular at $\overline{x}$, if it is outer second-order regular at $\overline{x}$ and $T^2_C(\overline{x},h)=T^{2,i}_C(\overline{x},h)$ for all $h\in T_C(\overline{x})$. A main category of sets that are second-order regular is the $C^2$-cone reducible sets \cite[Proposition 3.136]{bonnans2013perturbation}. A set $C\subseteq \R^n$ is said to be $C^2$-cone reducible at $\overline{x}$ \cite[Definition 3.135]{bonnans2013perturbation}, if there exist a neighborhood $U$ of $\overline{x}$, a $C^2$ mapping $G: U\to\R^m$ and a pointed closed convex cone $K\subseteq \R^m$, such that $DG(\overline{x})$ is surjective, $G(\overline{x})=0$ and $U\cap C=U\cap G^{-1}(K)$. Since $C^2$ embedded submanifold of $\R^n$ is $C^2$-cone reducible at any point in the manifold by using the defining equation \cite[Proposition 5.16]{lee2012smooth}, we have the following fact: 
\begin{fact}
    \label{defn:soregular}
    Any $C^2$ embedded submanifold $\cM$ of $\R^n$ is second-order regular at any $x\in \cM$.
\end{fact}
\subsection{Subanalytic geometry}
\label{subsec:subanalytic-geometry}
A subset $X$ of $\mathbb{R}^n$ is subanalytic \cite{shiota2012geometry} if for any $x\in \R^n$, there is a neighborhood $U$ of $x$ such that $X\cap U$ is of the form $\im(f_1)-\im(f_2)$, where $f_1,f_2$ are topologically proper maps from real analytic manifolds to $\R^n$. A subset $X\subseteq \R^n$ is said to be globally subanalytic if $G(X)$ is subanalytic \cite{van1996geometric}, where $G:\R^n\to [-1,1]^n$ is defined as 
   $$    G(x_1,\dots,x_n):=\left(\frac{x_1}{\sqrt{1+x_1^2}},\dots,\frac{x_n}{\sqrt{1+x_n^2}}\right).              $$
Let us note that globally subanalytic sets are always subanalytic, but not vice versa. However, bounded subanalytic sets are always globally subanalytic \cite{van1996geometric}. Globally subanalytic sets form an o-minimal structure on $\R$ \cite{van1986generalization}, while subanalytic sets is not an o-minimal structure \cite{shiota2012geometry}.

A function $f:\R^ n \to \overline\R$ is (globally) subanalytic if $\gph f$ is (globally) subanalytic. Globally subanalytic lsc functions satisfy two important inequalities.
    
    % If $f:\R^ n \to \overline\R$ is globally subanalytic and $\overline{x}\in \dom f$, and $f$ is lsc around $\overline{x}$, \cite{bai2022equivalence} then  
\begin{definition}
\label{def:weak}
    A function $f:\R^ n \to \overline\R$ has growth exponent $\beta\in[1,\infty)$ at $\overline{x}\in\dom f$ if there exist $\rho,\kappa>0$ such that
    $$\forall x\in B_\rho(\overline{x}),~~~(f(x)-f(\overline{x}))_+ \geq \kappa\hspace{.4mm}d(x,[f\leq f(\overline{x})])^ \beta.$$
\end{definition}
%    A function $f:\R^ n \to \overline\R$ has growth exponent $\beta\in[1,\infty)$ at $\overline{x}\in \dom f$ if $f-f(\overline{x})$ has growth exponent $\beta$ at $\overline{x}$.
    Globally subanalytic lsc functions admit a growth exponent at every point in their domain. This can be proved using the same arguments as in \cite[Theorem 1.14]{pham2016genericity}. 
\begin{definition}
\label{def:strong}
    A function $f:\R^ n \to \overline\R$ has \KL exponent $\alpha\in[0,1)$ at $\overline{x}\in \dom f$ if there exist $r,\ell,c>0$ such that
    $$ \forall x\in B_r(\overline{x})\cap[0<f-f(\overline{x})<\ell],~~~ d(0,\cp f(x)) \geq c(f(x)-f(\overline{x}))^\alpha.$$
\end{definition}
%A function $f:\R^ n \to \overline\R$ has \KL exponent $\alpha\in[0,1)$ at $\overline{x}\in \dom f$ if $f-f(\overline{x})$ has \KL exponent $\alpha$ at $\overline{x}$.
Globally subanalytic lsc functions admit a \KL exponent at every point in their domain \cite[Corollary 16]{bolte2007clarke}. In the literature \cite{li2018calculus,yu2022kurdyka,ouyang2025kurdyka}, the \KL exponent is often used for $\partial f$ rather than $\cp f$. Here, we use $\cp f$ because $\partial f\subseteq \cp f$, and hence \KL exponent for $\cp f$ is stronger than the \KL exponent for $\partial f$. Besides, in continuous-time dynamics, it is usually preferable to use a convex-valued subdifferential to ensure the existence of solutions. 

The \KL exponent $\alpha$ plays a crucial role in determining how fast algorithms converge:
\begin{itemize}
    \item $\alpha = 0$: finite (resp. linear) convergence for descent methods \cite[Theorem 2(i)]{attouch2009convergence} (resp. subgradient methods on weakly convex functions \cite{davis2018subgradient,davis2024stochastic,li2020nonconvex});
    \item $\alpha\in (0,1/2)$: finite convergence for descent methods \cite[Theorem 3]{bento2025convergence};
    \item $\alpha=1/2$: linear convergence
    for descent methods for iterates and function values \cite{polyak1963gradient,attouch2009convergence};
    \item $\alpha \in (1/2,1)$: sublinear convergence
    for descent methods at the rate $O(1/k^{(1-\alpha)/(2\alpha-1)})$ for the iterates \cite[Theorem 2(iii)]{attouch2009convergence} and $O(1/k^{1/(2\alpha-1)})$ for the function values \cite[Theorem 3.4(iii)(1)]{frankel2015splitting}.
\end{itemize}

    % In fact, under the above assumptions, one can strengthen the inequality by replacing $\partial f(x)$ with the superset $\cpf(x)$. In that case, we'll speak of \L{}ojasiewicz subgradient inequality and \KL exponent. Of course, if $f$ is Lipschitz continuous and regular near $\overline{x}$, then $\partial f(x)  =\cp f(x)$ near $\overline{x}$ by \cite[Theorem 9.61]{rockafellar2009variational}, so both inequalities coincide. The distinction is important when considering differential inclusions. Indeed, the local existence of trajectories is only guaranteed for set-valued mappings with convex values \cite[Theorem 3 p. 98]{aubin1984differential}. 

    The relationship between the growth and the \KL exponent is vital in this paper.  
    \begin{fact}
    \label{fact:conversion}
        Let $f:\R^ n \to \overline \R$ be lsc and $\overline{x} \in\R\p n$. If $f$ has \KL exponent $\alpha\in [0,1)$ at $\overline{x}$, then it has growth exponent $\beta =  1/(1-\alpha)$ at $\overline{x}$.
        %The \L{}ojasiewicz subgradient inequality with exponent $\alpha\in (-\infty,1)$ at $\overline{x}$ implies the \L{}ojasiewicz inequality with exponent $\beta =  1/(1-\alpha)$ at $\overline{x}$.
    \end{fact}
    \begin{proof}
       %Note that $\partial f(x) \subseteq \cp f(x)$ for all $x\in \R^ n$ with $f(x)<\infty$ and so $d(0,\partial f(x)) \geq d(0,\cp f(x))$. 
       Without loss of generality, $f(\overline{x})=0$.
        By the chain rule \cite[Lemma 2.4]{kruger2019holder}, the \KL inequality can be rewritten in a desingularized form $$\forall x\in B_r(\overline{x})\cap[0<f<\ell],~~~ d(0,\partial(\psi \circ f)(x))\geq 1,$$ with $\psi(t) := c^{-1}t^{1-\alpha}$. This was shown by Kurdyka in the smooth case \cite{kurdyka1998gradients}. By applying \cref{fact:conversion_0} below to $(\psi\circ f)_+=\psi\circ (f)_+$ where $\psi$ is monotone, there exists $\rho>0$ such that 
    $$\forall x\in B_\rho(\overline{x}),~~~ (\psi \circ f_+)(x) \geq d(x,[(\psi \circ f_+) = 0])/2.$$
    Since $\psi^{-1}(s) = c^ {1/(1-\alpha)} s^ {1/(1-\alpha)}$, this means that 
    \begin{equation*}
        \forall x\in B_\rho(\overline{x}),~~~ f(x)_+ \geq \psi^{-1}(d(x,\left[ f_+ = 0\right])/2) = (c/2)^ {1/(1-\alpha)} d(x,[f \leq 0])^ {1/(1-\alpha)}. \qedhere
    \end{equation*}
    \end{proof}
    \begin{fact}
    \label{fact:conversion_0}
        Let $f:\R^ n\to \overline{\R}_+$ be lsc and $\overline{x}\in[f=0]$. If $d(0,\partial f(x))\geq 1$ for all $(x,f(x))$ near $(\overline{x},f(\overline{x}))$, then $f(x) \geq d(x,[f=0])/2$ near $\overline{x}$.
    \end{fact}
    \cref{fact:conversion_0} is an application of the Ekeland variational principle in disguise (see \cite[Basic Lemma]{ioffe2000metric}, \cite[Lemma 2.5]{drusvyatskiy2015curves}, \cite[Lemma 3.1]{kruger2019holder}, and the proof in the Appendix).
    When $f:\R^ n \to \R_+$ is locally Lipschitz and semi-algebraic, another proof of the implication in \cref{fact:conversion} is possible using differential inclusions.
    Following arguments in \cite[Proposition 7]{josz2023global}, one readily obtains
%$$|x(0)-x(t)| = \left|\int_0^ t x'(s)ds\right| \leq \int_0^ t |x'(s)|ds \leq c^{-1} f(x(0))^ {1-\alpha}$$
$$d(x(0),[f = 0]) \leq |x(0)-x(\infty)| \leq \int_0^ \infty |x'(t)|dt \leq c^{-1} f(x(0))^ {1-\alpha}$$
where $x(\cdot)$ is a solution to $x'(t) \in - \cp f(x(t))$ for almost every $t>0$. 
A similar reasoning appears in \cite[Proposition 1]{otto2000generalization}, \cite[Appendix A]{karimi2016linear}, \cite[Theorem 5]{bolte2017error}, and \cite[Proposition 2.2]{rebjock2024fast}.

The following fact will be useful for the converse, that is, to convert a growth exponent into a \KL exponent.

\begin{fact}[{\cite[5.2]{van1996geometric}, \cite[Proposition 1.8.4]{valette2025subanalytic}}, Puiseux Lemma]\label{fact:puisuex}
    \label{fact_puisuex}
    Let $f:(0,\eta)\to \R$ be a globally subanalytic function and $\eta>0$. Then, there exist $\epsilon\in (0,\eta)$, $m\in \mathbb Z$ and $p\in \mathbb N_*$ such that for all $t\in (0,\epsilon)$ it holds that 
    \begin{equation}
        \label{frac_power_series}
          f(t)=\sum_{i=m}^\infty a_i t^{i/p},\quad ~a_i\in \R,~\forall i\geq m, 
    \end{equation}
    where this Puiseux series is convergent on $(0,\epsilon)$. 
\end{fact} 

\cref{fact_puisuex} is in general not true for subanalytic functions, e.g., consider $x\mapsto e^{1/x}$ for $x>0$, whose growth near $0$ is faster than any negative power of $x$, which means it is impossible to write $e^{1/x}$ as a fractional power series. However, it is true for subanalytic functions which are bounded near $0$, since bounded subanalytic functions are always globally subanalytic\cite[Section 3 and D.10]{van1996geometric}.

Utilizing the transformation $t\mapsto t^p$ and standard results for real analytic power series \cite{krantz2002primer}, we can easily prove the following fact:
\begin{fact}\label{fact:puiseux-derivative}
\label{fact_puiseux_diff}
    Let $f:(0,\epsilon)$ have the Puiseux expansion in \cref{frac_power_series}, which is also convergent on $(0,\epsilon)$. Then $f\in C^{\infty}(0,\epsilon)$, and for all $t\in (0,\epsilon)$. it holds that 
    \[   f'(t)=\sum_{i=m}^\infty \frac{ia_i}{p}t^{\frac{i}{p}-1},\quad \forall t\in (0,\epsilon).      \]
\end{fact}

\subsection{Differential geometry}
\label{subsec:differential-geometry}
Our starting point is a smooth manifold $\cM$, that is, a topological manifold equipped with a smooth structure. (By smooth, we mean $C\p \infty$, unless we specify $C^k$ smooth for some $k\in \{1,2,\hdots,\infty\}$). A topological manifold is a locally Euclidean (of constant dimension) second-countable Hausdorff topological space. In contrast to the branch of optimization dealing with optimization on smooth manifolds \cite{boumal2023introduction}, our variable will lie in a Euclidean space as usual.

The smooth structure enables one to define smooth maps between two manifolds $\cM,\cN$ as well as the tangent space $T_p\cM$ at a point $p \in \cM$. Tangent vectors $v\in T_p\cM$ are linear maps $v:C^\infty(\cM)\to \R$ such that $v(fg)=v(f)g+fv(g)$ for $f,g\in C^\infty(\cM)$. Tangent vectors can also be defined using an equivalence relation on the set of all smooth curves $\gamma:J\rightarrow \cM$ where $J$ is an interval of $\mathbb{R}$ containing $0$ and $\gamma(0) = p$ \cite[p. 71]{lee2012smooth}. Two such curves $\gamma_1:J_1\rightarrow \cM$ and $\gamma_2:J_2\rightarrow \cM$ are equivalent if $(f\circ \gamma_1)'(0) = (f\circ \gamma_2)'(0)$ for any smooth real-valued function defined in a neighborhood of $p$. The tangent space is then the set of equivalence classes.  

The differential of a smooth map $F:\cM\rightarrow \cN$ at $p\in \cM$ is the linear map $dF_p: T_p\cM \rightarrow T_{F(p)} \cN$ defined by $dF_p(v)(f)=v(f\circ F)$ for all $v\in T_p\cM$ and $f\in C^\infty(\cN)$. The rank of $F$ at $p$ is the rank of $dF_p$, namely the dimension of the image of $dF_p$. A map $F:\cM\rightarrow \cN$ between two smooth manifolds $\cM,\cN$ is a smooth immersion (respectively submersion) if it is smooth and $dF_p$ is injective (respectively surjective) for all $p\in\cM$. It is a smooth embedding if it a smooth immersion and a topological embedding, i.e., a homeomorphism onto its image $F(\cM) \subseteq \cN$ in the subspace topology.  

The differential enables one to define notions of submanifolds, which arise prominently in the study of symmetries. An embedded submanifold of $\cM$ is a subset $S \subseteq \cM$ that is a manifold in the subspace topology, endowed with a smooth structure with respect to which the inclusion map $S \hookrightarrow \cM$ is a smooth embedding (the inclusion map is defined by $S \ni x\mapsto x\in \cM$). An immersed submanifold is a subset $S \subseteq \cM$ endowed with a topology (not necessarily the subspace topology) with respect to which it is a topological manifold, and a smooth structure with respect to which the inclusion map $S\hookrightarrow \cM$ is a smooth immersion. From the definition, one sees that embedded submanifolds are immersed manifolds, but the converse is false, as illustrated by the figure eight \cite[Example 4.19]{lee2012smooth}. Embedded submanifolds can be expressed locally as level sets of smooth submersions \cite[Proposition 5.16]{lee2012smooth}, which is how they are often defined in $\mathbb{R}^n$ \cite[Example 6.8]{rockafellar2009variational}.

Suppose $\cS$ is an immersed submanifold of a smooth manifold $\cM$. Since inclusion map $\iota:\cS\to\cM$ is a smooth immersion, its differential $d\iota_p:T_p \cS\to T_p\cM$ is injective for all $p\in \cS$. One may thus view $T_p\cS$ as a subspace of $T_p\cM$ via the identification $T_p\cS\cong d\iota_p(T_p\cS)$. The following characterization is helpful \cite[Proposition 5.35]{lee2012smooth}. A vector $v\in T_p\cM$ is in $T_p\cS$ if and only if there is a smooth curve $\gamma:J\to \cM$ whose image is contained in $\cS$, and which is also a smooth as a map into $\cS$, such that $0\in J$, $\gamma(0)=p$, and $\gamma'(0)=v$. When $\cM=\R^n$, $T_p\cM\cong \R^n$. Hence, $T_p\cS$ can simply be viewed as a subset of $\R^n$. In that case, one can define the normal space $N_p\cS:=(T_p\cS)^\perp$. When $\cS$ is an embedded submanifold of $\R^n$, then the tangent space and the normal space agree with the tangent cone and the normal cones from variational analysis, respectively, namely $T_\cS(p)=T_p\cS$ and $\widehat{N}_\cS (p) = N_\cS (p) = N_{p} \cS$ \cite[Example 6.8]{rockafellar2009variational}.

An action of a group $G$ with identity $e$ on a set $\cM$ \cite[p. 161]{lee2012smooth} is a map $\theta: G \times \cM \rightarrow \cM$ such that 
\begin{enumerate}[label=\rm{(\roman{*})}]
    \item $\forall g,h\in G, ~ \forall x\in \cM, ~\theta(gh,x) = \theta(g,\theta(h,x))$,
    \item $\forall x\in \cM, ~ \theta(e,x) = x$.
\end{enumerate}
When such a map $\theta$ exists, $G$ is said to act on $\cM$ with the action $\theta$, and $\cM$ is referred to as a $G$-space. A Lie group $G$ is a smooth manifold and a group whose multiplication and inversion operations are smooth. A Lie group $G$ acts smoothly on a smooth manifold $\cM$ if there exists a smooth action $\theta: G \times \cM \rightarrow \cM$. To simplify the notation, when the action is clear from the context, we will denote $\theta(g,x)$ by $g\cdot x$.

A function $f:\cM\rightarrow \cN$ between sets $\cM$ and $\cN$ is invariant under an action of a Lie group $G$ on $\cM$, or simply $G$-invariant, if 
$$\forall (g,x)\in G\times \cM, ~~~  f(g\cdot x)=f(x).$$
% Suppose $G$ also acts smoothly on $N$. We say that $f$ is equivariant if
% % $$\forall (g,x)\in G\times M, ~~~  f(gx)=gf(x).$$
Suppose $G$ acts on $\cM$, The orbit of a point $x\in \cM$ is the set $Gx:=\{ g\cdot x : g\in G\}$. We need the following fact concerning the orbit: 
\begin{fact}[{\cite[Corollary 2.21]{kirillov2008introduction}}]
    \label{fact_orbit}
    If $G$ acts smoothly on $\cM$ and $x\in \cM$, then $Gx$ is an immersed submanifold of $\cM$ and $\im d(\theta\p x)_e=T_xGx$.
\end{fact}
While \cref{fact_orbit} shows that the orbit of an arbitrary Lie group is always an immersed submanifold, we will need the orbit to be embedded for our purpose in this paper. It is known that the orbit is embedded when the action is proper \cite[Proposition 21.7]{lee2012smooth}, but the action considered in this paper is usually improper. Fortunately, the action considered in this paper is always definable, which is defined as follows. Assume that $G\subseteq \R^{q}$, $\cM\subseteq \R^n$, and the graph of the action $\theta:~G\times \cM\to \cM$ has definable graph, then the action is said to be definable. The next fact is proved in \cite[Appendix (B4)]{gibson1979singular}.
\begin{fact}
    \label{fact_embeddedorbit}
    Let $\cM\subseteq \R^n$ be an embedded submanifold of $\R^n$. Assume the action $\theta:G\times \cM\to \cM$ is smooth and semi-algebraic. Then every orbit of $G$ is an embedded submanifold of $\cM$.
\end{fact}
\begin{fact}[{\cite[(3.6) and Theorem 3.8]{dudek1994nonlinear}}]
\label{fact_proj}
Let $\cM$ be a $C^2$ embedded submanifold of $\R^n$ and $\overline x\in \cM$. There exists a neighborhood $U$ of $\overline x$ in $\R\p n$ such that 
$$\forall x\in U\cap \cM,~\forall y\in U~,~~~ y-x\in \cN_x\cM ~~~\implies ~~~ P_\cM(y)=x.$$
\end{fact}
A Lie subgroup of a Lie group $G$ is a subgroup of $G$ endowed with a topology and smooth structure making into a Lie group and an immersed submanifold. Topologically closed subgroups of Lie groups are Lie subgroups by the closed subgroup theorem \cite[Theorem 20.12]{lee2012smooth}. Let $\fg$ denote the Lie algebra of a Lie group $G$, which we identify with its tangent space at $e$. The notation $\theta_g:\cM\to \cM$ and $\theta^x: G\to \cM$ denote the partial action of $\theta$ when one of the parameters is fixed. 

% We introduce some standard vocabulary. 

%A map between topological spaces is proper if preimages of compact sets are compact.
% Suppose $G$ acts smoothly on $M$.

% A smooth slice is an embedded submanifold $A \subseteq M$ such that $G(A|A)A=A$ and the restriction of the action $G\times A \to M$ is a smooth submersion. A smooth slice at a point $x \in M$ is a smooth slice such that $x\in A$ and $G(A|A) = G_x$. A smooth normal slice is a smooth slice $A$ such that $G(A|A) = G_y$ for all $y \in A$. If $A$ is a smooth normal slice at $x$, then passing down to the quotient $G/G_x \times A \to GA$ yields a diffeomorphism (see \cite[Section 2.3.3]{josz2025subdifferentiation} for a proof).

    Let $I_n$ denote the identity matrix of order $n$. The set of invertible matrices with real coefficients of order $n$, denoted  $\GL(n)$, is a Lie group. The natural action of a Lie subgroup $G$ of $\GL(n)$ on $\R^ n$ is defined by the matrix vector multiplication $G\times \R^ n \ni (g,x)\mapsto g\cdot x\in \R^ n$. The orthogonal group $$\Orth(n) := \{Q\in\R^{n\times n}: Q^ T Q = I_n\}$$ is a Lie subgroup of $\GL(n)$. 
    % \cite[Section 7.3]{boumal2023introduction}.
    % $$\St(n,k) = \{ A \in \R^{n\times k} : A^ T A = I_k \}$$
    % \cite[Section 7.5]{boumal2023introduction}
    The fixed-rank matrices form an embedded submanifold of $\R^{m\times n}$ \cite[Section 7.5]{boumal2023introduction}:
    $$\R_r^{m\times n} := \{ A \in \R^{m\times n} : \rk A = r\}.$$
    If $r = \min\{m,n\}$, then we let $\R_*^{m\times n} := \R_r^{m\times n}$, i.e., the set of full-rank matrices of size $m\times n$. It is an open subset of $\R^{m\times n}$, \cite[Proposition 2.1]{vandereycken2009embedded}. We will also consider bounded rank matrices:
    \[        \R_{\leq r}^{m\times n} := \{ A \in \R^{m\times n} : \rk A \leq r\}.                                        \]
    Let
    \begin{align*}
        \Sym\p n&:=\{A\in\R^{n\times n}:A^\top = A\},& &\Sym\p n_+:=\{A\in\Sym^n:A\succeq 0\}, \\
        \Sym\p n_r&:=\{A\in\Sym^{n}:~\rk(A) =r\},& &\Sym^n_{+,r}:=\{A\in\Sym\p n:A\succeq 0,~\rk(A)=r\}, \\
        \Sym\p n_{\leq r}&:=\{A\in\Sym\p {n}:~\rk(A)\leq r\},&&\Sym\p n_{+,\leq r}:=\{A\in\Sym\p n:A\succeq 0,~\rk(A)\leq r\}.
    \end{align*}
    % $$\Sym\p n:=\{A\in\R\p {n\times n}:A\p T = A\} ~~~\text{and}~~~ \Sym\p n_+:=\{A\in\Sym\p n:A\succeq 0\}.$$
    % Moreover, let us introduce the manifold of symmetric positive semidefinite fixed-rank matrices:
    % $$\Sym\p n:=\{A\in\R\p {n\times n}:A\p T = A\} ~~~\text{and}~~~ \Sym\p n_+:=\{A\in\Sym\p n:A\succeq 0\}.  $$
    % The notations $\bS^n_r$ and $\bS^n_{\leq r}$ are defined in a similar way without requiring positive semidefinite
The manifold $\Mr$ plays an important role in our paper. It is well-known that $\Mr$ is a smooth embedded submanifold of $\R^{m\times n}$ \cite[Proposition 2.1]{vandereycken2013low}. For $X\in \Mr$, and given an SVD of $X$:
\[   X=U\begin{bmatrix}
    \Alpha & 0 \\
    0 & 0
\end{bmatrix} V^\top,       \]
where $\Alpha \in \R^{r\times r}$ is a positive diagonal matrix. Then, $T_{\Mr}(X)=T_X\Mr$ is given by:
\begin{equation}
    \label{Mrtangent}
      T_{\Mr}(X)=U\begin{bmatrix}
    \R^{r\times r} & \R^{r\times (n-r)}  \\
    \R^{(m-r)\times r} &  0
\end{bmatrix} V^\top.   
\end{equation}
For $H\in T_{\Mr}(X)$, its second-order tangent set $T^2_{\Mr}(X,H)$ is given by \cite[(3.16)]{yang2025variational}:
\begin{equation}
    \label{Mrsotangent}
    T^2_{\Mr}(X,H)= \left\{HV\begin{bmatrix}
    \Alpha^{-1} & 0 \\
    0 & 0
\end{bmatrix}U^\top H\right\}+T_{\R^{m\times n}_r}(X). 
\end{equation}

\section{Calculus rules}
\label{sec:calculus-rules}

We propose two calculus rules in this section.  

\subsection{Composition rule}
\label{subsec:Composition rule}
%First, we state Li and Pong's calculus rule, mentioned in the introduction. It is founded on the chain rules in \cref{fact:change_coordinates}. Moreover, since the chain rule also holds for the Clarke subdifferential, we can deduce the chain rule for the \KL exponent as well.

Our first calculus rule harnesses a composite structure of the objective function for which it helpful to recall a fundamental result in variational analysis. 
%Roughly speaking, we want to generalize the results in \cref{fact:rule} to the case where the inner mapping $F$ is a submersion to an embedded submanifold of $\R^m$ instead of $\R^m$ itself. To achieve this goal, let us start with some definitions that are needed in the proofs of \cref{fact:rule}.

\begin{definition}[{\cite[Section 3E]{dontchev2009implicit}}]
    \label{def:metric_reg}
    A set-valued mapping $F:\R\p n\rightrightarrows \R\p n$ is metrically regular at $\overline{x}$ for $\overline{y}\in F(\overline{x})$ if there exist $\kappa>0$ along with neighborhoods $U$ of $\overline{x}$ and $V$ of $\overline{y}$ such that
    $$\forall x\in U,~\forall y\in V, ~~~ d(x,F\p{-1}(y))\leq \kappa \hspace{.3mm}d(y,F(x)).$$
\end{definition}

% \begin{definition}[{\cite[Section 3H]{dontchev2009implicit}}]
%     \label{def:metric_subreg}
%     A set-valued mapping $F:\R\p n\rightrightarrows \R\p n$ is metrically subregular at $\overline{x}$ for $\overline{y}\in F(\overline{x})$ if there exist $\kappa>0$ along with neighborhoods $U$ of $\overline{x}$ and $V$ of $\overline{y}$ such that
%     $$\forall x\in U, ~~~ d(x,F\p{-1}(\overline{y}))\leq \kappa \hspace{.3mm}d(\overline{y},F(x)\cap V).$$
% \end{definition}

\begin{theorem}{\normalfont (Lyusternik-Graves theorem \cite[Theorem 5.1]{dontchev2021lectures})\textbf{.}}
    \label{thm:LG}
    If $F:\R^n\to\R^m$ is $C\p 1$ near $\overline{x}\in\R\p n$ and a submersion at $\overline{x}$, then $F$ is metrically regular at $\overline{x}$ for $F(\overline{x})$.
    %there exist a neighborhood $U$ of $(\overline{x},F(\overline{x}))$ in $\R\p n\times \R\p m$ and $\kappa>0$ such that 
    %\[\forall (x,y)\in U,~~~ d(x,F^{-1}(y))\leq \kappa|F(x)-y|.       \]
\end{theorem}

The Lyusternik-Graves theorem allows one to generalize Li and Pong's calculus rule \cite[Theorem 3.2]{li2018calculus} for the \KL exponent mentioned in the introduction to the growth exponent. 

\begin{lemma}
    \label{lemma:submersion}
        Let $f := g \circ F$ where $F:\R^ n \to \R^ m$ is a $C^1$ submersion near $\overline{x}\in\dom f$ and $g:\R^ m \to \overline{\R}$ is lsc near $F(\overline{x})$. If $g$ has growth (resp. \KL) exponent $\beta$ at $F(\overline{x})$, then $f$ has growth (resp. \KL) exponent $\beta$ at $\overline{x}$. 
\end{lemma}

For pedagogical purposes, we now present a special case of \cref{lemma:submersion}. A function is positive definite if the origin is a strict global minimum.

\begin{proposition}
    Let $\mathcal{L}:\R\p m\to \R$ be lsc and positive semidefinite with \KL exponent $1/2$ at the origin. Let $F:\R\p n\to \R\p m$ be $C\p {\min\{n-m+1,1\}}$ such that $\im F$ has positive measure in $\R\p m$. For almost every $y \in \im F$, the function $f:\R\p n\to \R$ defined by
    $f(x) := \mathcal{L}(F(x)-y)$
    has \KL exponent $1/2$ at any global minimum.
\end{proposition}
\begin{proof}
   By Sard's theorem \cite{sard1942measure} (see also \cite[Theorem 6.10]{lee2012smooth}), the set of critical values of $F$ has measure zero in $\R\p m$, i.e., $N=\{ F(x): x\in \R\p n, \rk(dF_x)<m\}$ is a null set of $\R\p m$. Since $\im F$ has positive measure, $N$ is also a null set of $\im F$. Let $y\in \im F \setminus N$ and $x\in \R\p n$ be such that $f(x)=\min f =0$. Since $\mathcal{L}$ is positive definite, $F(x)=y$. Thus $F$ is a submersion at $x$. As $g:=\mathcal{L}(\cdot-y)$ has \KL exponent 1/2 at $F(x)$, 
    %by \cref{lemma:pham} it has \KL exponent $1/2$ at $F(x)$. Hence 
    $f$ has \KL exponent $1/2$ at $x$ by \cref{lemma:submersion}. 
\end{proof}

Our composition rule builds on \cref{lemma:submersion} to relax the submersion requirement of the inner mapping. We instead only need it to be of constant rank near the point of interest.
The main technical idea is to use the rank theorem \cite[Theorem 4.12]{lee2012smooth} to reduce the inner map to a canonical form so that \cref{lemma:submersion} is applicable by restricting the codomain of the inner map. To handle the Clarke subdifferential, we need the following auxiliary result concerning separable functions. There is a similar result in \cite[Proposition 10.5]{rockafellar2009variational} under different assumptions.
\begin{lemma}
    \label{lemma_clarke_separable}
    Let $f(x):=f_1(x_1)+\delta_0(x_2)$ for all $x\in \R\p n$ where $f_1:\R\p r\to \eR$, $x=(x_1,x_2)\in \R^n$, $x_1\in \R^r$, and $x_2\in \R^{n-r}$. Then $$\forall x_1 \in \R\p r,~~~ \cp f(x_1,0)=\cp f_1(x_1)\times \R^{n-r}.$$ 
\end{lemma}
\begin{proof}
 Let $\overline x\in \dom f$, which gives $\overline{x}_2=0$. By the definition of $\widehat\partial f$, we see that $v\in \widehat\partial f(\overline x)$ iff 
    \[   f(x)\geq f(\overline x)+\langle v,x- \overline x \rangle+o(|x-\overline x|).  \]
    Since $f(x)=\infty$ if $x_2\neq 0$, this inequality is equivalent to 
    \[  f_1(x_1)\geq f_1(\overline{x}_1)+\langle v_1,x_1-\overline{x}_1 \rangle+o(|x_1-\overline{x}_1|),     \]
    which is the definition of $v_1\in \widehat\partial f_1(\overline x_1)$. Thus for all $x\in \dom f$, $\widehat\partial f(x)=\widehat\partial f_1(x_1)\times \R^{n-r}$. Using the definition of $\partial f$ and $\partial^\infty f$, 
    \[  \forall x\in \dom f,\quad \partial f(x)=\partial f_1(x_1)\times \R^{n-r},~\partial^\infty f(x)=\partial^\infty f_1(x_1)\times \R^{n-r},    \]
    which implies $\cp f(x)=\cp f_1(x_1)\times \R^{n-r}$ by the definition $\cp f(x)=\overline{\mathrm{co}}[\partial f(x)+\partial^\infty f(x)]$.
\end{proof}

We are now ready to state our composition rule.

\begin{theorem}
\label{thm:composition}
        Let $f := g \circ F$ where $F:\R^ n \to \R^ m$ is $C^1$ with constant rank near $\overline{x}\in\dom f$ and $g:\R^ m \to \overline{\R}$ is lsc near $F(\overline{x})$. For all sufficiently small neighborhood $U$ of $\overline{x}$, if $g+\delta_{F(U)}$ has growth (resp. \KL) exponent $\beta$ at $F(\overline{x})$, then $f$ has growth (resp. \KL) exponent $\beta$ at $\overline{x}$.
        % There exists a neighborhood $U$ of $\overline{x}$ such that for all neighborhood $\widetilde U$ of $\overline{x}$ with $\widetilde U\subseteq U$ the next implication holds 
        % \begin{align*}
        %       &g+\delta_{F(\widetilde U)} \text{ has growth (resp. \L{}ojasiewicz) exponent } \beta \text{ at }F(\overline{x}) \\
        %      \implies &   f \text{ has growth (resp. \L{}ojasiewicz) exponent $\beta$ at $\overline{x}$}.
        % \end{align*}
        % be the restriction of $F$ on a sufficiently small neighborhood of $\overline{x}$. If $g+\delta_{\im \widetilde F}$ has growth (resp. \L{}ojasiewicz) exponent $\beta$ at $F(\overline{x})$, then $f$ has growth (resp. \L{}ojasiewicz) exponent $\beta$ at $\overline{x}$. 
\end{theorem}
\begin{proof}
Without loss of generality, $\overline{y}:=  F(\overline{x})$ is a local minimum of $g$ with $g(\overline y)=0$ (otherwise, replace $g$ by $(g-g(\overline{x}))_+$). Let $r$ denote the rank of $F$ near $\overline{x}$. By the rank theorem \cite[Theorem 4.12]{lee2012smooth}, there exist smooth charts $(U,\varphi)$ for $\R\p n$ centered at $\overline{x}$ and $(V,\psi)$ for $\R\p m$ centered at $F(\overline{x})$ such that $F(U)\subseteq V$ and
$$\forall x\in \varphi\p{-1}(U),~~~ \widehat F(x) := (\psi\circ F\circ \varphi\p{-1})(x_1,x_2) = (x_1,0),$$
where $x=(x_1,x_2)\in \R\p n$, $x_1\in\R\p r$, and $x_2\in \R\p{n-r}$. Observe that
$$f=g\circ F = g\circ\psi\p{-1}\circ \widehat F\circ \varphi.$$
With $\widehat f:= f\circ\varphi\p{-1}$ and $\widehat g := g\circ\psi\p{-1}$, we thus have $\widehat f =\widehat g\circ \widehat F$.
Since $\varphi$ and $\psi$ are diffeomorphisms, they leave the growth (resp. \L{}ojasiewicz) exponents unchanged by \cref{lemma:submersion}. We have thus reduced the problem to showing that $\widehat f$ has growth (resp. \L{}ojasiewicz) exponent $\beta$ at $\varphi(\overline{x})$. In other words, it suffices to show that $f=g\circ F$ has growth (resp. \L{}ojasiewicz) exponent $\beta$ at $\overline{x}$ when $F$ has the canonical form
$$\forall x\in U,~~~ F(x_1,x_2) = (x_1,0).$$
Let $\pi_1:\R\p m\to \R\p r$ denote the canonical projection onto the first $r$ coordinates. Then $$f = g\circ F = g_1\circ \pi_1 \circ F$$ where $g_1(y_1):=g(y_1,0)$ for all $y_1\in\R\p r$. Observe that $$g(y)+\delta_{F(U)} = g_1(y_1)+\delta_0(y_2)$$ 
for all $y=(y_1,y_2)\in \R\p m$ near $\overline{y}$ where $y_1\in\R\p r$, $y_2\in \R\p {m-r}$. 

Suppose $g+\delta_{F(U)}$ has growth exponent $\beta$ at $\overline{y}$. There are $\rho,\kappa>0$ such that we successively have
\begin{gather*}
    \forall y \in B_\rho(\overline{y}), ~~~ g(y)+\delta_{F(U)}(y) \geq \kappa\hspace{.3mm} d(y,[g+\delta_{F(U)}=0])\p \beta, \\
    \forall y \in B_\rho(\overline{y}), ~~~ g_1(y_1)+\delta_0(y_2) \geq \kappa\hspace{.3mm} d(y,[g_1+\delta_0=0])\p \beta, \\
    \forall y_1 \in B_\rho(\overline{y}_1), ~~~ g_1(y_1) \geq \kappa\hspace{.3mm} d(y_1,[g_1=0])\p \beta.
\end{gather*}
Thus $g_1$ has growth exponent $\beta$ at $\overline{y}_1$. Since $\pi_1\circ F$ is a submersion, $f$ has growth exponent $\beta$ by \cref{lemma:submersion}.

Suppose $\widetilde g :=g+\delta_{F(U)}$ has \KL exponent $\beta$ at $\overline{y}$. There are $\rho,\kappa>0$ such that
$$\forall y\in B_\rho(\overline{y}), ~~~ d(0,\overline{\partial}\widetilde g(y)) \geq \kappa \widetilde g(y)\p \beta.$$
By \cref{lemma_clarke_separable}, $\overline{\partial}\widetilde g(y) = \overline{\partial}g(y_1)\times \R\p{m-r}$ for all $y\in F(U)$. Thus $d(0,\overline{\partial}\widetilde g(y)) = d(0,\overline{\partial}g_1(y_1))$ and
$$\forall y_1\in B_\rho(\overline{y}_1), ~~~ d(0,\overline{\partial}g_1(y_1)) \geq \kappa g_1(y_1)\p \beta.$$
Thus $g_1$ has \KL exponent $\beta$ at $\overline{y}_1$. Since $\pi_1\circ F$ is a submersion, $f$ has \KL exponent $\beta$ by \cref{lemma:submersion}.
\end{proof}

Due to the requirement that the neighborhood $U$ be sufficiently small in \cref{thm:composition}, it is somewhat unwieldy in practice. But it admits corollaries which can readily be applied.

\begin{corollary}
\label{cor:composition}
     Let $f := g \circ F$ where $F:\R^ n \to \R^ m$ is $C^1$ near $\overline{x}\in\dom f$ and $g:\R^ m \to \overline{\R}$ is lsc near $F(\overline{x})$. Let $U\subseteq \R\p n$ be a neighborhood of $\overline{x}$ and $\cM$ be an embedded submanifold of $\R^m$ such that $F(U)\subseteq \cM$ and the restriction $\widetilde F:U\to \cM$ is a submersion at $\overline{x}$. If $g+\delta_{\cM}$ has growth (resp. \L{}ojasiewicz) exponent $\beta$ at $F(\overline{x})$, then $f$ has growth (resp. \L{}ojasiewicz) exponent $\beta$ at $\overline{x}$.
\end{corollary}
\begin{proof}
     The condition implies that $F$ is of constant rank near $\overline{x}$. Due to the submersion assumption, the set $F( U)$ agrees with $\cM$ around $F(\overline{x})$ for all sufficiently small neighborhood $U$ of $\overline{x}$. Therefore, if $g+\delta_{\cM}$ has growth (resp. \L{}ojasiewicz) exponent $\beta$ at $F(\overline{x})$, then so does $g+\delta_{F(U)}$ for all sufficiently small neighborhood $U$ of $\overline{x}$. Then the result follows from \cref{thm:composition}.
\end{proof}

When strict optimality holds in the outer function, the rule takes a particularly simple form. Before we state it, recall that isolated local minima of continuous semi-algebraic functions with growth exponent $\beta$ have \KL exponent $\alpha=1-1/\beta$ \cite[Theorem 4.2]{pham2020local}. A more general version of this fact is given below, whose proof is deferred to the Appendix.

\begin{lemma}\label{lemma:pham}
Let $f:\R^n\to \overline{\R}$ be lsc, subanalytic, and $\overline{x}\in \dom f$ be a strict local minimum of $f$. If $f$ has growth exponent $\beta$ at $\overline{x}$, then $f$ has \KL exponent $1-1/\beta$ at $\overline{x}$.
\end{lemma}

We can now deduce another corollary.

\begin{corollary}
\label{coro_composite_strict}
     Let $f := g \circ F$ where $F:\R^ n \to \R^ m$ is $C^1$ with constant rank near $\overline{x}\in \dom f$, $g:\R^ m \to \overline{\R}$ is lsc near $ F(\overline{x})$, and $g$ and $F$ are globally subanalytic. If $g$ has growth exponent $\beta$ at $F(\overline{x})$ and $F(\overline{x})$ is a strict local minimum of $g$, then $f$ has \KL exponent $1-1/\beta$ at $\overline{x}$. 
\end{corollary}
\begin{proof}
    Using the rank theorem \cite[Theorem 4.12]{lee2012smooth}, for sufficiently small neighborhood $U$ of $\overline{x}$, $F(U)$ is a $C^1$ embedded submanifold of $\R^m$. In particular, we can take $U$ to be a sufficiently small and globally subanalytic neighborhood of $\overline{x}$, in which case $F(U)$ is also globally subanalytic. Thus $\widetilde g:=g+\delta_{F(U)}$ is globally subanalytic and lsc near $F(\overline{x})$. Clearly, $\widetilde g$ has growth exponent $\beta$ at $F(\overline{x})$ since $g$ has growth exponent $\beta$ at the strict local minimum $F(\overline{x})$. \cref{lemma:pham} implies that $\widetilde g$ has \KL~exponent $1-1/\beta$ at $F(\overline{x})$. Then the result follows from \cref{thm:composition}.
\end{proof}
\subsection{Symmetry rule}
\label{subsec:Symmetry rule}

The symmetry rule relies on two simple lemmas. The first is a standard fact in differential geometry. A supplement $L\subseteq \R\p n$ of a linear subspace $V\subseteq \R\p n$ is a linear subspace such that $V+L=\R\p n$. Given a linear subspace $L\subseteq \R\p n$ and $\overline{x}\in \R\p n$, let $\vec L_{\overline{x}} := L+\{\overline{x}\}$ be the shifted affine passing through $\overline{x}$. For the normal space of a submanifold $\cM \subseteq \R\p n$, it will be convenient to write 
$\vec{N}_{\overline{x}} \cM := \overrightarrow{N_{\overline{x}}\cM}_{\overline{x}}=N_{\overline{x}}\cM+\{\overline{x}\}.$
\begin{lemma}\label{lemma:group_submersion}
Let $\theta:G\times\R\p n\to\R\p n$ be a smooth action and $\overline{x}\in\R\p n$. If $L$ is a supplement of $T_{\overline{x}}G\overline{x}$,
then $\theta|_{G\times \vec L_{\overline{x}}}$ is a submersion at $(e,\overline{x})$.
\end{lemma}
\begin{proof}
Let $\widetilde \theta :=\theta|_{G\times \vec L_{\overline{x}}}$. From
\[ \forall (v,w)\in \fg\times L,~~~
d\widetilde \theta_{(e,\overline{x})}(v,w)= d(\widetilde \theta\p {\overline{x}})_e(v) + d(\widetilde \theta_e)_{\overline{x}}(w) = d(\theta^{\overline{x}})_e(v) + w,
\]
it follows that $\im d\widetilde \theta_{(e,\overline{x})} = \im d(\theta^{\overline{x}})_e + L= T_{\overline{x}}G\overline{x} + L=\R\p n$ by \cref{fact_orbit}.
\end{proof}

The second is a chain rule for invariant functions.

\begin{lemma}
\label{lemma:G_chain}
Let $f:\R\p n\to\eR$ be lsc and invariant under a smooth action $\theta:G\times \R\p n\to\R\p n$. For all $g\in G$ and $x,y\in\dom f$, if $x=\theta(g,y)$, then
    \begin{subequations}
\begin{align*}
    \widehat\partial f(x) &= d(\theta_{g\p{-1}})_x\p * \widehat\partial f(y) = d(\theta_g)_y\p{-*}\widehat\partial f(y), \\
    \partial f(x) &= d(\theta_{g\p{-1}})_x\p * \partial f(y) = d(\theta_g)_y\p{-*}\partial f(y), \\
    \partial\p \infty f(x) & = d(\theta_{g\p{-1}})_x\p * \partial\p \infty f(y) = d(\theta_g)_y\p{-*}\partial\p \infty f(y), \\
    \overline{\partial} f(x) & = \overline{\co}[\partial f(x)+\partial\p\infty f(x)] = d(\theta_g)_y\p{-*}\overline{\partial} f(y).
\end{align*}
\end{subequations}
\end{lemma}
\begin{proof}
    Since $\theta_g\circ \theta_{g\p{-1}}=\mathrm{Id}_{\R\p n}$, by \cite[Proposition 3.6]{lee2012smooth} $d(\theta_g)_y\circ d(\theta_{g\p{-1}})_x = \mathrm{Id}_{\R\p n}$. The result now follows by applying the change of variables \cite[Exercise 10.7]{rockafellar2009variational} to $f = f\circ \theta_{g\p{-1}}$ at $x$.
\end{proof}

The symmetry rule is as follows.

\begin{theorem}
    \label{thm:symmetry}
   Let $f:\R^n\to \eR$ be lsc, $G$-invariant, $\overline{x}\in\R\p n$, and $L$ be a supplement of $T_{\overline{x}}G\overline{x}$.
   \begin{itemize}
       \item[\rnum1] If there exist $\rho,\kappa>0$ and $\beta\geq 1$ such that
       \begin{equation}
           \label{eb_normal_space}
           \forall x\in B_\rho(\overline{x})\cap\vec L_{\overline{x}},~~~ (f(x)-f(\overline{x}))_+\geq \kappa\hspace{.3mm}d(x,[f\leq f(\overline{x})])^\beta,
       \end{equation}
       then $f$ has growth exponent $\beta$ at $\overline{x}$. 
       \item[\rnum2]  If there exist $r,c,\ell>0$ and $\alpha\in [0,1)$ such that
       \begin{equation}
           \label{lojasiewicz_normal}
           \forall x\in B_r(\overline{x})\cap [0<f-f(\overline{x})<\ell]\cap\vec L_{\overline{x}}, ~~~ d(0,\cp f(x)) \geq c(f(x)-f(\overline{x}))^\alpha, 
       \end{equation}
       then $f$ has \KL exponent $\alpha$ at $\overline{x}$. 
   \end{itemize}
\end{theorem} 
\begin{proof}
Without loss of generality, $f$ is nonnegative and $f(\overline{x})=0$ (after possibly replacing $f$ by $(f-f(\overline{x}))_+$). Let $\theta:G\times \R\p n\to \R\p n$ denote the action. Since $\theta$ is $C\p 1$ smooth and $d(\theta_e)_{\overline{x}}=\mathrm{Id}_{\R\p n}$, there exists a neighborhood $U\times V$ of $(e,\overline{x})$ in $G\times \R\p n$ such that $V$ is convex and
       \begin{equation}
           \label{bound_singular_value_group}
           \forall (g,x)\in U\times V,~~~ \|d(\theta_g)_x\|_2\leq 2 ~~~\text{and}~~~ \sigma_{\min}(d(\theta_{g})_x^{-*})\geq 1/2.
       \end{equation}
By the mean value theorem, $$\forall (g,x)\in U\times V, ~~~|\theta(g,x)-\theta(g,y)|\leq  \sup_{z\in[x,y]}\|d(\theta_g)_z\|_2 |x-y| \leq 2|x-y|.$$
Since $f$ is lsc and nonnegative, $[f=0]$ is closed. Thus there exists a neighborhood $W$ of $\overline{x}$ in $\vec L_{\overline{x}}$ such that $\emptyset \neq P_{[f=0]}(W)\subseteq V$.
As $\theta|_{G\times \vec L_{\overline{x}}}$ is a submersion at $(e,\overline{x})$ by \cref{lemma:group_submersion}, it is an open map near $(e,\overline{x})$ due to \cite[Proposition 4.28]{lee2012smooth}. Thus $B_{\max\{\rho,r\}}(\overline{x})\subseteq \theta(U,W)$ after possibly reducing $\rho$ and $r$. 

Let $x\in B_\rho(\overline{x})$. There exists $(g,y)\in U\times W$ such that $x=\theta(g,y)$. Thus there exists $z \in P_{[f=0]}(y)\subseteq P_{[f=0]}(W)\subseteq V$. Hence
\begin{align*}
    f(x) & = f(\theta(g,y)) = f(y) \geq \kappa\hspace{.3mm}d(y,[f= 0])\p \beta=\kappa|y-z|\p \beta\geq (\kappa/2\p \beta)|\theta(g,y)-\theta(g,z)|\p \beta \\
    & = (\kappa/2\p \beta)|x-\theta(g,z)|\p \beta \geq (\kappa/2\p \beta)d(x,[f=0])\p \beta
\end{align*}
where $\theta(g,z)\in[f=0]$ because $f(\theta(g,z))=f(z)= 0$. If in addition $x\in [0<f<\ell]$, then $f(y)=f(\theta(g,y)) = f(x)\in(0,\ell)$ and
\begin{align*}
    d(0,\cp f(x)) & = d\bigl(0,d(\theta_g)_y^{-*}\cp f(y)\bigr)
\geq d(0,\cp f(y))/2 \geq (c/2)f(y)\p \alpha \\
& = (c/2) f(\theta(g,y))\p \alpha  = (c/2) f(x)\p\alpha  
\end{align*}
by \cref{lemma:G_chain}.
\end{proof}

While the symmetry rule is quite flexible, it does require the user to choose a supplement $L$. A natural choice is of course the orthogonal complement of the tangent space, namely, the normal space. If the level set is homogeneous and embedded, then we obtain a readily applicable corollary. It extends \cref{lemma:pham} from isolated to certain nonisolated local minima.

\begin{corollary}
    \label{cor:symmetry-lift}
    Let $f:\R^n\to \eR$ be lsc, subanalytic, and $G$-invariant. Let $\overline{x}\in\R\p n$ be a local minimum of $f$ and $\beta \geq 1$. Suppose $G\overline{x}$ is an embedded submanifold of $\R\p n$ that agrees with $[f = f(\overline{x})]$ near $\overline{x}$. Then $f$ has growth exponent $\beta$ at $\overline{x}$ iff $f$ has \KL exponent $1-1/\beta$ at $\overline{x}$.
\end{corollary}
\begin{proof}
    The reverse implication follows from \cref{fact:conversion}. As for the direct implication, we proceed in two steps. 
    Without loss of generality, $f$ is nonnegative and $f(\overline{x})=0$ (after possibly replacing $f$ by $(f-f(\overline{x}))_+$).
    
    \noindent\textit{Step 1:} \textit{Convert the growth exponent on $\R\p n$ to a \KL exponent on the shifted normal space.} By assumption, there exist $\rho,\kappa>0$ such that
$$\forall x\in B_\rho(\overline{x}),~~~f(x)\geq \kappa \hspace{.4mm} d(x,[f=0])\p \beta.$$
Since $G\overline{x}$ is embedded and locally agrees with $[f=f(\overline{x})]$, by \cref{fact_proj}, we have $$\forall x\in B_\rho(\overline{x})\cap \vec{N}_{\overline{x}}G\overline{x},~~~ d(x,[f=0])=|x-\overline{x}|,$$
after possibly reducing $\rho$. Hence
\[
\forall x \in B_\rho(\overline{x}),~~~
\widetilde f(x)\geq \kappa|x-\overline{x}|^\beta,
\]
where $\widetilde f:=f+\delta_{\vec{N}_{\overline{x}}G\overline{x}}$.
%Since $[f=f(\overline{x})]\cap \vec{N}_{\overline{x}}G\overline{x}=\{\overline{x}\}$ locally and $\overline{x}$ is a local minimum of $f$, $\overline{x}$ is a strict local minimum of $\widetilde f$. 
By \cref{lemma:pham},
there exist $r,\ell,c>0$ such that
$$ \forall y\in B_r(\overline{x})\cap
[0<\widetilde f<\ell] ~~ \Longrightarrow ~~ d(0,\cp \widetilde f(y)) \geq c\widetilde f(y)^{1-1/\beta}. $$
\noindent\textit{Step 2:} \textit{Compare $\cp f$ and $\cp \widetilde f$ on the shifted normal space.}  In this step, we aim to prove
       \begin{equation}
       \label{eq:inclusion_cp}
        \forall y\in B_r(\overline{x})\cap \vec N_{\overline{x}}G\overline{x}, ~~~ \cp f(y)\subseteq \cp \widetilde f(y),
      \end{equation}
      by reducing $r$ if necessary. Then
\[  \forall y\in B_r(\overline{x}) \cap \vec N_{\overline{x}}G\overline{x}\cap[0<f<\ell],~~~ d(0,\cp f(y))\geq d(0,\cp \widetilde f(y))\geq cf(y)^{1-1/\beta}, \]
and $f$ has \KL exponent $1-1/\beta$ at $\overline{x}$ by \cref{thm:symmetry}, as desired.

      Recall that $\theta|_{G\times \vec L_{\overline{x}}}$ is an open map near $(e,\overline{x})$ by \cref{lemma:group_submersion}.
      Thus $\theta|_{U\times  W}$ is an open map for some neighborhood $U$ of $e$ in $G$ and $W:=B_r(\overline{x})\cap \vec N_{\overline{x}}G\overline{x}$, after possibly $r$.  In order to prove \cref{eq:inclusion_cp}, we will successively show that
$$\forall y \in  W, ~~~ \widehat \partial f(y)\subseteq \widehat\partial \widetilde f(y), ~\partial f(y)\subseteq \partial \widetilde f(y), ~\text{and}~ \partial\p\infty f(y)\subseteq \partial\p \infty \widetilde f(y).$$

The first inclusion follows the definition: $v\in \widehat\partial f(y)$ means that
$$f(z)\geq f(y)+\langle v, z-y\rangle + o(|z-y|)$$
and so
$$\widetilde f(z) = f(z) + \delta_{\vec{N}_{\overline x}G\overline{x}}(z)\geq f(z) \geq \widetilde f(y)+\langle v, z-y\rangle + o(|z-y|),$$
using the fact that $f(y) = \widetilde f(y)$ as $y\in \vec{N}_{\overline x}G\overline{x}$.

As for the second, let $v\in \partial f(y)$. There exists $(z_k,v_k)\in \gph \widehat\partial f$ such that $(z_k,f(z_k),v_k)\to (y,f(y),v)$.  Since $\theta|_{U\times W}$ is an open map and $z_k\to y\in V$, there exists $(g_k,y_k)\in U\times W$ such that $z_k=\theta(g_k,y_k)$ eventually and $(g_k,y_k)\to (e,y)$. By \cref{lemma:G_chain}, $\widehat \partial f(z_k) = d(\theta_{g_k})_{y_k}\p{-*} \widehat\partial f(y_k)$ and so 
$$w_k = d(\theta_{g_k})_{y_k}\p* v_k \in d(\theta_{g_k})_{y_k}\p* \widehat \partial f(z_k) = \widehat\partial f(y_k)\subseteq \widehat\partial\widetilde f(y_k).$$
Thus $(w_k,y_k) \in \gph \widehat\partial \widetilde   f$. Also, $$(y_k,f(y_k),w_k) = (y_k,f(z_k),d(\theta_{g_k})_{y_k}\p* v_k)\to(y,f(y),d(\theta_{e})_{y}\p* v)=(y,f(y),v).$$
Hence $v\in \widehat\partial\widetilde f(y)$.

The third inclusion is proved similarly. To sum up,
\begin{equation*}
    \forall y\in V, ~~~ \cp f(y) = \overline{\co}[\partial f(y)+\partial\p\infty f(y)]\subseteq \overline{\co}[\partial\widetilde f(y)+\partial\p\infty \widetilde f(y)] = \cp\widetilde f(y).\qedhere
\end{equation*}
\end{proof}

The embeddedness assumption in \cref{cor:symmetry-lift} is generally easy to verify since it holds when the action is semi-algebraic by \cref{fact_embeddedorbit}. If the solution set is merely an embedded submanifold locally, then the conclusion may fail, highlighting the role of symmetry. We construct counterexamples in \cref{sec:ce} for $\beta=1$ and $\beta=2$ where $f$ is locally Lipschitz continuous, semi-algebraic, and Clarke regular. Recall that if $f$ is $C\p 2$ and has quadratic growth, then it has \KL exponent $1/2$  \cite{drusvyatskiy2013second,rebjock2024fast} and the solution set is locally a $C\p 2$ embedded submanifold \cite[Theorem 2.16]{rebjock2024fast} (see \cite[Theorem 1]{feehan2020morse} in the real analytic case).

\section{Applications}
\label{sec:Applications}

We begin by recalling some basic facts in linear algebra. Given $A\in \R^{m\times r}$ and $B\in \R^{r\times n}$,
\begin{gather*}
    \rk (AB) \leq \min\{\rk (A),\rk (B)\}, \\
    \rk (A) + \rk (B)-r \leq \rk (AB), \\
    \rk(AB)=\rk(B)-\dim \ker(A)\cap \im(B).
\end{gather*}
The second is Sylvester's inequality \cite[Section 2.4, Fact 18(e)]{hogben2013handbook} and the third is stated in \cite[Section 16.5, Fact 10(a)]{hogben2013handbook}.
Von Neumann's trace inequality \cite[Theorem 2.1]{lewis1995convex}
   states that given any $A,B\in \R^{m\times n}$, 
    $$ \langle A,B \rangle\leq  \sigma(A)^\top \sigma(B) ,$$
    with equality iff there exist $U\in \mathrm{O}(m)$ and $V\in\mathrm{O}(n)$ such that 
    \[   A=U
       \Sigma(A)V^\top ~~~ \text{and} ~~~  B=U\Sigma(B) V^\top.              \]
Fan's inequality \cite[Theorem 1.2.1]{borwein2006convex} states that given any $A,B\in \Sym^n$, 
    \[    \tr(AB)\leq \lambda(A)^\top\lambda(B),    \]
    with equality iff there exists $U\in \Orth(n)$ such that
    \[  A=U \Lambda(A)U^\top ~~~ \text{and} ~~~  B=U \Lambda(B) U^\top.       \]       

\subsection{Matrix factorization}
\label{subsec:matrix-factorization}
% In this section, we study the \KL exponent of the symmetric and asymmetric matrix factorization problem, that is, $f_{\mathrm{s}}(X):=\|XX^\top-M\|_F^2$ with $M\in \bS^n_+$, and $ f_{\mathrm{a}}(X,Y):=\|XY-M\|_F^2$ with $M\in \R^{m\times n}$. Naturally, $f_{\mathrm{s}}=g_{\mathrm{s}}\circ F_{\mathrm{s}}$ and $f_{\mathrm{a}}=g_{\mathrm{a}}\circ F_{\mathrm{a}}$, where 
% \begin{equation}
%         \label{matrix_fac_defn}
% \begin{aligned}
%     g_{\mathrm{s}}: \bS^n\to \overline{\R}, && g_{\mathrm{s}}(A):=\|A-M\|_F^2+\delta_{\bS^{ n}_{\leq r}}(A), \\
%     F_{\mathrm{s}}:\R^{n\times r}\to \bS^n, && F_{\mathrm{s}}(X):=XX^\top,\\
%     g_{\mathrm{a}}:\R^{m\times n}\to \eR, && g_{\mathrm{a}}(A):=\|A-M\|_F^2+\delta_{\R^{m\times n}_{\leq r}}(A),\\
%     F_{\mathrm{a}}:\R^{m\times r}\times \R^{r\times n}\to \R^{m\times n}, && F_{\mathrm{a}}(X,Y):=XY.
% \end{aligned}
% \end{equation}
% To utilize \cref{cor:composition}, we first recall the properties of the mappings $F_{\mathrm{s}}$ and $F_{\mathrm{a}}$ in the full rank case.

We first define matrix factorization then list some basic properties. 

\begin{definition}[Asymmetric]
    \label{def:asymmetric}
    Given $m,r,n\in \N\p *$ and $M\in \R\p{m\times n}$, let $f_\mathrm{a} := g_\mathrm{a}\circ F_\mathrm{a}$ where 
    $$
    \begin{array}{rccc}
    g_\mathrm{a}: & \R^{m\times n} & \mapsto & \R \\
     & A & \to & \|A-M\|_F^2
    \end{array}
    ~~~\text{and}~~~
    \begin{array}{rccc}
    F_\mathrm{a}: & \R^{m\times r}\times \R^{r\times n} & \mapsto & \R^{m\times n} \\
     & (X,Y) & \to & XY.
    \end{array}
    $$
\end{definition}

\begin{definition}[Symmetric]
    \label{def:symmetric}
    Given $r,n\in \N\p *$ and $M\in \bS^n_+$, let $f_\mathrm{s} := g_\mathrm{s}\circ F_\mathrm{s}$ where 
    $$
    \begin{array}{rccc}
    g_\mathrm{s}: & \bS^n & \mapsto & \R \\
     & A & \to & \|A-M\|_F^2
    \end{array}
    ~~~\text{and}~~~
    \begin{array}{rccc}
    F_\mathrm{s}: & \R^{n\times r} & \to & \bS^n \\
     & X & \to & XX\p T.
    \end{array}
    $$
\end{definition}

We only compute the \KL exponent at global minima because matrix factorization has no spurious second-order stationary points, as proved in \cite{baldi1989neural} when $m=n$, and in \cite{valavi2020revisiting} for any $m,n$. We take this opportunity to present a more elementary proof in \cref{subsec:proof_landscape} which also implies the Eckart-Young theorem.

\begin{theorem}
     \label{theorem:landscape}
     For any $M \in \R^{m\times n}$ with $m\leq n$, the function $f_{\mathrm{a}}$ has no spurious second-order stationary points. Also, $(X,Y)\in \arg\min f_{\mathrm{a}}$ iff there exists an SVD $U\Sigma V^ T$ of $M$ such $XY = \sum_{i=1}^ r \sigma_i u_i v_i^ T$, in which case $\min f_{\mathrm{a}} = \sigma_{r+1}^ 2(M) + \cdots + \sigma_m^ 2(M)$.
\end{theorem}

It implies that symmetric matrix factorization also enjoys a benign landscape, as we show in \cref{subsec:cor_landscape}.

\begin{corollary}
\label{corollary_landscape}
  Let $M\in \bS^n_+$. If $X\in \R^{n\times r}$ is a second-order stationary point of $f_{\mathrm{s}}$, then $(X,X^\top)$ is a second-order stationary point of $f_{\mathrm{a}}$. Thus $f_{\mathrm{s}}$ also has no spurious second-order stationary points. 
\end{corollary}

In order to apply the composition rule (i.e., \cref{cor:composition}), it is useful to know when the inner map has constant rank, as previously discussed in \cite[Section 1.2]{rebjock2024fast}.

\begin{proposition}[{\cite[Section 2]{absil2014two}}]
\label{prop:mf_submersion}
If $\rk(XY)=\min\{m,r,n\}$, then $F_{\mathrm{a}}$ has constant rank at $(X,Y)$
%If $r \geq \min\{m,n\}$, then $F_{\mathrm{a}}$ is surjective.     
% If $\rk X=m$ or $\rk Y = n$, then $d(F_{\mathrm{a}})_{(X,Y)}$ is surjective.
and the restriction $\widetilde F_{\mathrm{a}} : \R_*^{m\times r}\times \R_*^{r\times n} \to \R_r^{m\times n}$ is a submersion at $(X,Y)$. 
\end{proposition}
% \begin{proof}
%Suppose $r \geq \min\{m,n\}$. If $r\geq m$, then let $X := (I_m ~ 0_{m\times (r-m)})$ and observe that $\im F \supseteq X\R^{r\times n} = I_m \R^{m\times n} + 0_{m\times (r-m)} \R^{(r-m)\times n} = \R^{m\times n}$. If $r\geq n$, then let $Y^ T := (I_n ~ 0_{n\times (r-n)})$ and observe similarly that $\im F^\top \supseteq Y^\top \R^{r\times m} = I_n \R^{n\times m} + 0_{n\times (r-n)} \R^{(r-n)\times m} = \R^{n\times m}$. Again, $\im F = \R^{m\times n}$.
% Compute       
% $$\begin{array}{cccc}
%          d(F_{\mathrm{a}})_{(X,Y)}: & \R^{m\times r}\times \R^{r\times n} & \longrightarrow & \R^{m\times n} \\
%         & (H,K) & \longmapsto & XK + HY.
%     \end{array}$$
%     Suppose $r\geq \min\{m,n\}$. Since $\min\{m,n\}=\rk(XY)\leq \min\{\rk(X),\rk(Y)\} \leq \min\{m,n\}$, the last inequality holds with equality. If $\rk(X)<m$ and $\rk(Y)<n$, then $\min\{\rk(X),\rk(Y)\} < \min\{m,n\}$, a contradiction. Thus $\rk(X)<m$ or $\rk(Y)<n$. If $\rk X = m$, then $\im X = \R^ m$ and $\im d(F_{\mathrm{a}})_{(X,Y)} \supseteq X\R^{r\times n} = \R^{m\times n}$. If $\rk Y = n$, then $\im Y^ T = \R^ n$ and $Y^ T\R^{r\times m} = \R^{n\times m}$. Thus $\im d(F_{\mathrm{a}})_{(X,Y)} \supseteq \R^{m\times r}Y = \R^{m\times n}$.
%    The case $r\geq \min\{m,n\}$ is clear. The case where $r\leq \min\{m,n\}$ has been studied in \cite[Section 2]{absil2014two}.
% \end{proof}

An analogous result holds in the symmetric case.

\begin{proposition}[{\cite[Section 2]{vandereycken2009embedded}}]
\label{prop:smf_submersion}
If $\rk(X) = \min\{n,r\}$, then $F_{\mathrm{s}}$ has constant rank at $X$ and the restriction $\widetilde F_{\mathrm{s}} : \R_*^{n\times r} \to \Sym_r^n$ is a submersion.
\end{proposition}
% \begin{proof}
% %If $r\geq n$, then $F$ is surjective by the spectral decomposition.
% %Next, compute       
% Compute
% $$\begin{array}{cccc}
%          d(F_{\mathrm{s}})_X: & \R^{n\times r} & \longrightarrow & \Sym^n \\
%         & H & \longmapsto & XH^ T + HX^ T.
%     \end{array}$$
%     If $r\geq n$ then $\rk (X) = n$, $\im X = \R^ n$, and $X\R^{r\times n}=\R^{n\times n}$, so that $\im d(F_{\mathrm{s}})_X = \Sym^ n$. The case where $r\leq n$ follows from .
% \end{proof}

The \KL exponent of matrix factorization and $\ell_1$ matrix factorization in the exactly parametrized case and overparametrized case with full rank data follow immediately \cref{prop:mf_submersion}, \cref{prop:smf_submersion}, and the composition rule in \cref{coro_composite_strict}. Indeed, those cases are equivalent to $\rk(M)=\min\{m,r,n\}$ in the asymmetric case (resp. $\rk(M)=\min\{n,r\}$ in the symmetric case). The overparametrized case with rank deficient data will be treated in a unified manner with matrix sensing in \cref{subsec:matrix-sensing}. That leaves us with a single case.

\subsubsection{Underparametrized case}

We consider the case where $r< \rk(M)$ (in fact this subsection also covers the case where $r=\rk(M)$).
We begin with a global growth condition on the outer function $g_\mathrm{a}$ using von Neumann's trace inequality. It again implies the Eckart-Young theorem, yielding a proof that seems more direct than existing ones.

%The next lemma provides a global growth condition for underparametrized matrix factorization in the nondegenerate case, i.e., $\sigma_{r}(M)>\sigma_{r+1}(M)$. The proof is based on von Neumann's trace inequality and provides yet another proof of the Eckart-Young theorem (in the nondegenerate case; the degenerate will follow).

\begin{lemma}
    \label{lemma:up}
    Let $M\in\R^{m\times n}$ and $r\in [m]$ with $m\leq n$. Consider an SVD, $M=U \Sigma(M) V^\top$ where
    \[  \Sigma(M)=\begin{bmatrix}
        \Alpha & 0 \\
           0 & \Beta
    \end{bmatrix}, ~~~ L:= U\begin{bmatrix}
        \Alpha & 0 \\
           0 & 0
    \end{bmatrix}V\p \top,~~~ \Delta: = U\begin{bmatrix}
        0 & 0 \\
        0 & \Beta
    \end{bmatrix}V\p \top, \]
    $\Alpha := \diag(\alpha) \in \R^{r\times r}$, $\Beta := \diag(\beta) \in \R^{(m-r)\times (n-r)}$, and $\alpha_1 \geq \cdots \geq \alpha_r \geq \beta_1 \geq \cdots \geq \beta_{m-r}\geq 0$. If $\sigma_{r}(M)=\alpha_r>0$, then
    $$\forall A \in \R_{\leq r}^{m\times n}, ~~~\|A-M\|_F^ 2 - \|\Delta\|_F^ 2 \geq (1-\kappa)\|A-L\|_F^ 2$$
    where $\kappa := \sigma_{r+1}(M)/\sigma_r(M)$. 
\end{lemma}
\begin{proof}
For all $A \in \R_{\leq r}^{m\times n}$,
    \begin{align*}
        & ~ \|A-M\|_F^ 2 - \|\Delta\|_F^ 2 - (1-\kappa)\|A-L\|_F^ 2 \\
        = & ~ \|A\|_F^ 2 - 2\langle A,M\rangle +\|M\|_F^ 2 - \|\Delta\|_F^ 2 - (1-\kappa)(\|A\|_F^ 2-2\langle A,L\rangle +\|L\|_F^ 2) \\
        = & ~ \|A\|_F^ 2 - 2\langle A,L+\Delta\rangle +\|L\|_F^ 2 - (1-\kappa)(\|A\|_F^ 2-2\langle A,L\rangle +\|L\|_F^ 2) \\
        = & ~ \kappa (\|A\|_F^ 2 + \|L\|_F^ 2) - 2\langle A,\kappa L+\Delta\rangle \\
        \geq & ~ \kappa (\|A\|_F^ 2 + \|L\|_F^ 2) - 2\sum_{i=1}^ {\min\{m,n\}} \sigma_i(A)\sigma_i(\kappa L+\Delta) \\
        = & ~ \kappa \sum_{i=1}^ r \sigma_i(A)^ 2 + \sigma_i(L)^ 2 - 2\sigma_i(A)\sigma_i(L) \\
        = & ~ \kappa \sum_{i=1}^ r (\sigma_i(A)-\sigma_i(L))^ 2 \geq 0
    \end{align*}
    where the first inequality is due to von Neumann's trace inequality. 
\end{proof}

% In the  $\kappa=1$, the bound in \cref{lemma:up} is trivial, and we will prove the local growth condition in general in the following. 

When $\kappa<1$ in \cref{lemma:up}, one obtains quadratic growth, which we can convert into a \KL exponent. In the degenerate case $\kappa =1$, \cref{lemma:up} only provides nonnegativity, which is too weak to deduce a \KL exponent. One can in fact obtain quadratic growth in the general case $\kappa \leq 1$, but the proof is substantially harder. It is the object of the forthcoming results. Let $\cP^n_k$ be the set orthogonal projection matrices of order $n$:

$$\cP_k^ n := \{ P \in \Orth(n) : P^ T = P ~\land~ \rk(P) = k  \}$$
Recall the following fact concerning the orthogonal matrices in the SVD of a single matrix.
\begin{fact}[{\cite[Proposition 5]{ding2014introduction}}]
\label{fact_svd}
 Let $\Sigma:=\begin{bmatrix}
     \diag(a_1I_{s_1},\dots,a_{r+1}I_{s_{r+1}}) & 0
 \end{bmatrix}$ with $a_1>a_2\dots>a_r>a_{r+1}=0$ and $\sum_{i=1}^{r+1} s_i=m$. Then, the two orthogonal matrices $P \in \rO(m)$ and $W \in \rO(n)$ satisfy $P \begin{bmatrix} \Sigma & 0\end{bmatrix} = \begin{bmatrix} \Sigma & 0\end{bmatrix}  W$ if and only if there exist $Q \in \rO(m-s_{r+1}), Q^{\prime} \in \rO(s_{r+1})$ and $Q^{\prime \prime} \in \rO(n-m+s_{r+1})$ such that
$$
P=\left[\begin{array}{cc}
Q & 0 \\
0 & Q^{\prime}
\end{array}\right] \quad \text { and } \quad W=\left[\begin{array}{cc}
Q & 0 \\
0 & Q^{\prime \prime}
\end{array}\right],$$
where $Q:=\operatorname{diag}\left(Q_1, Q_2, \ldots, Q_r\right)$ is a block diagonal orthogonal matrix with the $k$-th diagonal block given by $Q_k \in \rO(s_k), k=1, \ldots, r$.
\end{fact}

Below, we analyze the solution set of the outer function.

\begin{lemma}
\label{lemma:sol_under}
    Let $M\in\R^{m\times n}$ and $r\in [m]$ with $m\leq n$. Consider an SVD $M=U \Sigma(M) V^\top$ where
    % $\Sigma = \diag(\sigma)$, $\sigma_1 \geq \cdots \geq \sigma_s > \sigma_{s+1} = \cdots = \sigma_{s+t} > \sigma_{s+t+1} \geq \cdots \geq \sigma_m$ and $r\in\llbracket s+1, s+t \rrbracket$.
    \begin{equation}
        \label{svd_M}
        \Sigma(M)=\begin{bmatrix}
        \mathrm{A} & 0  & 0 \\
           0 & \beta I_t & 0 \\
           0 & 0 & \Gamma
    \end{bmatrix},
    \end{equation}
    $\mathrm{A} := \diag(\alpha)\in \R^{s\times s}$, $\beta \in \R$, $\Gamma  :=\diag(\gamma)\in \R^{(m-s-t)\times (n-s-t)}$, $r\in [s+t]\setminus[s]$, and $\alpha_1 \geq \cdots \geq \alpha_s > \beta > \gamma_1 \geq \cdots \geq \gamma_{m-s-t}\geq 0$. Then
    % If $\sigma_{r+1}(M)=0$, then $\arg\min_{A\in \R^{m\times n}_{\leq r}}\|A-M\|_F^2=\{M\}$. Otherwise, if $\sigma_{r+1}(M)>0$, then 
    % \[   \min_{A\in \R^{m\times n}_{\leq r}} \|A-M\|_F^2=\sum_{i=r+1}^m \sigma^2_i(M),    \]
    % where the optimal value is achieved iff 
%     \begin{equation}
%         \label{defn_omega}
% \widetilde A \in \arg \min_{A\in \R^{m\times n}_{\leq r}} \|A-M\|_F^2 ~~~ \Longleftrightarrow ~~~ \exists P\in \mathcal{P}^t_{r-s} : ~~~  \widetilde{A}=U\begin{bmatrix}
%             \mathrm{A} & 0 & 0 \\
%              0 & \beta P & 0 \\
%              0 & 0 & 0
%         \end{bmatrix}V^\top.
%     \end{equation}
%     the set $\arg \min_{A\in \R^{m\times n}_{\leq r}} \|A-M\|_F^2 $ can be rewritten as 
    \begin{equation}
        \label{defn_omega}
        \Omega:=\arg \min_{A\in \R^{m\times n}_{\leq r}} \|A-M\|_F^2=\left\{U\begin{bmatrix}
            \rA & 0 & 0 \\
             0  & \beta P & 0 \\
              0 & 0 & 0
        \end{bmatrix} V^\top:~P\in \cP^t_{r-s}\right\}.
    \end{equation}
\end{lemma}
\begin{proof} 
    Without loss of generality, $U=I_m$ and $V=I_n$. Using a similar analysis as in \cref{lemma:up}, we have 
    \begin{subequations}
    \begin{align}
        \|A-M\|_F^2-\sum_{i=r+1}^{m}\sigma_i^2(M)&=\|A\|_F^2+\|M\|_F^2-\sum_{i=r+1}^{m}\sigma_i^2(M) -2\langle A,M \rangle\\
        &\geq \|A\|_F^2+\|M\|_F^2-\sum_{i=r+1}^{m}\sigma_i^2(M) -2\sum_{i=1}^m\sigma_i(A)\sigma_i(M) \label{eq:neumann} \\
        & =\sum_{i=1}^r\sigma_i^2(A)+\sum_{i=1}^r \sigma^2_i(M) -2\sum_{i=1}^r\sigma_i(A)\sigma_i(M) \label{eq:rank<=r} \\
        &=\sum_{i=1}^r (\sigma_i(A)-\sigma_i(M))^2\geq 0 \label{eq:>=0}
    \end{align}
    \end{subequations}
    where in \cref{eq:neumann} we have used Von Neumann's trace inequality, and in \cref{eq:rank<=r} we have used the fact that $A\in \R^{m\times n}_{\leq r}$. According to the condition for equality in Von Neumann's trace inequality, \cref{eq:neumann} and \cref{eq:>=0} hold with equality iff there exists $\widehat U\in \mathrm{O}(m)$ and $\widehat V\in \mathrm{O}(n)$ such that 
    \begin{equation}
       A=\widehat U\begin{bmatrix}
           \Alpha & 0  & 0 \\
            0 & \begin{bmatrix}
                \beta I_{r-s} & 0 \\
                0 & 0
            \end{bmatrix} & 0 \\
            0 & 0 & 0
       \end{bmatrix}\widehat V^\top ~~~\text{and}~~~ M=\Sigma=\widehat U \Sigma \widehat V^\top.
    \end{equation}
    Utilizing \cref{fact_svd}, from $\widehat U \Sigma=\Sigma \widehat V$ it follows that $\widehat U=\diag(\widehat U_1,\widehat U_2,\widehat U_3)$, $\widehat V=\diag(\widehat U_1,\widehat U_2, \widehat U_4)$ with $\widehat U_1\in\mathrm{O}(s),~\widehat U_2\in \mathrm{O}(t),~\widehat U_3\in\mathrm{O}(m-s-t),~\widehat U_4\in \mathrm{O}(n-s-t)$. Moreover, the equality $\Sigma=\widehat U \Sigma \widehat V^\top$ also gives that $\widehat U_1 \Alpha \widehat U_1^\top=\Alpha$. Therefore, by direct calculation, we have:
    \[   A=\begin{bmatrix}
           \Alpha & 0  & 0 \\
            0 & \widehat U_2\begin{bmatrix}
                \beta I_{r-s} & 0 \\
                0 & 0
            \end{bmatrix}\widehat U_2^\top & 0 \\
            0 & 0 & 0
       \end{bmatrix}.          \]
       %This proves the necessity of \cref{sol_set_rankr}. The sufficiency of \cref{sol_set_rankr} can be proved by direct calculation:
       The converse inclusion in \cref{defn_omega} can be proved by direct calculation:
       \begin{equation*}
           \|\bar{A}-M\|_F\p 2 = \beta\p 2\|P - I_t\|_F\p 2+\|\Gamma\|_F\p 2 = \beta\p 2 (t-r+s)+\gamma_1\p 2+ \cdots + \gamma_{m-s-t}\p 2 = \sum_{i=r+1}\p m \sigma_i(M)\p 2.\qedhere
       \end{equation*}
\end{proof}

In order to project the matrix variable onto the solution set, the following result will be helpful.

\begin{proposition}
\label{prop1}
    Let $A\in \R^{m\times m}$ and $r\in[m]$. With $H:=(A+A^\top)/2$ and $K:=(A-A^\top)/2$, we have
    \[  \min_{P\in \cP^m_r} \|A-P\|_F^2=\|H\|_F^2+r-2\sum_{i=1}^r\lambda_i(H) +\|K\|_F^2, \]
    where the minimal value is achieved iff there exists an orthogonal matrix $U\in\Orth(m)$ such that 
    \[     H=U^\top \Lambda(H)U ~~~\text{and}~~~ P=U^\top \begin{bmatrix}
        I_r & 0 \\
        0 & 0
    \end{bmatrix}   U,          \]
    in which case $UAU^\top=\Lambda(H)+UKU^\top$. 
\end{proposition}
\begin{proof}
  For any $P\in \cP^m_r$, by direct calculation we have 
     \begin{align*}
        \|A-P\|_F^2&\overset{\rm (a)}{=}\|H+K-P\|_F^2\overset{\rm (b)}{=}\|H-P\|_F^2+\|K\|_F^2 = \|H\|^2_F+\|P\|_F^2-2\langle H,P \rangle +\|K\|_F^2 \\
        &\overset{\rm (c)}{\geq } \|H\|^2_F+r-2\sum_{i=1}^r\lambda_i(H) +\|K\|_F^2,
     \end{align*}
where (a) follows from the definition of $H$ and $K$, (b) follows from the fact that $H,P$ are both symmetric and $K$ is skew-symmetric, (c) follows from Fan's inequality and the fact that $\lambda_1(P)=\dots=\lambda_r(P)=1>\lambda_{r+1}(P)=\dots=\lambda_m(P)=0$ due to that $P\in \cP^m_r$. Moreover, (c) holds with equality iff there exists $U\in \Orth(m)$ such that 
\begin{equation*}
     H=U^\top\Lambda(H) U,~P=U^\top \begin{bmatrix}
    I_r  & 0 \\
     0 & 0
\end{bmatrix}U.  \qedhere
\end{equation*}  
\end{proof} 

We now arrive at the main result of this section.

\begin{theorem}
\label{under_asym_eb}
    Given $M\in \R\p {m\times n}$ with $\sigma_r(M)>0$, there exists $\epsilon>0$ such that 
    \[ \forall A\in \R^{m\times n}_r\cap B_\epsilon(\Omega),~~\|A-M\|_F^2-\sum_{i=r+1}^m\sigma_i^2(M) \geq \epsilon d(A,\Omega)^2 ,  \]
    where $\Omega := \arg\min_{A\in \R^{m\times n}_r} \|A-M\|_F^2$.
        % For the optimization problem $\min_{A\in \R^{m\times n}_r} \|A-M\|_F^2$, under the assumption that $\sigma_r(M)>0$, and the global minima set $\Omega$ given in \cref{defn_omega},
\end{theorem}
\begin{proof}
Without loss of generality, we assume that $U=I_m$ and $V=I_n$ in the SVD of $M$, and then $M=\Sigma(M)$, which is given in \cref{svd_M}. In the following, we define $f(A):=\|A-M\|_F^2+\delta_{\Mr}(A)$. 

Consider the action of $G:=\rO(t)$ on $\R^{m\times n}$:
\[  Q\cdot X:= \diag(I_{s},Q,I_{m-s-t}) X \diag(I_{s},Q^\top,I_{n-s-t}).                     \]
Then, $f$ is $G$-invariant, lsc around $\Omega$, and $\Omega$ is an orbit of $G$, which is also a smooth embedded submanifold of $\R^{m\times n}$, since $\Omega\simeq \cP^{t}_{r-s}$. To simplify the calculation, we select a specific $\bar{A}\in \Omega$ given by 
\[    \bar{A}=\begin{bmatrix}
   \rA & 0 & 0 & 0 & 0 \\
    0 & \beta I_{r-s} & 0 & 0 &0  \\
     0 & 0  & 0 & 0 & 0  \\
      0 & 0 & 0 & 0 & 0
\end{bmatrix}.               \]

We proceed with the following steps:

\noindent \textit{Step 1:} \textit{Calculate the normal space of $\Omega$.} Using \cref{fact_orbit}, we know that 
$$T_{\Omega}(\bar{A})=\fg\cdot \bar{A}=\{\diag(0_s,K,0_{m-s-t})\bar{A}+\bar{A} \diag(0_s,K^\top,0_{n-s-t}) :~K=-K^\top\in \R^{t\times t}\}. $$
Rewrite skew symmetric matrix $K\in \R^{t\times t}$ in block forms:
\begin{align*}
   &  K=\begin{bmatrix}
    K_1 & K_2 \\
    -K_2^\top & K_3
\end{bmatrix},\\
&K_1\in \R^{(r-s)\times (r-s)},K_3\in\R^{(t+s-r)\times (t+s-r)} \text{ are skew symmetric, } K_2\in \R^{(r-s)\times (t+s-r)}.      
\end{align*}
By direct calculation, we see that 
\begin{equation*}
   T_{\Omega}(\bar{A})=\left\{\begin{bmatrix}
   0 & 0 & 0 & 0 & 0 \\
    0 & 0 & -K_2 & 0 &0  \\
     0 & -K_2^\top  & 0 & 0 & 0  \\
      0 & 0 & 0 & 0 & 0
\end{bmatrix}:~ K_2\in \R^{(r-s)\times (t+s-r)}\right\}.  
\end{equation*}
Let $J_1:=[r]\setminus[s]$,~$J_2:=[s+t]\setminus[r]$. Then, by direct calculation, we have
\begin{equation}
    \label{normal_omega}
    N_{\Omega}(\bar{A})=(T_{\Omega}(\bar{A}))^\perp=\{B\in \R^{m\times n}:~B_{J
    _1J_2}=-B_{J_2J_1}^\top\}.
\end{equation}

\noindent \textit{Step 2:} \textit{Argue by contradiction and transform the condition to the relationship between $T^2_{\Mr}(\bar{A},B)$ and $T_{\Mr}(\bar{A})$ for some $B\in T_{\Mr}(\bar{A})$.} We aim to prove that there exist $\epsilon,\rho>0$ such that   
\begin{equation}
    \label{matrix_ebnormal}
  \forall C\in  \vec N_{\bar{A}}\Omega\text{  with  }\|C-\bar{A}\|_F<\epsilon,\quad   f(C)-f(\bar A)\geq \rho\|C-\bar{A}\|_F^2.
\end{equation}
If \cref{matrix_ebnormal} is not true, then we can find $B_k\in N_{\Omega}(\bar{A})$ and $t_k,\rho_k\downarrow  0$ such that 
\begin{equation}
    \label{ce_eq}
     \|B_k\|_F=1,\quad  A_k:=\bar{A}+t_kB_k\in \Mr,~ \|A_k-M\|_F^2-\|\bar{A}-M\|_F^2<\rho_kt_k^2.              
\end{equation}
Passing to a subsequence, we may assume that $B_k\to B$. Since $A_k=\bar{A}+t_kB_k\in \Mr\cap \vec N_{\overline{x}}\Omega$, we know that 
\begin{equation}
    \label{inclusion_B}
    B\in T_{\Mr}(\bar A) \cap [T_{\vec N_{\overline{x}}\Omega}(\bar{A})]=T_{\Mr}(\bar A) \cap N_{\Omega}(\bar{A}). 
\end{equation}
Since $\bar{A}\in \Omega=P_{\Mr}(M)$, we know that $M-\bar{A}\in N_{\Mr}(\bar{A})$, which implies that
\begin{equation}
    \label{BperpAm}
  \forall G\in T_{\Mr}(\bar{A}),\quad   \langle G,\bar{A}-M \rangle=0. 
\end{equation}
Then, we can rewrite $A_k$ as:
\[      A_k=\bar{A}+t_kB+\frac{t_k^2}{2}\xi_k,\quad  \xi_k=2(B_k-B)/t_k,~t_k\|\xi_k\|_F\to 0.     \]
According to \cref{ce_eq}, expanding the Frobenius norm, we have 
\begin{subequations}
    \begin{align}
   \|A_k-M\|_F^2-\|\bar{A}-M\|_F^2&=2\langle t_kB+\frac{t_k^2}{2}\xi_k,  \bar{A}-M\rangle+t_k^2\|B+\frac{t_k}{2}\xi_k\|_F^2 \label{23a} \\
   &=t_k^2\langle \xi_k, \bar{A}-M\rangle+t_k^2\|B+\frac{t_k}{2}\xi_k\|_F^2<\rho_kt_k^2. \label{23b}
\end{align}
\end{subequations}
Indeed, in \cref{23a} we rewrite $A_k-M=\bar{A}-M+t_kB+t_k^2/2\xi_k$ and then expand the Frobenius norm. In \cref{23b} we have used the fact $B\in T_{\Mr}(\bar{A})$ and \cref{BperpAm}. Using second-order regularity of the smooth manifold $\Mr$, \cref{defn:soregular}, and the structure of the second-order tangent set of $\Mr$ in \cref{Mrsotangent}, we know that 
\begin{equation}
    \xi_k = B\bar{A}^\dagger B+ G_k+H_k,~G_k\in T_{\Mr}(\bar{A}),~H_k\to 0.
\end{equation}
Dividing both sides of \cref{23b} by $t_k^2$, and using \cref{BperpAm} to show that $\langle G_k,\bar{A}-M \rangle=0$, we have 
\begin{equation*}
    \langle  B\bar{A}^\dagger B+H_k, \bar{A}-M\rangle+\|B+\frac{t_k}{2}\xi_k\|_F^2<\rho_k.
\end{equation*}
Let $k\to\infty$ and recall that $\|B\|_F=1$ and $t_k\|\xi_k\|\to 0$, we have 
\begin{equation}
    \label{step2final}
    \langle  B\bar{A}^\dagger B, \bar{A}-M\rangle\leq -1=-\|B\|_F^2.
\end{equation}

\noindent \textit{Step 3:} \textit{Further use the fact that $B\in N_{\Omega}(\bar{A})$ to derive a contradiction.}  Let us now write $B$ and $\bar{A}-M$ in block forms:
\begin{equation}
    \label{blockab}
    B=\begin{bmatrix}
        B_1 & B_2 & B_3 & B_4 & B_5 \\
        B_6 & B_7 & B_8 & B_9 & B_{10} \\
        B_{11} & B_{12} & B_{13} & B_{14} & B_{15} \\
        B_{16} & B_{17} & B_{18} & B_{19} & B_{20} 
    \end{bmatrix},\quad \bar{A}-M=\begin{bmatrix}
        0 & 0 & 0 & 0 &  0 \\
                0 & 0 & 0 & 0 &  0 \\
        0 & 0 & -\beta I_{s+t-r} & 0 &  0 \\
        0 & 0 & 0 & -\Gamma &  0 
    \end{bmatrix}.
\end{equation}
Using the block forms in \cref{blockab}, we now recalculate the inner product in \cref{step2final}:
\begin{subequations}
\begin{align}
      &\langle  B\bar{A}^\dagger B, \bar{A}-M\rangle\notag\\
      &=-\langle B_{11}\rA^{-1}B_3+1/\beta B_{12}B_8, -\beta I_{s+t-r}\rangle  - \langle B_{16}\rA^{-1}B_4+1/\beta B_{17}B_9,\Gamma\rangle  \label{30a}\\
    &= -\beta\tr(B_{11}\rA^{-1}B_3) -\tr(B_{12}B_8)-\tr(\Gamma B_{16}\rA^{-1}B_4)-\tr(\Gamma B_{17}B_9)/\beta  \label{30b} \\
    &\geq -\beta\|B_{11}\|_F\|\rA^{-1}B_3\|_{F}-\tr(B_{12}B_8)-\|\Gamma B_{16}\|_F\|\rA^{-1}B_4\|_F-\|\Gamma B_{17}\|_F\|B_9\|_F/\beta \label{30c}\\
    &\geq  -\beta\|B_{11}\|_F\|\rA^{-1}B_3\|_{F}-\|\Gamma B_{16}\|_F\|\rA^{-1}B_4\|_F-\|\Gamma B_{17}\|_F\|B_9\|_F/\beta \label{30d}\\ 
    &\geq  -\beta\|\rA^{-1}\|_2\|B_{11}\|_F\|B_3\|_{F}-\|\Gamma\|_2\| B_{16}\|_F\|\rA^{-1}\|_2\|B_4\|_F-\|\Gamma\|_2\| B_{17}\|_F\|B_9\|_F/\beta \label{30e}       \\ 
    &= -\frac{\beta}{\alpha_s}\|B_{11}\|_F\|B_3\|_{F}-\frac{\gamma_1}{\alpha_s}\| B_{16}\|_F\|B_4\|_F-\frac{\gamma_1}{\beta}\| B_{17}\|_F\|B_9\|_F \label{30f} \\
    &\geq -\frac{\beta}{\alpha_s}(\|B_{11}\|_F^2+\|B_3\|_F^2)-\frac{\gamma_1}{\alpha_s}(\|B_{16}\|_F^2+\|B_4\|_F^2)-\frac{\gamma_1}{\beta}(\|B_17\|_F^2+\|B_9\|_F^2)  \label{30g} \\
    &\geq -\max\left\{\frac{\beta}{\alpha_s},\frac{\gamma_1}{\alpha_s},\frac{\gamma_1}{\beta}\right\}( \|B_{11}\|_F^2+\|B_3\|_F^2+\|B_{16}\|_F^2+\|B_4\|_F^2+\|B_{17}\|_F^2+\|B_9\|_F^2)  \\ 
    &\geq -\max\left\{\frac{\beta}{\alpha_s},\frac{\gamma_1}{\alpha_s},\frac{\gamma_1}{\beta}\right\}\|B\|_F^2=-\max\left\{\frac{\beta}{\alpha_s},\frac{\gamma_1}{\alpha_s},\frac{\gamma_1}{\beta}\right\}. \label{30i}
\end{align}
\end{subequations}
Indeed, in \cref{30a} we have used the block structures in \cref{blockab} and omit all the unused blocks. In \cref{30b}, we have used the fact that $\langle A,B\rangle=\tr(A^\top B)=\tr(B^\top A)$ by the definition of the inner product between matrices. In \cref{30c}, we have used Cauchy-Schwartz inequality to show that $\tr(AB)\leq \|A\|_F\|B\|_F$. In \cref{30d}, we have used the condition $B\in N_{\Omega}(\bar{A})$ and the structure of $N_{\Omega}(\bar{A})$ in \cref{normal_omega} to show that $B_{8}=-B_{12}^\top$, and then 
$$\tr(B_8B_{12})=-\tr(B_{12}^\top B_{12})=-\|B_{12}\|_F^2\leq 0.$$  
In \cref{30e} we have used the fact that $\|AB\|_F\leq \|A\|_2\|B\|_F$ for any matrices $A,B$. In \cref{30f}, we calculate $\|\rA^{-1}\|=1/\alpha_s$ and $\|\Gamma\|_2=\gamma_1$ given the structure of $\rA$ and $\Gamma$ in \cref{svd_M}. In \cref{30g}, we have used the basic inequality $xy\leq (x^2+y^2)/2$ for all $x,y\in \R$. \cref{30i} follows from the block structure of $B$ in \cref{blockab} and the condition that $\|B\|_F=1$. Therefore, using \cref{step2final}, we get that 
\[          1\leq     \max\left\{\frac{\beta}{\alpha_s},\frac{\gamma_1}{\alpha_s},\frac{\gamma_1}{\beta}\right\} ,       \]
which yields contradiction since $\alpha_s>\beta >\gamma_1\geq 0$ as in \cref{svd_M}. Therefore, we know \cref{eb_normal_space} holds for some $\epsilon,\rho>0$.

\noindent \textit{Step 4:} \textit{Transform growth condition on the normal space to a neighborhood of $\Omega$.} Using \cref{eb_normal_space}, the fact that $\|C-\bar{A}\|_F\geq d(C,\Omega)$ and \cref{thm:symmetry}(i), we know that there exists $\epsilon,\rho>0$ such that:
\[  \forall C\in B_\epsilon(\bar{A}),~ f(C)-f(\bar{A})\geq \rho d(C,\Omega)^2.           \]
Since $\Omega=G\bar{A}$ and the action of $G$ does not change Frobenius norm, the conclusion then follows.
\end{proof}
The growth condition in the asymmetric case immediately implies the growth condition in the symmetric case. Indeed, the symmetric case can be relaxed to the asymmetric case while preserving the solution set.
\begin{corollary}
\label{under_sym_eb}
  Given $M\in \bS^n_+$ with $\lambda_r(M)>0$, there exists $\epsilon>0$ such that 
    \[ \forall A\in \bS^{n}_{r}\cap B_\epsilon(\Omega),~~\|A-M\|_F^2-\sum_{i=r+1}^n\lambda_i^2(M) \geq \epsilon d(A,\Omega)^2,    \]
    where $\Omega := \arg\min_{A\in \bS^{n}_{r}} \|A-M\|_F^2$.
    %. For the optimization problem $\min_{A\in \bS^{n}_r} \|A-M\|_F^2$, and the global minima set $\Omega$
\end{corollary}
Now, using the single-orbit structure of the solution set, we can deduce the \KL exponent of the underparametrized matrix factorization problem.
\begin{corollary}
\label{coro_under_KL}
  Consider $g_{\mathrm{a}}$ in \cref{def:asymmetric} and $g_{\mathrm{s}}$ in \cref{def:symmetric}. When $\sigma_r(M)>0$, $\widetilde g_{\mathrm{a}}:=g_{\mathrm{a}}+\delta_{\R^{m\times n}_{\leq r}}$ has \KL exponent $1/2$ at any global minimum. When $M\in \bS_+^n$ with $\lambda_r(M)>0$, $\widetilde g_{\mathrm{s}}:=g_{\mathrm{s}}+\delta_{\bS^n_{\leq r}}$ has \KL exponent $1/2$ at any global minimum. Moreover, both $f_{\mathrm{a}}$ and $f_{\mathrm{s}}$ have \KL exponent $1/2$ at any global minimum.
\end{corollary}
\begin{proof}
Let $t$ be the multiplicity of $\sigma_r(M)$ (resp. $\lambda_r(M)$), and $s$ be the total multiplicities of all the singular values (resp. eigenvalues) that are greater than $\sigma_r(M)$ (resp. $\lambda_r(M)$), see \cref{lemma:sol_under} for a more concrete definition. Let the group $\rO(t)$ act on both $\R^{m\times n}$ and $\bS^n$ in the following ways:
  \begin{align*}
   \forall A\in \R^{m\times n},\quad    U\cdot A:= \diag(I_s,U, I_{m-s-t})A\diag(I_s,U^\top,I_{n-s-t}),  \\ 
   \forall A\in\bS^n , \quad U\cdot A :=   \diag(I_s,U, I_{n-s-t})A\diag(I_s,U^\top,I_{n-s-t}). 
  \end{align*}
  Both $\widetilde  g_{\mathrm{a}}$ and $\widetilde g_{\mathrm{s}}$ are invariant under the action of $\rO(t)$, and the global minima set $\Omega$ is an orbit of $\rO(t)$. Since $\widetilde g_{\mathrm{a}}$ and $\widetilde g_{\mathrm{s}}$ have quadratic growth by \cref{under_asym_eb} and \cref{under_sym_eb}, they have \KL $1/2$ by \cref{cor:symmetry-lift}. Together with the submersive properties of $F_{\mathrm{a}}$ in \cref{prop:mf_submersion} and $F_{\mathrm{s}}$ in \cref{prop:smf_submersion}, $f_{\mathrm{a}}$ and $f_{\mathrm{s}}$ have \KL exponent $1/2$ by the composition rule in \cref{cor:composition}. 

  Another possibility is to convert the growth exponent of the outer functions $\widetilde  g_{\mathrm{a}}$ and $\widetilde g_{\mathrm{s}}$ to the objective functions $f_{\mathrm{a}}$ and $f_{\mathrm{s}}$ via the composition rule in \cref{cor:composition}, then apply the equivalence between quadratic growth and \KL 1/2, as shown by Rebjock and Boumal \cite{rebjock2024fast}.
\end{proof}

%Let $\{(X_k,Y_k)\}_{k\in\N}\subseteq\R\p{m\times r}\times \R\p{r\times n}$, $\{\alpha_k\}_{k\in\N}\subseteq(0,\infty)$ with $\im(X_0)\subseteq \im(M)$, and $(\widehat X_k,\widehat Y_k)=\Phi(X_k,Y_k)$.
% where
% $$X = U \begin{bmatrix}
%         \widehat X_k \\
%         \widetilde X_k 
%     \end{bmatrix} ~~~\text{and}~~~Y_k=\widehat Y_kV^\top .$$
 
% \begin{remark}
% \cref{coro:linearconver} means that for the initialization given in \cref{initial_over}, one can view it as a gradient iteration with the same step sizes applied on a full rank or exactly parametrized matrix factorization problem, which has \KL exponent $1/2$ at any global minimum.  
%     % Unfortunately, there is no similar initialization in the symmetric over-parametrized case to get linear convergence for constant step size gradient method, see \cite{xiong2023over}. However, it is possible to achieve (nearly) linear convergence rate by using adaptive step size gradient method, see \cite{davis2025gradient}.
% \end{remark}

\subsubsection{Linear convergence}
Since $f(x,y)=(xy)\p 2$ has tight \KL exponent 3/4 at the origin, it is not clear how to obtain linear convergence in the overparametrized case with rank deficient data. Nevertheless, the initialization $(X_0,Y_0)=(MA,B)$ where $A\in \R^{n\times r}$ and $B\in \R^{r\times n}$ are i.i.d. Gaussian random matrices yields linear convergence of alternating gradient descent \cite[Theorem 5.1]{ward2023convergence} with high probability. Below, we show that linear convergence actually holds for almost every $(A,B)\in \R^{n\times r} \times \R^{r\times n}$ with gradient descent by reducing to case where $\rk(M)=\min\{m,r,n\}$ where the \KL exponent is 1/2.
% \begin{corollary}
% \label{coro:linearconver}
%     Gradient descent scheme for $f_a$ with data matrix $M\in \R^{m\times n}_{\leq r}$ and initial value $\im(X_0)\subseteq \im (M)$ is equivalent to gradient descent scheme for $f_a$ with data matrix $\widehat M\in \R^{s\times n}_*$.
% \end{corollary}

\begin{lemma}
\label{lemma_equi_deficient}
    Given $M\in \R^{m\times n}$, consider an SVD 
    \begin{equation}
        \label{svd_M_deficient}
          M =U\Sigma(M)V^\top = U\begin{bmatrix}
    \Alpha & 0 \\
      0 & 0
\end{bmatrix}V^\top,
    \end{equation}
where $\Alpha \in \R^{s\times s}$ is a positive diagonal matrix and $s=\rk(M)$. Let $\Phi:\R\p {s\times r}\times \R\p{r\times n}\to\R\p{m\times r}\times \R\p{r\times n}$ and $\Psi:\R^{n\times r}\times \R^{r\times n}\to \R^{s\times r}\times \R^{r\times n}$ be linear maps respectively defined by 
$$\Phi(\widehat X,\widehat Y) := \left(U\begin{bmatrix}
     \widehat X \\
     0
\end{bmatrix},\widehat YV^\top\right)~~~\text{and}~~~
       \Psi(A,B):=\left(U\begin{bmatrix}
          \rA & 0
      \end{bmatrix} A, BV^\top\right).$$
Also, let $\widehat f_{\mathrm{a}} := f_{\mathrm{a}}\circ \Phi$. The following are true:
\begin{enumerate}[label=\rm{(\rm{\roman*})}]
    \item $\Phi$ is injective and $\Psi$ is surjective; \label{item:deficient_2}
    \item $\forall (A,B)\in \R^{n\times r}\times \R^{r\times n},~(\Phi \circ \Psi)(A,B)=(MA,B)$. \label{item:deficient_3}
    \item $\forall (\widehat X,\widehat Y)\in \R\p {s\times r}\times \R\p{r\times n},~ \widehat f_{\mathrm{a}}(\widehat X,\widehat Y) = \|\widehat X \widehat Y-[\Alpha,0]\|_F\p 2$; \label{item:deficient_4} 
    \item $\Phi\circ \nabla \widehat f_\mathrm{a} = \nabla f_{\mathrm{a}} \circ \Phi$; \label{item:deficient_5}
    \item \label{item:deficient_6} Let $\{(\widehat X_k,\widehat Y_k)\}_{k\in\N}\subseteq\R\p{s\times r}\times \R\p{r\times n}$, $(X_k,Y_k)=\Phi(\widehat X_k,\widehat Y_k)$, and $\{\alpha_k\}_{k\in\N}\subseteq(0,\infty)$. For all $k\in\N$,
    $$(\widehat X_{k+1},\widehat Y_{k+1})=(\widehat X_k,\widehat Y_k)-\alpha_k \nabla \widehat f_{\mathrm a}(\widehat X_k,\widehat Y_k)~\Longleftrightarrow~ ( X_{k+1}, Y_{k+1})=( X_k, Y_k)-\alpha_k \nabla  f_{\mathrm a}( X_k, Y_k).$$
\end{enumerate}
\end{lemma}
\begin{proof} We treat each item in succession.
\begin{enumerate}[label=\rm{(\rm{\roman*})}]
%\item Obvious. 
\item Obvious.
\item This follows from the SVD of $M$ in \cref{svd_M_deficient} and direct calculation.
\item Observe that
   \[ \widehat f_{\mathrm{a}}(\widehat X, \widehat Y)=\left\|U\begin{bmatrix}
           \widehat X \\
             0 
       \end{bmatrix}\widehat Y V^\top   -U \begin{bmatrix}
           \rA & 0 \\
            0 & 0 
       \end{bmatrix}V^\top   \right\|_F^2= \left\|\widehat X \widehat Y -\begin{bmatrix}
            \rA & 0
       \end{bmatrix}\right\|_F^2.
   \]
\item Since $\Phi$ is semi-orthogonal, i.e., $\Phi\circ\Phi^*=\mathrm{Id}_{\R\p{m\times r}\times\R\p{r\times n}}$, by the chain rule
\[  \Phi\circ  \nabla \widehat f_{\mathrm a}=\Phi\circ \Phi^* \circ \nabla f_{\mathrm{a}}\circ \Phi=\nabla f_{\mathrm a}\circ \Phi.      \]
\item The direct implication is obtained by applying $\Phi$ on both sides and using \ref{item:deficient_5}. As for the converse, we have
\begin{align*}
    \Phi(\widehat X_{k+1},\widehat Y_{k+1}) & =\Phi(\widehat X_k,\widehat Y_k)-\alpha_k (\nabla f_{\mathrm a}\circ \Phi)(\widehat X_k,\widehat Y_k) \\
    & = \Phi(\widehat X_k,\widehat Y_k)-\alpha_k (\Phi \circ \nabla \widehat f_{\mathrm a})(\widehat X_k,\widehat Y_k) \\
    & = \Phi[(\widehat X_k,\widehat Y_k)-\alpha_k \nabla \widehat f_{\mathrm a}(\widehat X_k,\widehat Y_k)]
\end{align*}
and we conclude by injectivity of $\Phi$.   \qedhere
\end{enumerate}
\end{proof}

\cref{lemma_equi_deficient} implies the following general convergence result.
 
\begin{theorem}
    \label{thm:GD_linear}
  Let $M\in \R^{m\times n}$ and $\mathcal{A}\times \mathcal{B}$ be an open bounded subset of $\R^{n\times r}\times \R^{r\times n}$. There exists $\overline{\alpha}>0$ such that for all $\alpha\in(0,\overline{\alpha}]$, gradient descent with step size $\alpha$ applied to $f_\mathrm{a}$ and initialized at $(X_0,Y_0) = (MA,B)$ converges linearly to a global minimum for almost every $(A,B)\in \mathcal{A}\times \mathcal{B}$.
\end{theorem}
\begin{proof}
We use the notations in \cref{lemma_equi_deficient}. Since $\Psi$ is linear and surjective by \cref{lemma_equi_deficient} \ref{item:deficient_2}, it is an open map. Thus $\Psi(\mathcal A\times \mathcal B)$ is an open bounded subset of $\R\p{s\times r}\times \R\p {r\times n}$. By \cite[Example 1]{josz2023global}, there exists $\overline{\alpha}>0$ such that for every $\alpha\in (0,\overline{\alpha}]$, there is a null set $\mathcal{N}_\alpha\subseteq \R\p{s\times r}\times \R\p {r\times n}$ such that gradient descent with step size $\alpha$ applied to $\widehat f_{\mathrm{a}}$ initialized in $\Psi(\mathcal A\times \mathcal B)\setminus \mathcal N_\alpha$ converges to a global minimum of $\widehat f_{\mathrm{a}}$. The rate is linear since $\rk([\rA ~~ 0 ])=\min\{s,n\}$ and every global minimum of $\widehat f_{\mathrm{a}}$ has \KL exponent $1/2$ by \cref{coro_under_KL} if $r<s$ (resp. by \cref{prop:mf_submersion} and \cref{prop:smf_submersion} if $r\geq s$).
We next show that, for every $\alpha\in (0,\overline{\alpha}]$, gradient descent with step size $\alpha$ applied to $f_{\mathrm{a}}$ initialized in $$\{(MA,B):(A,B)\in (\mathcal A\times \mathcal B)\setminus \Psi^{-1}(\mathcal N_{\alpha})\}$$ converges linearly to a global minimum of $f_{\mathrm{a}}$. 
The conclusion then follows because $\Psi^{-1}(\mathcal N_{\alpha})$ is a null set. Indeed, the preimage of a zero measure set has zero measure for surjective linear maps (using a change of coordinates, one can reduce to the case where the map is a projection, then use Tonelli's theorem).

Let $\alpha\in(0,\overline\alpha]$ and $\{(X_k,Y_k)\}_{k\in\N}$ be a gradient sequence of $ f_\mathrm{a}$ with step size $\alpha$ initialized at $(MA,B)$ for some $(A,B)\in (\mathcal A\times \mathcal B)\setminus \Psi^{-1}(\mathcal N_{\alpha})$. 
Let $\{(\widehat X_k,\widehat Y_k)\}_{k\in\N}$ denote a gradient sequence of $\widehat f_\mathrm{a}$ with step size $\alpha$ with initial point $(\widehat X_0,\widehat Y_0)=\Psi(A,B)$. By \ref{item:deficient_6}, $\{\Phi(\widehat X_k,\widehat Y_k)\}_{k\in\N}$ is a gradient sequence of $f_\mathrm{a}$ with step size $\alpha$ and initial point $$\Phi(\widehat X_0,\widehat Y_0) = (\Phi\circ\Psi)(A,B)=(MA,B)=(X_0,Y_0)$$
by \ref{item:deficient_3}.
Thus $\Phi(\widehat X_k,\widehat Y_k)=(X_k,Y_k)$ for all $k\in\N$. Since $(\widehat X_0,\widehat Y_0)\in\Psi[(\mathcal A\times \mathcal B)\setminus \Psi^{-1}(\mathcal N_{\alpha})]=\Psi(\mathcal A\times \mathcal B)\setminus \mathcal N_\alpha$, $(\widehat X_k,\widehat Y_k)$ converges linearly to a global minimum of $\widehat f_{\mathrm{a}}$. As $\Phi$ is linear, $(X_k,Y_k)$ also converges linearly. By \ref{item:deficient_4}, 
$$f_\mathrm{a}(X_k,Y_k) = \widehat f_\mathrm{a}(\widehat X_k,\widehat
 Y_k)\to \sum_{i=r+1}^m\sigma_{i}(M)^2.$$ Thus $(X_k,Y_k)$ converges to a global minimum of $f_{\mathrm{a}}$.
\end{proof}
\subsection{Linear neural network}

We first define what a linear neural network is.
\begin{definition}
    \label{def:lnn}
        Given $d_0,\dots,d_\ell\in \N\p *$, $X\in \R^{d_0\times n}$, and $Y\in \R^{d_\ell\times n}$, let $f_\mathrm{n}:=g_\mathrm{n}\circ F_\mathrm{n}$ where
    \begin{align*}
     &\begin{array}{rccc}
    g_\mathrm{n}:~&\R^{d_\ell\times n}&\to& \R \\
    &A &\to& \|A-Y\|_F^2
    \end{array}     ~~~\text{and}~~~
   &  \begin{array}{rccc}
    F_\mathrm{n}:~& \R^{d_\ell\times d_{\ell-1}} \times \dots \times \R^{d_{1}\times d_{0}} &\to& \R^{d_{\ell}\times n}
  \\
    &(W_\ell,\dots, W_1)&\to& W_{\ell}\cdots W_1X.
    \end{array} 
\end{align*}
\end{definition}

The next proposition gives the condition for $F_{\mathrm n}$ being a constant rank mapping.
\begin{proposition}
    If $\rk(W_{\ell}\cdots W_{1})=d_{\ell}$, then $F_{\mathrm{n}}$ has constant rank near $(W_{\ell},\dots,W_1)$.
\end{proposition}
\begin{proof}
Since 
    \[   d_{\ell}=\rk(W_{\ell}\cdots W_1)\leq \min_{j\in [\ell]}\{ \rk(W_{\ell}\cdots  W_{j})\}\leq d_\ell   \],
     for all $j\in [\ell]$, $\rk(W_{\ell}\cdots  W_{j})=d_\ell$. Since for all $j\in [\ell -1]$, 
      \begin{align*}
            F_{j}=\widetilde F_j\times \mathrm{Id}_{\R^{d_{\ell-j-1}\times d_{\ell - j -2}}}\times \cdots \times \mathrm{Id}_{\R^{d_1\times d_0}},
    \end{align*}
    where 
      \begin{align*}
          \begin{array}{rccc}
    \widetilde F_j:~&\R^{d_{\ell}\times d_{\ell-j}}\times \R^{d_{\ell-j}\times d_{\ell-j-1}}&\to& \R^{d_{\ell}\times d_{\ell-j-1}}\times\cdots\times \R^{d_1\times d_0}
  \\
    &(W_{\ell-j+1},W_{\ell-j})&\to& W_{\ell-j+1}W_{\ell-j}.
    \end{array}    
      \end{align*}
    Each $d(\widetilde F_j)_{({W}_{\ell}\cdots {W}_{\ell-j+1}  ,{W}_{\ell-j})}$ for $j\in [\ell-1]$ is surjective by \cref{prop:mf_submersion},  because 
     \[ \rk({W}_{\ell}\cdots  W_{\ell-j+1})=d_{\ell}.     \]
     Then, each $d(F_{j})_{({W}_{\ell}\cdots {W}_{\ell-j+1} ,\cdots, {W}_1)}$ for $j\in [\ell-1]$ is surjective. By the chain rule, $$\im d(F_{\mathrm n})_{(\overline{W}_\ell,\dots,\overline{W}_1)}=\left(d(F_{\ell})_{{W}_{\ell}\cdots{W}_1}\circ \cdots \circ  d(F_{1})_{({W}_{\ell},\dots,{W}_1)}\right)[\R^{d_{\ell}\times d_{\ell-1}}\times \dots\times \R^{d_1\times d_0}].$$
    Consequently, due to the surjectivity of each $d(F_{j})_{({W}_{\ell}\cdots {W}_{\ell-j+1} ,\cdots, {W}_1)}$ for $j\in [\ell-1]$, we have 
    \[    \im d(F_{\mathrm n})_{({W}_\ell,\dots,{W}_1)}=d(F_{\ell})_{W}(\R^{d_\ell\times d_0})=\R^{d_\ell \times d_0}X.               \]
    Since full rank matrices are stable under small perturbation, the previous equality holds for all $(W_\ell',\dots, W_1')$ near $(W_{\ell},\cdots, {W}_1)$. Therefore, $F_{\mathrm{n}}$ has constant rank near $( W_{\ell},\cdots,{W}_1)$, 
\end{proof}

Below is a technical lemma.

\begin{lemma}
   If $\min\{d_0,n\}\geq d_{\ell}$, then for almost every $(X,Y)\in \R^{d_0\times n}\times \R^{d_{\ell}\times n}$, $\rk(YP)=d_{\ell}$ where $P\in\cP^n_{\rk(X)}$ is the orthogonal projection matrix onto $\im(X^\top)$.    
\end{lemma}
\begin{proof}
   Note that $\R^{d_0\times n}\setminus\R^{d_0\times n}_*$ is a null set. For all $X\in \R^{d_0\times n}_*$, $\rk(YP)<d_\ell$ iff $\det(YPPY)=0$. Since zero set of a real analytic function is either a null set or the whole space, and by choosing the rows of $Y$ be linearly independent vectors in $\im(X^\top)$, which is possible since $\dim(\im (X^\top))=\min\{d_0,n\}\geq d_\ell$, we have $\det(YPPY^\top)=\det(YY^\top)\neq 0$. This shows that $\rk(YP)<d_\ell$ is a null set for all $X\in \R^{d_0\times n}_*$. By Tonelli's theorem, the set where the stated condition fails is a null set.
\end{proof}

The lemma shows that the assumption holds almost surely over the data matrices.

\begin{proposition}
\label{prop_constant_rank_lnn}
Let $P\in \cP^n_{\rk(X)}$ be the orthogonal projection matrix onto $\im(X^\top)$. If $\min\{d_1,\dots,d_{\ell-1}\}\geq d_0$ and $\rk(YP)=d_\ell$, then $f_{\mathrm{n}}$ has \KL~exponent $1/2$ at any global minimum.
\end{proposition}
\begin{proof}
  Since the row space of $W_\ell\cdots W_1X$ is contained in the row space of $X$ and $Y(I-P) X^\top=0$, for all $(W_{\ell},\dots,W_\ell)$, it holds that 
  \[ f_{\mathrm{n}}(W_{\ell},\cdots,W_1)=\|W_\ell\cdots W_1X- YP\|_F^2+\|Y(I-P)\|_F^2.    \]
  Since adding constant term does not change the \KL~exponent, without loss of generality, we assume $Y=YP$, and then $\rk(Y)=d_{\ell}$ and $\im(Y^\top)\subseteq \im(X^\top)$. The condition $\im(Y^\top)\subseteq \im (X^\top)$ means that the row space of $Y$ is contained in the row space of $X$. Together with $\min\{d_{\ell-1},\dots,d_1\}\geq d_0$, this yields $\arg\min f_{\mathrm{n}} =F^{-1}_{\mathrm n}(Y)$. For any $(\overline{W}_{\ell},\dots,\overline{W}_0)\in F_{\mathrm n}^{-1}(Y)$, we have 
      \[   d_{\ell}=\rk(Y)=\rk(\overline W_{\ell}\cdots \overline W_1X)\leq \rk(\overline W_{\ell}\cdots \overline W_{1})\leq d_\ell   \]
      Therefore, $\rk(\overline W_{\ell}\cdots \overline W_{1})=d_\ell$. By \cref{prop_constant_rank_lnn}, $F_{\mathrm{n}}$ has constant rank near $(\overline W_{\ell},\cdots, \overline{W}_1)$, and the result follows from \cref{coro_composite_strict}.
\end{proof}

\subsection{Structure of $XY=M$}

When $r\geq \rk(M)$, the global minima of $f_{\mathrm{a}}$ is clearly 
$$\Omega:=\{(X,Y)\in \R^{m\times r}\times \R^{r\times n}:~XY=M \}.$$
In this subsection, we study the structure of the solution set $\Omega$. In particular, we show that $\Omega$ is a finite union of the orbits of the action of a linear Lie group. We start with a technical lemma which allows us to simplify the structure of the elements in $\Omega$.
\begin{lemma}
\label{lemma_t1}
    Assume that $A\in \R^{s\times r}$ and $B\in \R^{r\times s}$ with $s\leq r$ and $\rk(AB)=s$. Then, there exists invertible matrix $S\in \R^{r\times r}$ such that 
    \begin{align*}
        AS =\begin{bmatrix}
            I_{s}& 0
        \end{bmatrix}    , S^{-1}B=\begin{bmatrix}
            AB \\
            0
        \end{bmatrix}. 
    \end{align*}
\end{lemma}
\begin{proof}
    Notice that 
    \[     s=\rk(AB)=\rk(B)-\dim \ker(A)\cap \im(B) \leq \rk(B)\leq s.        \]
    This implies that $\rk(B)=s$ and $\ker(A)\cap \im(B)=0$. Moreover, the rank inequality $s\leq \rk(AB)\leq \min\{\rk(A),\rk(B)\}\leq \rk(A)\leq s$ also gives $\rk(A)=s$. Using the fact that $\dim(\im(A))+\dim(\ker(A))=r$, we know that $\dim\ker(A)=r-s$ and $\dim\im(B)=s$. Therefore, $\R^r= \ker(A)\oplus \im(B)$, and we may select a basis $[v_1,\dots,v_s,~v_r]$ of $\R^r$ such that $\ker(A)=\spa\{v_{s+1},\dots,v_r\}$ and $\im(B)=\spa\{v_1,\dots,v_s\}$. Let $V:=[v_1,\dots,v_r]$, then $V\in \R^{r\times r}$ is invertible. Now, let $\widetilde B:=V^{-1}B$. By consider the equation $V\widetilde B=B$, we know $\widetilde B=\begin{bmatrix}
        B_1 \\
        0
    \end{bmatrix}$ with $B_1\in \R^{s\times s}$ being invertible. Similarly, setting $\widetilde A:=AS$, we immediately see that $\widetilde A=\begin{bmatrix}
        A_1  & 0
    \end{bmatrix}$ with $A_1$ being invertible. Then, it suffices to take $S:=V\begin{bmatrix}
        A_1^{-1}  & 0 \\
         0 & I_{r-s}
    \end{bmatrix}$.
\end{proof}
Below is the main result of this subsection. It provides the orbit structure of $\Omega$. We focus on the essential properties useful in this manuscript, but it is worth noting that the orbits actually form a Verdier stratification.
\begin{proposition}
\label{prop-tt1}
    Let $M\in \R^{m\times n}$ with $\rk (M)=s\leq r$. Set $\Omega:=\{X\in \R^{m\times r},~Y\in \R^{r\times n}:~XY=M\}$. Assume an SVD of $M$ is given by $$M=U\begin{bmatrix}
        \Alpha & 0 \\
        0 & 0 
    \end{bmatrix}V^\top$$ where $\Alpha\in \R^{s\times s}$ is a positive diagonal matrix. Consider the action of $G:=\rO(m-s)\times \GL (r)\times  \rO(n-s)$ on $\R^{m\times r}\times \R^{r\times n}$ given by 
$$ (Q_1,A,Q_2) \cdot (X,Y):=(\diag(I_s,Q_1)^\top X A, A^{-1}Y\diag(I_s,Q_2)).  $$
Given $p,q\in \N$ such that $\min\{m-s-p, n-s-q, r-s-p-q\}\geq 0$, consider the orbit
\begin{equation}
\label{standard_form}
    \orbit(p,q):= G\cdot \left( U\begin{bmatrix}
        I_s & 0 & 0 & 0\\
         0 & I_p & 0 & 0\\
         0 & 0 & 0 & 0 
    \end{bmatrix},\begin{bmatrix}
        \Alpha & 0 & 0 \\
         0 & 0  & 0  \\
         0 & I_q & 0 \\
         0 & 0 & 0 
    \end{bmatrix}V^\top \right), 
\end{equation}
where the inner matrix will be referred to as a standard form. Then
    \begin{equation}
    \label{decomp_omega}
        \begin{aligned}
          \Omega=\bigcup_{\scriptsize \begin{array}{c} p,q\in \N \\ p+q\leq r-s  \\
        p\leq m-s,~q\leq n-s
          \end{array}} \orbit(p,q)
        \end{aligned}
    \end{equation}
  and for any $(X,Y)\in \Omega$, the following are equivalent:
  \begin{enumerate}[label=\rm{(\roman{*})},ref=\rm{(\roman{*})}]
      \item \label{item:stru1}  $\Omega$ is locally a smooth embedded submanifold near $(X,Y)$;
      \item \label{item:stru2} $\Omega$ is locally a $C^1$ embedded submanifold near $(X,Y)$;
      \item \label{item:stru3}$(X,Y)\in \orbit(p,q)$ with $\min\{m-s-p, n-s-q,r-s-p-q\}=0$;
      \item \label{item:stru4} $f_{\mathrm{a}}$ satisfies the quadratic growth condition at $(X,Y)$.
  \end{enumerate}
\end{proposition}
\begin{proof}
   Without loss of generality, $U=I_m$ and $V=I_n$, in which case we also have $M=\begin{bmatrix}
       \Alpha & 0 \\
        0 & 0
   \end{bmatrix}$. First, let us assume that 
    \begin{align}
        X=\begin{bmatrix}
             X_{1 }\\
             X_2
        \end{bmatrix},~  Y=\begin{bmatrix}
             Y_1  & Y_2
        \end{bmatrix},\quad X_1\in \R^{s\times r},~Y_1\in \R^{r\times s},~X_2\in \R^{(m-s)\times r},~Y_2\in \R^{r\times (n-s)}.
    \end{align}
Using \cref{lemma_t1}, there exists invertible matrix $S\in \R^{r\times r}$ such that 
\begin{align*}
    \widetilde X: = XS = \begin{bmatrix}
        I_s & 0 \\
        X_{2,1} & X_{2,2}
    \end{bmatrix},\quad \widetilde Y:= S^{-1}Y=\begin{bmatrix}
        \Alpha & Y_{2,1} \\
          0 & Y_{2,2}
    \end{bmatrix}.
\end{align*}
Clearly, $\widetilde X\widetilde Y=XY=M$, which implies that 
\begin{align}
    \widetilde X = XS = \begin{bmatrix}
        I_s & 0 \\
        0 & X_{2,2}
    \end{bmatrix},\quad \widetilde Y =\begin{bmatrix}
        \Alpha & 0 \\
          0 & Y_{2,2}
    \end{bmatrix},\quad  X_{2,2}Y_{2,2}=0. 
\end{align}
Now, assume an SVD of $X_{2,2}$ is given by $X_{2,2}=U_1\begin{bmatrix}
    \Alpha_1 & 0 \\
     0 & 0
\end{bmatrix}V_1^\top$, where $\Alpha_1\in \R^{p\times p}$ is a positive diagonal matrix. Let $U_2:=\diag(I_s,U_1)$ and $R_1:=\diag(I_s,V_1)\in \mathrm{O}(r)$. Then, we have 
\begin{equation}
\begin{aligned}
        \hat X&:=U_2\widetilde X  R_1 \diag(\Alpha_1^{-1},I_{r-s-p})=\begin{bmatrix}
            I_s & 0 & 0 \\
             0 & I_p & 0 \\
              0 &  0 & 0 
        \end{bmatrix}, \\
        ~\hat Y&:= \diag(\Alpha_1,I_{r-s-p})R_1^\top \widetilde Y =\begin{bmatrix}
            \Alpha & 0   \\
               0 & Y_{3,1} \\
               0 & Y_{3,2}
        \end{bmatrix}=\begin{bmatrix}
            \Alpha & 0   \\
               0 &0 \\
               0 & Y_{3,2}
        \end{bmatrix},
\end{aligned}
\end{equation}
where for the last equality we have used the condition that $\hat X\hat Y = M$, we see that $Y_{3,1}=0$. Next, we take an SVD of $Y_{3,2}$ as $Y_{3,2}=U_3\begin{bmatrix}
    \Alpha_2 & 0 \\
    0 & 0 
\end{bmatrix}V_2^\top$, where $\Alpha_2\in \R^{q\times q}$ is a positive diagonal matrix. Consequently, we have 
\begin{equation}
    \begin{aligned}
        \widehat X &:= \hat X\diag(I_{s+p},U_3)\diag(I_{s+p},\Alpha_2,I_{r-s-p-q})=\hat X\\
        \widehat Y&:= \diag(I_{s+p},\Alpha_2^{-1},I_{r-s-p-q})\diag(I_{s+p},U_3^\top)\hat Y\diag(I_s,I_p,V_2)\\
        &=\begin{bmatrix}
         I_s & 0 & 0 \\
         0 & 0 & 0 \\
         0 & I_q & 0 \\
         0 & 0 &  0
        \end{bmatrix}. 
    \end{aligned}
\end{equation}
This finishes the proof of \cref{decomp_omega} by considering the composition of these linear transformation. Next, we prove the equivalence between \cref{item:stru1}--\cref{item:stru4}. That ``(i)$\implies$(ii)'' is clear. 
% since this Lie group action is smooth and semi-algebraic, and we can use \cref{fact_embeddedorbit} to show that any orbit is a smooth embedded submanifold. The implication (ii)$\implies$(iii) is trivial. 
Assume \cref{item:stru2}, and that \cref{item:stru3} does not hold. Then $p+q<r-s,~p<m-s$, and $q<n-s$. Without loss of generality, we may assume that 
\begin{equation} 
    \label{xy_standard}
     X=  \begin{bmatrix}
        I_s & 0 & 0 & 0\\
         0 & I_p & 0 & 0\\
         0 & 0 & 0 & 0 
    \end{bmatrix},\quad Y=\begin{bmatrix}
        \Alpha & 0 & 0 \\
         0 & 0  & 0  \\
         0 & I_q & 0 \\
         0 & 0 & 0 
    \end{bmatrix} , 
\end{equation}
    and we set \[ D_1:= U\begin{bmatrix}
        0 & 0 & 0 & 0\\
         0 & 0 & 0 & 0\\
         0 & 0 & 0 & \begin{bmatrix}
             I_{\min\{m,r\}-s-p} & 0 \\
             0 & 0
         \end{bmatrix} 
    \end{bmatrix} ,~D_2:= \begin{bmatrix}
        0 & 0 & 0 \\
         0 & 0  & 0  \\
         0 & 0 & 0 \\
         0 & 0 & \begin{bmatrix}
             I_{\min\{r,n\}-s-q} & 0 \\
             0 & 0
         \end{bmatrix}  
    \end{bmatrix}.  \] 
Then, by direct calculation, we have $(D_1,0),(0,D_2)\in T_{\Omega}(X,Y)$ and $(D_1,D_2)\notin T_{\Omega}(X,Y)$, which contradicts the assumption that $\Omega$ is locally a $C^1$ embedded submanifold near $(X,Y)$. Therefore, we have ``\cref{item:stru2}$\implies$\cref{item:stru3}''. 
% Next, we assume \cref{item:stru3}. 
% Therefore, we have proved that (i)--(iv) are mutually equivalent. The equivalence between (v) and (vi) follows from the definition.  
After that, we aim to prove \cref{item:stru3}$\implies$\cref{item:stru4}. If $m-s-p=0$ or $n-s-q=0$, then we have either $\rk(X)=s+p=m$ or $\rk(Y)=s+q=n$. In this case, by \cref{prop:mf_submersion} we know $d(F_{\mathrm{a}})_{(X,Y)}$ is surjective. Since $g_{\mathrm{a}}$ has growth exponent $2$ at $M$, by \cref{lemma:submersion}, $f:=g_{\mathrm{a}}\circ F_{\mathrm{a}}$ has growth exponent $2$ at $(X,Y)$. Thus, we only need to consider the case where $m-s-q>0,~n-s-q>0,~r-s-p-q=0$. Since any matrix $(X',Y')\in \Omega$ satisfies that $\rk(X')+\rk(Y')\leq s+r$, and the action of $G$ does not change the rank of $X$ and $Y$. We know that $\rk(X)+\rk(Y)=s+p+s+q=r+s$. Due to the fact that the rank function is lsc, we know that near $(X,Y)$, all the matrices $X'$ and $Y'$ with $(X',Y')\in\Omega$ have the same rank as $X$ and $Y$, respectively. Therefore, due to the rank stratification of $\Omega$ in \cref{standard_form}, $\Omega$ is locally an orbit near $(X,Y)$, which is $G\cdot (X,Y)$. Since the action of $G$ is semi-algebraic, this orbit is a smooth embedded submanifold of $\R^{m\times r}\times \R^{r\times n}$. Suppose \cref{item:stru4} fails. Then, we can find a subsequence $(X_k,Y_k)\to (X,Y)$ and $\sigma_k\downarrow 0$ such that 
\[    \|X_kY_k-M\|_F< \sigma_k d((X_k,Y_k), \Omega).         \]
Let the projection of $(X_k,Y_k)$ onto $\Omega$ be $\overline{X}_k,\overline{Y}_k$, and by \cref{fact_proj}, we have $t_k(H_k,K_k):=(X_k,Y_k)-(\overline{X}_k,\overline{Y}_k)\in N_{(X_k,Y_k)}\Omega$, where we assume $\|(H_k,K_k)\|=1$. By taking a subsequence if necessary, we may assume that $(H_k,K_k)\to (H,K)$, and due to the Clarke regularity of manifold \cite[Example 6.8]{rockafellar2009variational}, we know that $(H,K)\in N_{(X,Y)}\Omega$. On the other hand, we have
\begin{align*}
    \|X_kY_k-M\|_F=\|(\overline X_k+t_kH_k)(\overline{Y}_k+t_kK_k)-M\|_F=\|t_kH_k\overline{Y}_k+t_k\overline{X}_kK_k+ t_k^2H_kK_k\| <\sigma_k t_k.
\end{align*}
Dividing both side by $t_k$ and let $k\to\infty$, we know that 
$$XK+HY=0.$$
Using \cref{fact_orbit}, we see that
\begin{align*}
     T_{\Omega}(X,Y)=\fg\cdot (X,Y)=\{&(XC+\diag(0_s,A)X,-CY+Y\diag(0_s,B)):\\
     &~A=-A^\top\in \R^{(m-s)\times (r-s)},~B=-B^\top\in \R^{(r-s)\times (n-s)}     \}.  
\end{align*} 
Without loss of generality, we may assume that $(X,Y)$ has the standard form in \cref{xy_standard}, and since $p+q=r-s$, we may rewrite $(X,Y)$ as 
\begin{equation*}
     X=  \begin{bmatrix}
        I_s & 0 & 0 \\
         0 & I_p & 0 \\
         0 & 0 & 0
    \end{bmatrix},\quad Y=\begin{bmatrix}
        \Alpha & 0 & 0 \\
         0 & 0  & 0  \\
         0 & I_q & 0 
    \end{bmatrix} ,
\end{equation*}
Rewrite $H$ and $K$ in block forms:
\begin{equation*}
    H=\begin{bmatrix}
        H_1 & H_2 & H_3\\
        H_4 & H_5 & H_6 \\
        H_7 & H_8 & H_9 \\
    \end{bmatrix}, \quad K=\begin{bmatrix}
        K_1 & K_2 & K_3\\
        K_4 & K_5 & K_6 \\
        K_7 & K_8 & K_9 \\
    \end{bmatrix}.
\end{equation*}
Calculating the equation $XK+HY=0$, we have:
\begin{equation*}
    \begin{bmatrix}
        K_1 & K_2 & K_3 \\
        K_4 & K_5 & K_6\\
        0 & 0 & 0
    \end{bmatrix}+\begin{bmatrix}
        H_1\Alpha &  H_3 & 0 \\
        H_4\Alpha & H_6 & 0 \\
        H_7\Alpha & H_9 & 0
    \end{bmatrix} = 0
\end{equation*}
which yields
\[     K_1=-H_1\Alpha,~K_2=-H_3,~K_3=0,~K_4=-H_2\Alpha,~K_5=-H_6,~K_6=0,~H_7=0,~H_9=0.                  \]
Hence, we have 
\begin{equation*}
    H=\begin{bmatrix}
        H_1 & H_2 & H_3\\
        H_4 & H_5 & H_6 \\
        0 & H_8 &  0 \\
    \end{bmatrix},~K=\begin{bmatrix}
        -H_1\Alpha & K_2 & 0\\
        -H_4\Alpha & K_5 & 0 \\
        K_7 & K_8 & K_9 \\
    \end{bmatrix}.
\end{equation*}
We select $A\in \R^{(m-s)\times (r-s)},B\in \R^{(r-s)\times (n-s)},C\in \R^{r\times r}$ in the following way:
\begin{align*}
    A=\begin{bmatrix}
        0 & -H_8 \\
        H_8 & 0
    \end{bmatrix},\quad B=\begin{bmatrix}
        0 & K_9 \\
        -K_9 & 0
    \end{bmatrix},\quad C=\begin{bmatrix}
        H_1 & H_2 & H_3 \\
        H_4 & H_5 & H_6 \\
        -K_7\Alpha^{-1}& 0 & -K_8
    \end{bmatrix}.
\end{align*}
Then, it can be verified that 
\[    XC+\diag(0_s,A)X=H,\quad -CY+Y\diag(0_s,B)=K,~A=-A^\top,~B=-B^\top,\]
which proves that $(H,K)\in T_\Omega(X,Y)$. Since $(H,K)\in N_{(X,Y)}\Omega$, this implies that $(H,K)=0$, which yields contradiction. Finally, we prove ``\cref{item:stru4}$\implies$\cref{item:stru1}'', by \cite[Corollary 2.17]{rebjock2024fast} $f_{\mathrm{a}}$ has \KL~exponent $1/2$ at $(X,Y)$, and hence $\Omega$ is locally an analytic manifold around $(X,Y)$ since $f_{\mathrm{a}}$ is real analytic \cite{feehan2020morse}.
\end{proof}

\subsection{Matrix sensing}
\label{subsec:matrix-sensing}
In this subsection, we consider the matrix sensing problems in both the asymmetric and symmetric case.
\begin{definition}
    \label{def:matrixsensing}
    Given $m,r,n,p\in \N\p *$, $b\in \R^p$, and a linear map $\mathcal A:\R^{m\times n}\to \R^p$. Let $f_\mathrm{ms} := g_\mathrm{ms}\circ F_\mathrm{a}$ where 
    $$
    \begin{array}{rccc}
    g_\mathrm{ms}: & \R^{m\times n} & \mapsto & \R \\
     & A & \to & |\mathcal A(A)-b|^2.
    \end{array}
    $$
\end{definition}
\begin{definition}
    \label{def:symmatrixsensing}
    Given $r,n,p\in \N\p *$, $b\in \R^p$, and a linear map $\mathcal A:\bS^{n}\to \R^p$. Let $f_\mathrm{sms} := g_\mathrm{sms}\circ F_\mathrm{s}$ where 
    $$
    \begin{array}{rccc}
    g_\mathrm{sms}: & \bS^{n} & \mapsto & \R \\
     & A & \to & |\mathcal A(A)-b|^2.
    \end{array}
    $$
\end{definition}
Typically, in matrix sensing problem, the outer function $g_{\mathrm{ms}}$ or $g_{\mathrm{sms}}$ is assumed to have the restricted isometry property (RIP) \cite{zhao2023improving,tu2016low}, and a global minimum of rank $s\leq r$ is assumed.  Let us recall that, for $\delta >0$ and integers $s,t \ge 0$, a twice continuously differentiable function $g:\mathbb{R}^{m\times n} \to \mathbb{R}$ is said to satisfy $\delta$-$\RIP_{w,t}$ condition~\cite{li2017geometry,zhu2018global,zhang2021general} if for all $A,H\in \R^{m\times n}$ with $\rk(A)\leq w$ and $\rk(H)\leq t$, it holds that
\begin{equation*}
    \label{rip}
    (1-\delta) \|H\|_F^2  \leq  \nabla^2g(A)[H,H]\leq  (1+\delta)\|H\|_F^2.
\end{equation*}
In the following, we  will build a unified framework to handle the matrix sensing problems. The proposed framework also covers exactly and overparametrized matrix factorization with both the Frobenius norm and $\ell_1$ norm.

% the general composition function $f:=g\circ F_{\mathrm{a}}$ (resp. $f:=g\circ F_{\mathrm{s}}$) where $g$ has a unique global minimizer $M$ with $\rk(M)=s\leq r$. This covers over- and exactly-parametrized  matrix sensing, and matrix factorization with both $\ell_1$ and Frobenius norm. We start with the asymmetric case and then we reduce the symmetric case to the asymmetric case.

\subsubsection{Asymmetric case}
%\subsubsection{Over and exactly-parametrized}
With \cref{prop-tt1} in hand, we consider the optimization problem $\min_{X,Y} f$ with $f:=g\circ F_{\mathrm{a}}$. The case for symmetric parameterization will be reduced to the asymmetric case. That is 
\begin{equation}
    \label{eq_prob}
     \min_{X\in \R^{m\times r},~Y\in \R^{r\times n}} f(X,Y)=g(XY). 
\end{equation}
Assume the following assumptions:
\begin{enumerate}[label=(A\arabic*), ref=\textup{(A\arabic*)}]
  \item \label{A1} The function $g:\R^{m\times n}\to \R$ is convex.
  \item \label{A2} The point $M$ is a unique minimum of $g$ on $\R^{m\times n}$ with $\rk(M)=s\leq r$.
  \item \label{A3} There exist constants $c,\delta>0$ such that for all $A\in \R^{m\times n}_{\leq r}\cap B_{\delta}(M)$, it holds that
  $$g(A)\geq g(M)+c\|A-M\|_F^\beta.$$
  In addition, either $\beta\in [1,2)$, or $\beta=2$ and $\sup_{G\in \partial g(A)}\|G\|_F\leq \sigma_s(M)c/8 $ for all $A\in \R^{m\times n}_{\leq r}\cap B_\delta(M)$.
\end{enumerate}
For matrix sensing problem, the function $g$ is convex quadratic, and under the restricted isometric property (RIP) with rank $r+r^*$, we can verify \cref{A3} for $g$. It is clear that when $g$ satisfies the $\delta$-$\RIP_{s+r,s+r}$ condition for any $\delta>0$ and \cref{A2}, it also satisfies \cref{A3} with $\beta=2$. Moreover, for the case where $g(A)=\|A-M\|_1$, \cref{A3} holds with $\beta=1$.

 In this section, We aim to prove the following result:
\begin{theorem}
\label{thm:ms}
    Assume \cref{A1}--\cref{A3}. Let $(\overline{X},\overline{Y})$ be a global minimum of $f:=g\circ F_{\mathrm{a}}$. Let the group $G$ be defined in \cref{prop-tt1}. Then, the following holds. 
    \begin{itemize}
        \item[\rnum1] If $(\overline{X},\overline{Y})\in \orbit(p,q)$ with $p+q=\widetilde r-s$, then $f$ has \KL exponent $1-1/\beta$ at $(\overline{X},\overline{Y})$.
        \item[\rnum2]  If $(\overline{X},\overline{Y})\in \orbit(p,q)$ with $p+q<\widetilde r-s$, then $f$ has \KL exponent $1-1/(2\beta)$ at $(\overline{X},\overline{Y})$.
    \end{itemize}
\end{theorem}
The main issue for proving \cref{thm:ms} is that, while the solution set $\Omega$ is invariant under the action of $G$ defined in \cref{prop-tt1}, the function $f$ is not. Hence, \cref{cor:symmetry-lift} is not applicable. This technical difficulty is resolved by using the convexity of $g$ and \cref{thm:symmetry}. We start with an elementary inequality, which is essential in our proof, and appears to be new.
\begin{lemma}
\label{lemma-trace-bound}
    Given $A\in \R^{m\times n},~B\in \R^{n\times r},~C\in \R^{r\times m}$, we have 
    \[      |\tr(ABC)|\leq  \rk(B)^{1/4}\|AB\|_F^{\frac12} \|BC\|_F^{\frac12}\|CA\|_F^{\frac12}.        \]
\end{lemma}
\begin{proof}
    Assume an SVD of $B$ is given by $B=U\Sigma(B)V^\top$. Replacing $A$ by $AU$, $B$ by $\Sigma(B)$, $C$ by $V^\top C$ if necessary, we may assume that
    \[   B=\begin{bmatrix}
         D & 0 \\
         0 & 0 
    \end{bmatrix},~D:=\diag(b_1,\dots, b_s,0,\dots 0),~b_1\geq b_2\geq \dots \geq b_s>0.  \]
   Rewrite $A$ and $C$ in block forms:
   \begin{equation*}
       A=\begin{bmatrix}
           A_1 & A_3 \\
           A_2 & A_4
       \end{bmatrix},~C=\begin{bmatrix}
           C_1 & C_2 \\
           C_3 & C_4
       \end{bmatrix}.
   \end{equation*}
  It is clear that if we replace $B$ by $D$, $A$ by $\begin{bmatrix}
      A_1 \\
      A_2 
  \end{bmatrix}$, $C$ by $\begin{bmatrix}
      C_1 & C_2
  \end{bmatrix}$, then only $\|CA\|_F$ would reduce and other terms will remain the same. Therefore, we may assume that $B=D$, and rewrite $A$ and $C$ as 
  \[   A=\begin{bmatrix}
        a_1 & a_2 & \cdots & a_s
  \end{bmatrix}, ~C=\begin{bmatrix}
      c_1^\top \\
      c_2^\top \\
      \vdots \\
      c_s^\top 
  \end{bmatrix}.         \]
Then, we have 
\begin{subequations}
    \begin{align}
    \tr(ABC)&=\tr(DCA)=\sum_{i=1}^s b_i c_i^\top a_i,\label{37a} \\
    \|AB\|_F^2&=\sum_{i=1}^s  b_i^2|a_i|^2,~\|BC\|_F^2=\sum_{i=1}^s b_i^2|c_i|^2,~\|CA\|_F^2\geq \sum_{i=1}^s|c_i^\top a_i|^2. \label{37b}
\end{align}
\end{subequations}
Consequently, we can deduce that 
\begin{subequations}
    \begin{align}
      |\tr(ABC)|&\leq \sum_{i=1}^sb_i|c_i^\top a_i|=\sum_{i=1}^sb_i^{\frac12}|c_i^\top a_i|^{\frac12}|c_i^\top a_i|^{\frac12} b_i^{\frac12}  \label{38a}\\
      &\leq \sum_{i=1}^s|c_i^\top a_i|^{\frac12} b_i^{\frac12}|c_i|^{\frac12} b_i^{\frac12}|a_i|^{\frac12}  \label{38b}   \\
      &\leq \left(\sum_{i=1}^s |c_i^\top a_i|^2 \right)^{\frac14} \left(\sum_{i=1}^s b_i^2|a_i|^2\right)^{\frac14}  \left(\sum_{i=1}^s b_i^2|c_i|^2\right)^{\frac14}   s^{\frac14} \label{38c} \\
      &\leq \rk(B)^{\frac14} \|CA\|_F^{\frac12}\|AB\|_F^{\frac12}\|BC\|_F^{\frac12}. \label{38d}
    \end{align}
\end{subequations}
Indeed, in \cref{38a} we have used \cref{37a} and triangle inequality. In \cref{38b}, we have used Cauchy-Schwartz inequality to prove that $|c_i^\top a_i|\leq |c_i||a_i|$ for all $i\in [s]$. In \cref{38c}, we have used Holder inequality with exponent $4,4,4,4$. In \cref{38d}, we have used \cref{37b} and the fact that $\rk(B)=s$.
\end{proof} 
\begin{proof}[Proof of \Cref{thm:ms}]
 Since for real-valued convex function $g$, the Clarke subgradient $\cp g$ agrees with the subgradient $\partial g$, we can use \cite{bolte2017error}[Theorem 5 (ii)] to show that $g$ has \KL~exponent $1-1/\beta$ at $M$ by \cref{A3}. If $m=s+p$ or $n=s+q$, then $d(F_{\mathrm{a}})_{\overline{X},\overline{Y}}$ is surjective. Then, $f$ has \KL exponent $1-1/\beta$ at $(X,Y)$ by \cref{lemma:submersion}. Therefore, we assume $\min\{m-s-p,n-s-q\}>0$ in the following. We proceed with the following steps.
   
\noindent \textit{Step 1:} \textit{Reduction to the standard forms.} Assume that an SVD of $M$ is given by $M=U\Sigma(M)V^\top$. Replacing $g$ by $g(U\cdot V^\top)$ does not change \cref{A1}--\cref{A3} (in \cref{A2} the optimizer would be $\Sigma(M)$), and hence without loss of generality, we may assume that $M=\Sigma(M)$. Here, we assume 
\[   M=\begin{bmatrix}
    \Alpha & 0 \\
     0 & 0 
\end{bmatrix},      \]
where $\Alpha\in \R^{s\times s}$ is a positive diagonal matrix. Using \cref{prop-tt1}, there exist $A\in \GL(r)$ and $\widehat U\in \rO(m-s),~\widehat V\in \rO(n-s)$ such that $\diag(I_s,\widehat U)\overline{X}A^{-1}$ and $A\overline{Y}\diag(I_s,\widehat V^\top)$ have the standard forms in \cref{standard_form}. By replacing $g$ by $g(\diag(I_s,\widehat U)\cdot \diag(I_s,\widehat V))$ and $f$ by $f(\diag(I_s,\widehat U)\cdot \diag(I_s,\widehat V))$, we see that $f(X,Y)=g(XY)$ still holds, and \cref{A1}--\cref{A3} does not change. Notice that invertible linear transformation does not change \KL exponent, we may assume without loss of generality that 
  \begin{equation}
      \label{reduction_XY}
      \overline{X}=\begin{bmatrix}
          I_s & 0 & 0 & 0 \\
          0 & I_p & 0 & 0 \\
          0 & 0 & 0 & 0
      \end{bmatrix},\quad  \overline{Y}=\begin{bmatrix}
          \Alpha & 0 & 0\\
          0 & 0 & 0 \\
          0 & I_q & 0 \\
          0 & 0 & 0
      \end{bmatrix}.
  \end{equation}
\textit{Step 2:} \textit{Calculate the normal space.} Since $f$ is not invariant under the action of $G$, we need to consider the action of $\GL(r)$ on $\R^{m\times r}\times \R^{r\times n}$ given by 
\[   A\cdot(X,Y)=(XA^{-1},AY).       \]
Since this action is smooth semi-algebraic, by \cref{fact_embeddedorbit} we know every orbit of this action is smooth embedded submanifold of $\R^{m\times r}\times \R^{r\times n}$. Then, by differentiating this action and using $T_{\GL(r)\cdot (X,Y)}(X,Y)=\mathfrak{gl}(r)\cdot(X,Y)$ in \cref{fact_orbit}, we see that 
\[   T_{\GL(r)\cdot (X,Y)}(X,Y)=\{(-XK,KY):~K\in \R^{r\times r}  \}.     \]
Therefore, the normal space is the orthogonal complement of $T_{\GL(r)\cdot (X,Y)}(X,Y)$, which is given by 
\begin{equation}
    \label{normal_G}
    N_{\GL(r)\cdot (X,Y)}(X,Y)=\{(A,B)\in \R^{m\times r}\times \R^{r\times n}:~ X^\top A-BY^\top =0   \}.
\end{equation}

\noindent \textit{Step 3:} \textit{Construct a linear subspace $\cL$ such that $\cL+T_{\GL(r)\cdot(\overline{X},\overline{Y})}=\R^{m\times r}\times \R^{r\times n}$.} According to \cref{normal_G}, we know that $(A,B)\in N_{\GL(r)\cdot (\overline X,\overline  Y)}(\overline X,\overline Y)$ if and only if 
\begin{align*}
    \overline X^\top A -B\overline Y^\top =0,
\end{align*}
Now, rewrite $A$ and $B$ in block forms:
\begin{align*}
    A=\begin{bmatrix}
        A_1 & A_2 & A_3  & A_4 \\
        A_5 & A_6 & A_7 & A_8 \\
        A_9 & A_{10} & A_{11} & A_{12} 
    \end{bmatrix}, \quad   B=\begin{bmatrix}
        B_1 & B_2 & B_3 \\
        B_4 & B_5 & B_6 \\
        B_7 & B_8 & B_9\\
        B_{10} & B_{11} & B_{12}
    \end{bmatrix}.
\end{align*}
Then, we have 
\begin{align*}
  \begin{bmatrix}
     A_1 & A_2 & A_3  & A_4 \\
        A_5 & A_6 & A_7 & A_8 \\
       0 &  0  & 0  & 0 \\
       0 &  0  & 0  & 0 
  \end{bmatrix}    - \begin{bmatrix}
     B_1 \Alpha & 0 &  B_2 & 0 \\
     B_4\Alpha & 0 & B_5 & 0 \\
     B_7\Alpha & 0 & B_8 & 0 \\
     B_{10} \Alpha & 0 & B_{11} & 0 
  \end{bmatrix}=0.
\end{align*}
Solving this equation, we get that 
\begin{align*}
   A=\begin{bmatrix}
        B_1\Alpha & 0 & B_2  & 0 \\
        B_4\Alpha & 0 & B_5 & 0 \\
        A_9 & A_{10} & A_{11} & A_{12} 
    \end{bmatrix}, \quad   B=\begin{bmatrix}
        B_1 & B_2 & B_3 \\
        B_4 & B_5 & B_6 \\
        0 & 0 & B_9\\
        0 & 0 & B_{12}
    \end{bmatrix}.
\end{align*}
We construct $\cL$ in this way:
\begin{equation}
    \label{structrue_cl}
     (A,B)\in \cL \iff   A=\begin{bmatrix}
    0 & 0 & 0 & 0 \\
    A_5& 0 & 0 & 0 \\
    A_9 & A_{10} & A_{11} & A_{12} 
\end{bmatrix}   ,\quad  B=\begin{bmatrix}
        B_1 & B_2 & B_3 \\
        0 & B_5 & B_6 \\
        0 & 0 & B_9\\
        0 & 0 & B_{12}
    \end{bmatrix}.       
\end{equation}
In this case, we have 
\[   (A,B)\in \cL^\perp \iff A=\begin{bmatrix}
       A_1 & A_2 & A_3  & A_4 \\
       0 & A_6 & A_7 & A_8 \\
       0 &  0  & 0  & 0
  \end{bmatrix}   ,\quad  B=\begin{bmatrix}
        0 & 0 & 0 \\
        B_4& 0 & 0 \\
        B_7 & B_8 & 0 \\
        B_{10} & B_{11} & 0
    \end{bmatrix}.              \]
It can be verified that $\cL^\perp\cap N_{\GL(r)\cdot(\overline{X},\overline{Y})}(\overline{X},\overline{Y})=\{0\}$, and hence 
$$\cL+T_{\GL(r)\cdot(\overline{X},\overline{Y})}(\overline{X},\overline{Y})=\R^{m\times r}\times \R^{r\times n}. $$
Then, by \cref{thm:symmetry}, we know $f$ has \KL exponent $\alpha$ at $(\overline{X},\overline{Y})$, if we can prove that there exists $\eta,\rho>0$ such that:
\begin{equation}
\label{restricted_KL}
     \forall (X,Y)\in B_{\eta}(\overline{X},\overline{Y})\cap (\{(\overline{X},\overline{Y})\}+\cL),~ d(0,\cp f(X,Y))\geq \rho( f(X,Y)-f(\overline{X},\overline{Y}))^\alpha.       
\end{equation}

\noindent \textit{Step 3:} \textit{Use convexity to deduce the exponent.} For all $(X,Y)\in \{(\overline{X},\overline{Y})\}+\cL$, we rewrite $(X,Y)$ in the following block forms to simplify the calculation:
\begin{equation}
    \label{blockxythm}
    X=\begin{bmatrix}
        I & 0 \\
        X_1 & X_2
    \end{bmatrix},\quad    Y=\begin{bmatrix}
         \Alpha + Y_1 & Y_2 \\
         0 & Y_3
    \end{bmatrix}.
\end{equation}
Since $g$ is a real-valued convex function, we know $g$ is locally Lipschitz and regular by \cite[Corollary 8.41]{bauschke2017convex} and \cite[Proposition 2.3.6(b)]{clarke1990}. Using the chain rule in \cite[Proposition 7.11(b)]{facchinei2003finite}, we have 
\begin{equation}
\label{chain_rule_clarke}
   \cp f(X,Y)=\begin{bmatrix}
        \partial g(XY) Y^\top \\
        X^\top \partial g(XY) 
   \end{bmatrix}.
\end{equation}
Therefore, we may take arbitrary
$$G=\begin{bmatrix}
    G_1 & G_2 \\
    G_3 & G_4
\end{bmatrix}\in \partial g(XY) . $$ 
By direct calculation, we have
\begin{subequations}
\begin{align}
        \|GY^\top\|_F^2 &= \|G_1(\Alpha+Y_1^\top)+G_2Y_2^\top\|_F^2+\|G_2Y_3\|_F^2 \notag \\
        &\quad +\|G_3(\Alpha+Y_1^\top)+G_4Y_2^\top\|_F^2+\|G_4Y_3\|_F^2, \label{44a}  \\
        \|X^\top G\|_F^2&=\|G_1+X_1^\top G_3\|_F^2+\|G_2+X_1^\top G_4\|_F^2+\|X_2^\top G_3\|_F^2+\|X_2^\top G_4\|_F^2 ,\label{44c}~\\
        XY-M&=\begin{bmatrix}
        Y_1 & Y_2 \\
       X_1\Alpha +X_1Y_1  & X_1Y_2 + X_2Y_3 
    \end{bmatrix} \label{44d}
\end{align}
\end{subequations}
Using the convexity of $g$, we have 
\begin{subequations}
    \begin{align}
   & f(X,Y)-f(\overline{X},\overline{Y})=g(XY)-g(M)\leq  \langle G,XY-M \rangle   \label{45a}  \\
    &=\langle G_1,Y_1 \rangle +\langle G_2,Y_2  \rangle +\langle G_3,X_1(\Alpha+Y_1) \rangle+\langle G_4, X_1Y_2+X_2Y_3\rangle \label{45b} \\
    &=\langle  G_1+X_1^\top G_3,Y_1\rangle+\langle  G_2+X_1^\top G_4,Y_2  \rangle + \langle G_3,X_1\Alpha \rangle+\langle G_4,X_2Y_3 \rangle \label{45c} \\
    &\leq \|G_1+X_1^\top G_3\|_F\|Y_1\|_F+\|G_2+X_1^\top G_4\|_F\|Y_2\|_F+ \langle G_3,X_1\Alpha \rangle+\langle G_4,X_2Y_3 \rangle    \label{45d} \\
    &\leq 2\|X^\top G\|_F\|XY-M\|_F + \langle G_3,X_1\Alpha \rangle+\langle G_4,X_2Y_3 \rangle  \label{45e}  \\
    &\leq 2\|X^\top G\|_F\|XY-M\|_F + \langle G_3(\Alpha +Y_1^\top)+G_4Y_2^\top,X_1\rangle \notag \\
    &\quad -\langle G_3,X_1Y_1 \rangle+\langle G_4,X_2Y_3-X_1Y_2 \rangle  \label{45f} \\
    &\leq 2\|X^\top G\|_F\|XY-M\|_F + \|G_3(\Alpha +Y_1^\top)+G_4Y_2^\top\|_F\|X_1\|_F \notag \\
    &\quad \|G_3\|_F\|X_1\|_2\|Y_1\|_F +\langle G_4,X_2Y_3\rangle + \|G_4\|_F\|X_1\|_2\|Y_2\|_F \label{45g}   \\
    &\leq 2\|X^\top G\|_F\|XY-M\|_F + \|GY^\top\|_F\|X_1\|_F \notag \\
    &\quad \|G_3\|_F\|X_1\|_2\|XY-M\|_F +\langle G_4,X_2Y_3\rangle + \|G_4\|_F\|X_1\|_2\|XY-M\|_F \label{45h} 
    \end{align}
\end{subequations}
Indeed, \cref{45a} follows from the definition that $f(X,Y)=g(XY)$, $\overline{X}\overline{Y}=M$, and the convexity of $g$. \cref{45b} follows from the block structures of $G$ and $XY-M$ in \cref{44d}. \cref{45c} follows from rearranging those terms. \cref{45d} follows from Cauchy-Schwarz inequality. \cref{45e} follows from the block structures of $GY^\top$ and $XY-M$ in \cref{44c} and \cref{44d}. \cref{45f} follows from rearranging the terms. \cref{45g} follows from Cauchy-Schwartz inequality. \cref{45h} follows from the block structure of $GY^\top$ and $XY-M$ in \cref{44a} and \cref{44d}. Next, we provide a bound on $\|X_1\|_F$. When $2\|Y_1\|_2\leq \sigma_{\min}(\Alpha)=\sigma_s(M)$, we have 
\[    \|(\Alpha+Y_1)^{-1}\|_2=\frac{1}{\sigma_{\min}(\Alpha+Y_1)}\leq \frac{1}{1/2\sigma_{\min}(\Alpha)}=\frac{2}{\sigma_s(M)}.      \]
In this case, we have
\begin{equation}
\label{boundx1}
    \begin{aligned}
    \|X_1\|_2 &\leq \|X_1\|_F=\|X_1(\Alpha+Y_1)(\Alpha+Y_1)^{-1}\|_F \\
     &\leq \|X_1(\Alpha+Y_1)\|_F\|(\Alpha+Y_1)^{-1}\|_2\leq \frac{2}{\sigma_s(M)}\|XY-M\|_F,
    \end{aligned}
\end{equation}
where in the last inequality we have used the block structure of $XY-M$ in \cref{44d} to show that $\|X_1(\Alpha+Y_1)\|_F\leq \|XY-M\|_F$. Substituting this bound on $\|X_1\|_F$ into \cref{45h}, we can deduce that 
\begin{align*}
     f(X,Y)-f(\overline{X},\overline{Y})  &\leq 2\|X^\top G\|_F\|XY-M\|_F + \frac{2}{\sigma_s(M)}\|GY^\top\|_F\|XY-M\|_F  \\
    &\quad +\frac{2}{\sigma_s(M)}\|G_3\|_F\|XY-M\|_F^2 +\langle G_4,X_2Y_3\rangle \\
    & \quad +\frac{2}{\sigma_s(M)}\|G_4\|_F\|XY-M\|_F^2. 
\end{align*} 
In \cref{A3}, if $\beta<2$, then we have 
\[       \|XY-M\|_F^2=o(g(XY)-g(M))=o(f(X,Y)-f(\overline{X},\overline{Y})).                 \]
Since $g$ is locally Lipschitz, by \cite[Proposition 2.1.2(a)]{clarke1990}, we know that $\|G\|_F$ is bounded when $XY$ is near $M$. Reducing $\eta$ if necessary, for all $(X,Y)\in B_\eta(\overline{X},\overline{Y})$ we have 
\begin{subequations}
    \begin{align}
       \frac{2}{\sigma_s(M)}\|G_3\|_F\|XY-M\|_F^2&\leq \frac{1}{4}(f(X,Y)-f(\overline{X},\overline{Y})), \label{47a} \\
      \frac{2}{\sigma_s(M)}\|G_4\|_F\|XY-M\|_F^2&\leq \frac{1}{4}(f(X,Y)-f(\overline{X},\overline{Y})), \label{47b}
\end{align}
\end{subequations}
If $\beta=2$ and $\|G\|_F\leq c\sigma_s(M)/8$, then using \cref{A3}, when $\|XY-M\|_F<\delta$, we have 
$$ c\|XY-M\|_F^2\leq g(XY)-g(M)=f(X,Y)-f(\overline{X},\overline{Y}).  $$
This also yields that \cref{47a} and \cref{47b}. Consequently, reducing $\eta$ if necessary, we may assume that for all $(X,Y)\in B_{\eta}(\overline{X},\overline{Y})\cap (\{(\overline{X},\overline{Y})\}+\cL)$, it holds that 
\begin{equation}
\label{bound_funcv1}
    \begin{aligned}
      \frac{1}{2}\left( f(X,Y)-f(\overline{X},\overline{Y}) \right)&\leq   2\|X^\top G\|_F\|XY-M\|_F\\
      &~~~~+ \frac{2}{\sigma_s(M)}\|GY^\top\|_F\|XY-M\|_F  +\langle G_4,X_2Y_3\rangle.
    \end{aligned}
\end{equation}
 
 \noindent\textit{Step 4:} \textit{Consider the case where $(\overline{X},\overline{Y})\in \orbit(p,q)$ with $p+q=\widetilde r-s$.} Our next goal is to bound $\langle G_4,X_2Y_3\rangle$. When $p+q=\widetilde r-s$, the last column in the blocks of $A$ and the last row in the blocks of $B$ are null. In this case, we can rewrite $X_2,Y_3$ and $G_4$ in block forms:
\begin{equation}
    \label{blocks_fullrankcase}
  X_2=\begin{bmatrix}
      I_p & 0 \\
      X_3 & X_4
  \end{bmatrix},~~Y_3=\begin{bmatrix}
      Y_4 & Y_5 \\
      I_q & Y_6
  \end{bmatrix},~G_4=\begin{bmatrix}
      G_5 & G_6 \\
      G_7 & G_8
  \end{bmatrix}.  
\end{equation}
By direct calculation, we have 
\begin{subequations}
    \begin{align}
       X_2Y_3&= \begin{bmatrix}
          Y_4 & Y_5 \\
          X_3Y_4 + X_4 & X_3 Y_5 +X_4 Y_6
       \end{bmatrix},   \label{49a}  \\
      X_2^\top G_4 &=\begin{bmatrix}
         G_5+X_3^\top G_7 & G_6+X_3^\top G_8 \\
         X_4^\top G_7 &  X_4^\top G_8
      \end{bmatrix},  \label{49b} \\ 
      G_4Y_3^\top  &= \begin{bmatrix}
         G_5 Y_4^\top + G_6 Y_5^\top & G_5 + G_6 Y_6^\top  \\
         G_7 Y_4^\top + G_8 Y_5^\top  &  G_7+G_8 Y_6^\top
      \end{bmatrix}. \label{49c}
    \end{align}
\end{subequations}
Therefore, we have 
\begin{subequations}
    \begin{align}
            \langle G_4, X_2Y_3  \rangle &= \left\langle \begin{bmatrix}
      G_5 & G_6 \\
      G_7 & G_8
  \end{bmatrix}  ,\begin{bmatrix}
          Y_4 & Y_5 \\
          X_3Y_4 + X_4 & X_3 Y_5 +X_4 Y_6
       \end{bmatrix}    \right\rangle  \label{50a} \\ 
       &= \langle G_5+X_3^\top G_7, Y_4\rangle+\langle G_7+G_8Y_6^\top,X_4 \rangle +\langle G_6+X_3^\top G_8, Y_5\rangle  \label{50b}\\
       &\leq \|G_5+X_3^\top G_7\|_F\|Y_4\|_F+\|G_7+G_8Y_6^\top\|_F \|X_4\|_F + \|G_6+X_3^\top G_8\|_F\|Y_5\|_F \label{50c} \\
       &\leq \|X_2^\top G_4\|_F\|X_2Y_3\|_F+\|G_4Y_3^\top\|_F\|X_4\|_F+\|X_2^\top G_4\|_F\|X_2Y_3\|_F \label{50d}
    \end{align}
\end{subequations}
Indeed, in \cref{50a} we have used the block structures of $G_4$ and $X_2Y_3$ in \cref{blocks_fullrankcase} and \cref{49a}. In \cref{50b}, we have rearranged all the terms. In \cref{50c}, we have used Cauchy-Schwartz inequality. In \cref{50d}, we have used the block structures in \cref{49a}--\cref{49c}. Next, using the block structures of $X_2Y_3$ in \cref{49a}, we can deduce that 
\[   \|X_4\|_F-\|X_3Y_4\|_F     \leq \|X_3Y_4+X_4\|_F\leq \|X_2Y_3\|_F,                           \]
which can be rearranged as 
\[   \|X_4\|_F\leq \|X_2Y_3\|_F+\|X_3Y_4\|_F\leq \|X_2Y_3\|_F+\|X_3\|_2\|Y_4\|_F\leq (1+\|X_3\|_2) \|X_2Y_3\|_F.       \]
When $(X,Y)$ is close to $(\overline{X},\overline{Y})$, we know that $X_3$ is close to $0$. Therefore, we may assume that $\|X_3\|_2\leq 1$ for all $(X,Y)\in B_\eta(\overline{X},\overline{Y})$. Using \cref{50d}, we have 
\begin{equation}
    \label{g4bound1}
     \langle G_4,X_2Y_3 \rangle\leq  2( \|X_2^\top G_4\|_F+\|G_4 Y_3^\top \|_F) \|X_2Y_3\|_F.   
\end{equation}
Our next goal is to bound $\|X_2Y_3\|_F$. Using the block structure of $XY-M$ in \cref{44a}, we have
\[     \|X_2Y_3\|_F-\|X_1Y_2\|_F \leq \|X_1Y_2+X_2Y_3\|_F\leq \|XY-M\|_F .   \]
which can be rearranged as
\begin{subequations}
    \begin{align}
    \|X_2Y_3\|_F&\leq \|XY-M\|_F+\|X_1Y_2\|_F\leq  \|XY-M\|_F+\|X_1\|_2\|Y_2\|_F \label{51a}  \\
    &\leq \|XY-M\|_F+\frac{2}{\sigma_s(M)}\|XY-M\|_F\|Y_2\|_F   \label{51b} \\
    &\leq  \|XY-M\|_F+\frac{2}{\sigma_s(M)}\|XY-M\|_F^2.  \label{51c}
    \end{align}
\end{subequations}
\cref{51a} is obvious. \cref{51b} follows from the bound on $\|X_1\|_2$ in \cref{boundx1}. In \cref{51c}, we have used the block structure of $XY-M$ to show that $\|Y_2\|_F\leq \|XY-M\|_F$. Therefore, by reducing $\eta$ if necessary, we may assume that
\begin{equation}
    \label{upper_bound_x2y3}
  \forall (X,Y)\in B_\eta(\overline{X},\overline{Y}),\quad   \|X_2Y_3\|_F\leq 2\|XY-M\|_F.
\end{equation}
Consequently, from \cref{g4bound1} we have 
\begin{equation}
    \label{upper_bound_inner_g4}
    \begin{aligned}
         \langle G_4,X_2Y_3 \rangle&\leq 4( \|X_2^\top G_4\|_F+\|G_4 Y_3^\top \|_F) \|XY-M\|_F\\
         &\leq  4(\|X^\top G\|_F+\|GY^\top\|_F)\|XY-M\|_F,   
    \end{aligned}
\end{equation}
where the last inequality follows from \cref{44a} and \cref{44c}. Substituting the bound in \cref{upper_bound_inner_g4} into \cref{bound_funcv1}, we have 
\begin{subequations}
\begin{align}
       & \frac{1}{2}\left( f(X,Y)-f(\overline{X},\overline{Y}) \right)\leq  \left( 6\|X^\top G\|_F+ \left(\frac{2}{\sigma_s(M)}+4\right)\|GY^\top\|_F\right)\|XY-M\|_F  \\
        &\leq  \left( 6\|X^\top G\|_F+ \left(\frac{2}{\sigma_s(M)}+4\right)\|GY^\top\|_F\right)(f(X,Y)-f(\overline{X},\overline{Y}))^{1/\beta} c^{-1/\beta}. \label{57b}
\end{align}
\end{subequations}
\cref{57b} follows from \cref{A3}. Rearranging this inequality, we have proved that 
\begin{equation}
\begin{aligned}
        (f(X,Y)-f(\overline{X},\overline{Y}))^{1-1/\beta} &\leq c^{-1/\beta} \max\left\{6,\frac{2}{\sigma_s(M)}+4\right\}(\|X^\top G\|+\|GY^\top\|) \\
    &\leq 2c^{-1/\beta}c^{-1/\beta} \max\left\{6,\frac{2}{\sigma_s(M)}+4\right\}\left\|\begin{bmatrix}
        X^\top G\\
        GY^\top
    \end{bmatrix} \right\|_F, 
\end{aligned}
\end{equation}
which holds for any $G\in \partial g(XY)$, and hence by taking the infimum on all $G\in \partial g(XY)$ and the chain rule in \cref{chain_rule_clarke}, we can deduce that 
\[      (f(X,Y)-f(\overline{X},\overline{Y}))^{1-1/\beta}\leq 2c^{-1/\beta}c^{-1/\beta} \max\left\{6,\frac{2}{\sigma_s(M)}+4\right\} d(0,\cp f(X,Y)),                     \]
which proves \cref{restricted_KL}, and by \cref{thm:symmetry}, we see that $f$ has \KL exponent $\alpha=1-1/\beta$ at $(\overline{X},\overline{Y})$. 

\noindent \textit{Step 5:} \textit{Consider the case where $p+q<\widetilde r-s$.}  In this case, we use \cref{lemma-trace-bound} to give a direct bound on $\langle G_4,X_2Y_3 \rangle$:
\begin{subequations}
    \begin{align}
        \langle G_4,X_2Y_3 \rangle&=\tr(G_4^\top X_2Y_3)\leq \rk(X_2)^{\frac14}\|X_2^\top G_4\|_F^{\frac12} \|X_2Y_3\|_F^{\frac12} \|G_4Y_3^\top\|_F^{\frac12} \label{60a}\\
        &\leq r^{\frac14}\|X^\top G\|_F^{\frac12}\|GY^\top\|_F^{\frac12}\|X_2Y_3\|_F^{\frac12}   \label{60b}   \\
        &\leq   \sqrt2 r^{\frac14}\|X^\top G\|_F^{\frac12}\|GY^\top\|_F^{\frac12}\|XY-M\|_F^{\frac12}         \label{60c} \\
        &\leq     \frac{\sqrt2 r^{\frac14}}{2}(\|X^\top G\|_F+\|GY^\top\|_F)\|XY-M\|_F^{\frac12} .           \label{60d}
    \end{align}
\end{subequations}
In \cref{60a} we have used \cref{lemma-trace-bound}. In \cref{60b}, we have used the fact that $\rk(X_2)\leq r$, \cref{44a} and \cref{44c}. \cref{60c} follows from \cref{upper_bound_x2y3}. In \cref{60d} we have used the inequality that $xy\leq (x^2+y^2)/2$ for all $x,y\in \R$. By reducing $\eta$ if necessary, we may assume that for all $(X,Y)\in B_\eta(\overline{X},\overline{Y})$, $\|XY-M\|_F$ is sufficiently small such that 
\begin{equation}
\label{bound2}
\begin{aligned}
         &2\|X^\top G\|_F\|XY-M\|_F+\frac{2}{\sigma_s(M)}\|GY^\top\|_F\|XY-M\|_F \\
         &\leq  (\|X^\top G\|_F+\|GY^\top\|_F)\|XY-M\|_F^{\frac12}.
\end{aligned}
\end{equation}
Substituting \cref{bound2} and \cref{60d} into \cref{bound_funcv1}, we can deduce that
\begin{align*}
          \frac{1}{2}\left( f(X,Y)-f(\overline{X},\overline{Y}) \right) &\leq  \left(\frac{\sqrt{2}r^{\frac14}}{2}+1\right)(\|X^\top G\|_F+\|GY^\top\|_F)\|XY-M\|_F^{\frac12}  \\
          &\leq   \left(\frac{\sqrt{2}r^{\frac14}}{2}+1\right)(\|X^\top G\|_F+\|GY^\top\|_F)(f(X,Y)-f(\overline{X},\overline{Y}))^{\frac{1}{2\beta}}c^{-\frac{1}{2\beta}}, 
\end{align*}
where for the last inequality we have used \cref{A3}. Rearranging this inequality, we have 
\begin{align*}
    \left( f(X,Y)-f(\overline{X},\overline{Y}) \right)^{1-\frac{1}{2\beta}} &\leq (\sqrt{2}r^{\frac14}+1)c^{-\frac{1}{2\beta}} (\|X^\top G\|_F+\|GY^\top\|_F)   \\
    &\leq 2(\sqrt{2}r^{\frac14}+1)c^{-\frac{1}{2\beta}} \left\|\begin{bmatrix}
        X^\top G\\
        GY^\top
    \end{bmatrix} \right\|_F. 
\end{align*}
Since this inequality holds for all $G\in \partial g(XY)$, a similar argument as in step 4 proves \cref{restricted_KL} holds with $\alpha = 1-{1}/(2\beta)$, which proves that $f$ has \KL exponent $1-1/(2\beta)$ by \cref{thm:symmetry}.
\end{proof}
Let us now apply \cref{thm:ms} to asymmetric matrix sensing and overparametrized $\ell_1$ matrix factorization. 
\begin{corollary}
    Given $M\in \R^{m\times n}_{s}$ with $s\leq r$, let $f_1:\R^{m\times r}\times \R^{r\times n}\to \R$ be defined by 
    \begin{equation}
        f_1(X,Y):=\|XY-M\|_1. 
    \end{equation}
   Assume that $g_{\mathrm{ms}}$ satisfies the $\delta$-$\RIP_{s+r,s+r}$ condition for some $\delta>0$. Moreover, assume also that $M$ is a global minimum of $g$. Let $(\overline{X},\overline{Y})$ satisfies that $\overline{X}\hspace{.3mm}\overline{Y}=M$. Then,
    \begin{itemize}
        \item[\rnum1] If $(\overline{X},\overline{Y})\in \orbit(p,q)$ with $\min\{m-s-p,n-s-q,r-s-p-q\}=0$, then $f_1$ has \KL exponent $0$, and $f_{\mathrm{ms}}$ has \KL exponent $1/2$ at $(\overline{X},\overline{Y})$.
        \item[\rnum2]  If $(\overline{X},\overline{Y})\in \orbit(p,q)$ with $\min\{m-s-p,n-s-q,r-s-p-q\}>0$,  then $f_1$ has \KL exponent $1/2$, and $f_{\mathrm{ms}}$ has \KL exponent $3/4$ at $(\overline{X},\overline{Y})$.  
    \end{itemize}
\end{corollary}
\subsubsection{Symmetric case}
Next, we show that the symmetric case can be reduced to the asymmetric case. The next lemma builds an asymmetric version of a symmetric function. 
\begin{lemma}
\label{symmetric_ms}
Assume that $\widehat g_{\mathrm{s}}:\bS^n\to \R$ is convex. Let $\gamma>0$ and define
\[   \widehat g_{\mathrm{a}}:\R^{n\times n}\to \R,\quad \widehat g_{\mathrm{a}}(A):=\widehat g_{\mathrm{s}}\left(\frac{A+A^\top}{2}\right)+\gamma\|A-A^\top\|_F    \]
and 
\begin{align*}
    &\widehat f_{\mathrm{a}}:\R^{n\times r}\times \R^{r\times n}\to \R,~
\widehat f_{\mathrm{a}}:=\widehat g_{\mathrm{a}}\circ F_{\mathrm{a}}, \\
&\widehat f_{\mathrm{s}}:\bS^n\to \R,~\widehat f_{\mathrm{s}}:=\widehat g_{\mathrm{s}}\circ F_{\mathrm{s}}.
\end{align*}
Then for all $X\in \R^{n\times r}$,
\[
d\bigl(0,\cp \widehat f_{\mathrm{a}}(X,X^\top)\bigr)=\frac{1}{\sqrt{2}}\,d\bigl(0,\cp \widehat f_{\mathrm{s}}(X)\bigr),~\widehat f_{\mathrm{a}}(X,X^\top)=\widehat f_{\mathrm{s}}(X).
\]
 
\end{lemma}
\begin{proof}
\textit{Step 1:} \textit{Calculate the subdifferential of $\widehat g_{\mathrm{a}}$.} Because $\widehat g_{\mathrm{s}}$ is convex,  $\widehat g_{\mathrm{s}}$ is locally Lipschitz and regular by \cite[Corollary 8.41]{bauschke2017convex} and \cite[Proposition 2.3.6(b)]{clarke1990}. Also, $\widehat g_{\mathrm{a}}$ is convex, locally Lipschitz, and regular as a composition of convex function and linear map. Let us note that for convex functions the Clarke subdifferential agrees with its convex subdifferential \cite{drusvyatskiy2015curves}. Since $\widehat g_{\mathrm{s}}+\|\cdot\|_F$ is real-valued, we can apply the calculus rule for the convex subdifferential in \cite[Theorem 23.9]{rockafellar1970convex}:
\[\forall A\in \bS^n,\quad   \partial \widehat g_{\mathrm{a}}(A)= \{ G+\gamma(W-W^\top):~G\in \partial \widehat g_{\mathrm{s}}(A),~W\in \partial \|0\|_F\}.\]
Since $\partial \|0\|_F=B_1(0)$ by \cite[Exercise 8.27]{rockafellar2009variational}, we see that 
\begin{equation}
    \label{subg_g_s}
    \forall A\in \bS^n,\quad   \partial \widehat g_{\mathrm{a}}(A)= 
  \{G+W:~G\in \partial \widehat g_{\mathrm{s}}(A),~\|W\|_F\leq 2\gamma,~ W=-W^\top\}.  
\end{equation}

\textit{Step 2:} \textit{Calculate $\cp\widehat f_{\mathrm{a}}$ and $\cp\widehat f_{\mathrm{s}}$.} Since $\widehat g_{\mathrm{a}}$ and $\widehat g_{\mathrm{s}}$ are both regular, and $F_{\mathrm{a}}$ and $F_{\mathrm{s}}$ are both smooth, we can apply the chain rule in \cite[Proposition 7.1.11(b)]{facchinei2003finite} to show that 
\begin{align*}
    \cp\widehat f_{\mathrm{a}}(X,X^\top)&=d(F_{\mathrm{a}})_{(X,X^\top)}^* \partial \widehat g_{\mathrm{a}}(XX^\top), \\
    \cp \widehat f_{\mathrm{s}}(X)&= d(F_{\mathrm{s}})^*_{X}\partial \widehat g_{\mathrm{s}}(XX^\top).
\end{align*}
By direct calculation, we see that 
\begin{align*}
 \forall G\in \R^{n\times n},\quad   d(F_{\mathrm{a}})_{(X,X^\top)}^*(G)&= \big(GX,\;X^\top G\big),  \\
 \forall G\in \bS^n,\quad   d(F_{\mathrm{s}})^*_X(G)&=2GX.   
\end{align*}
Therefore, we have 
\begin{equation}
\label{cpfs}
    \cp \widehat f_{\mathrm{s}} (X) = \{2GX,~G\in  \partial \widehat g_{\mathrm{s}}(XX^\top) \},
\end{equation}
and 
\begin{equation}
\label{cpfa}
\begin{aligned}
    & \cp \widehat f_{\mathrm{a}}(X,X^\top) \\
     &=
    \{((G+ W)X,~X^\top(G+W)),~ G\in \partial\widehat g_{\mathrm{s}}(XX^\top),\|W\|_F\leq 2\gamma,~W=-W^\top \}. 
\end{aligned}
\end{equation}
\textit{Step 3:} \textit{Calculate $d(0,\cp\widehat f_{\mathrm{a}}$ and $d(0,\cp\widehat f_{\mathrm{s}})$.} Given the exact form of $\cp\widehat f_{\mathrm{s}}$ in \cref{cpfs}, we have 
\[  d(0,\cp\widehat f_{\mathrm{s}}(X))^2= \inf_{G\in \partial \widehat g_{\mathrm{s}}(XX^\top)} 4\|GX\|^2_F.       \]
For $d(0,\cp\widehat f_{\mathrm{a}})$, using \cref{cpfa} we have 
\begin{subequations}
    \begin{align}
        &d(0,\cp \widehat f_{\mathrm{a}}(X,X^\top))^2\notag \\
        &= \inf_{G\in \partial \widehat g_{\mathrm{s}}(XX^\top),\|W\|_F\leq 2\gamma,~W=-W^\top} \|G+W)X\|_F^2+\|X^\top (G+W)\|_F^2 \label{61a} \\
        &= \inf_{G\in \partial \widehat g_{\mathrm{s}}(XX^\top),\|W\|_F\leq 2\gamma,~W=-W^\top} \|GX\|_F^2+\|WX\|_F^2+\|X^\top G\|_F^2+\|X^\top W\|_F^2     \label{61b} \\
        &= \inf_{G\in \partial \widehat g_{\mathrm{s}}(XX^\top)} \|GX\|_F^2+\|X^\top G\|_F^2  \label{61c} \\
        &=\inf_{G\in \partial \widehat g_{\mathrm{s}}(XX^\top)} 2\|GX\|_F^2 = \frac{1}{2}d(0,\cp \widehat f_{\mathrm{s}}(X))^2. \label{61d}
    \end{align}
\end{subequations}
Indeed, \cref{61a} follows from the expression of $\cp f_{\mathrm{a}}(X,X^\top)$ in \cref{cpfa}. To justify \cref{61b}, we expand the expression in the following way:
\begin{subequations}
    \begin{align}
            &\|G+W)X\|_F^2+\|X^\top (G+W)\|_F^2 - \left(\|GX\|_F^2+\|WX\|_F^2+\|X^\top G\|_F^2+\|X^\top W\|_F^2\right) \notag \\
    &=2\langle GX, WX \rangle + 2\langle  X^\top G, X^\top W\rangle\label{62a} \\
    &=2\langle GX, WX \rangle + 2\langle  GX, W^\top X\rangle  \label{62b} \\
    &=2\langle GX, WX \rangle - 2\langle  GX, W^\top X\rangle=0. \label{62c}
    \end{align}
\end{subequations}
In \cref{62a}, we expand all the squares of the sums. In \cref{62b}, we have used the fact that $\langle A,B \rangle=\langle A^\top, B^\top \rangle$ for any matrices $A,B$ and the fact that $G$ is symmetric as $G\in \partial \widehat g_{\mathrm{s}}(XX^\top)$. \cref{62c} follows from the fact that $W=-W^\top$. This justifies \cref{61b}. In \cref{61c}, we take $W=0$ to achieve the minimal value. In \cref{61d}, we have used the fact that $\|A\|_F=\|A^\top\|_F$ for any matrix $A$. This proves the equality $d(0,\cp f_{\mathrm{a}}(X,X^\top))=d(0,\cp \widehat f_{\mathrm{s}}(X))/\sqrt{2}$ for all $X\in \R^{n\times r}$. The equality $\widehat f_{\mathrm{a}}(X,X^\top)=\widehat f_{\mathrm{s}}(X)$ is clear.  
\end{proof}
The next lemma discusses how to approximate a rank $r$ matrix using a symmetric rank $r$ matrix.
\begin{lemma}
\label{lemma_approx}
   For any $A\in \R^{n\times n}_r$, there exists $B\in \bS^{n}_r$, such that 
   \[    \left\|\frac{A+A^\top}{2}-B\right\|_F\leq \frac{1}{2}\|A-A^\top \|_F.      \]
\end{lemma}
\begin{proof}
    Define $S:=(A+A^\top)/2$ and $T:=(A-A^\top)/2$. If $\rk S\leq r$, then this is clear. Otherwise, we select $B\in P_{\R^{n\times n}_{\leq r}}(S)$, which happens to be in $\bS^n_r$ (this can be computed via eigenvalue decomposition of $S$ and keep $r$ eigenvalues with largest modulus). In this case, we have
    \begin{equation*}
         \|S-B\|_F=d(S,\R^{n\times n}_{\leq r})\leq \|A-S\|_F=\|T\|_F. \qedhere
    \end{equation*}
\end{proof}
Next, we show how to apply \cref{thm:ms} to the symmetric case by utilizing the construction given in \cref{symmetric_ms}. 
\begin{corollary}
\label{corollary_sym_ms}
   Let $g:\bS^n\to \R$ be convex, and satisfies the following conditions:
   \begin{itemize}
       \item[(i)] The point $M\in \bS^n_+$ is a unique minimum of $g$ on $\bS^n$ with $\rk(M)=s\leq r$. 
       \item[(ii)] There exists constants $c,\delta>0$ such that for all $A\in \bS^n_r\cap B_\delta(M)$ it holds that 
       \[    g(A)-g(M) \geq c\|A-M\|^\beta,    \]
       In addition, either $\beta\in [1,2)$ or $\beta=2$ and $g$ is strictly differentiable at $M$.
   \end{itemize}
   Then, $g\circ F_{\mathrm{s}}$ has \KL exponent $1-1/(2\beta)$ at any of its global minimum, that is, a matrix in $\{X\in \R^{n\times r}:~XX^\top=M\}$. 
\end{corollary}
\begin{proof}
   Since real-valued convex functions are locally Lipschitz, we may assume that $g$ is Lipschitz continuous on $B_\delta(M)$ with modulus $\gamma$. 
   \begin{itemize}
       \item When $\beta\in [1,2)$, we enlarge $\gamma$ to ensure that $\gamma\geq 2c$. 
       \item When $\beta=2$, due to the strict differentiability of $g$ at $M$, we have $\partial g(M)=\{0\}$. Then by the upper semicontinuity of $\partial g$ and reducing $\delta$ if necessary, we may select $\gamma\leq c\sigma_s(M)/32$. 
   \end{itemize}
   Then, we set 
   \[ \widehat g:\R^{m\times n}\to \R,\quad  \widehat g(A)=g\left( \frac{A+A^\top}{2}\right)+\gamma \|A-A^\top\|_F. \]
   Then \cref{A1} and \cref{A2} hold for $\widehat g$. To verify \cref{A3}, for all $A\in \R^{m\times n}_r\cap B_{\delta/2}(M)$, let $B\in \bS^n_r$ be given in \cref{lemma_approx}, then it holds that 
   \begin{subequations}
       \begin{align}
            \widehat g(A)-\widehat g(M)&=g\left(\frac{A+A^\top}{2}\right)+\gamma\|A-A^\top\|_F -g(M) \label{65a} \\
            &\geq  g(S) -g(M)-\left|g\left(\frac{A+A^\top}{2}\right)-g(S)\right| +\gamma\|A-A^\top\|_F  \label{65b} \\
            & \geq g(S) -g(M)-\gamma\left\|\frac{A+A^\top}{2}-S\right\|_F +\gamma\|A-A^\top\|_F \label{65c} \\
            &\geq g(S)-g(M)+ \frac{\gamma}{2}  \|A-A^\top\|_F      \label{65d} \\
            &\geq  c\|S-M\|_F^\beta+\frac{\gamma}{2}\|A-A^\top\|_F. \label{65e} 
       \end{align}
   \end{subequations}
   In \cref{65a}, we have used the definition of $\widehat g$. \cref{65b} follows from triangle inequality. \cref{65c} follows from the fact that 
   $$\|S-M\|_F\leq \left\|S-\frac{A+A^\top}{2}\right\|_F+\left\|\frac{A+A^\top}{2}-M\right\|_F\overset{\rm (a)}{\leq}\frac{1}{2}\|A-A^\top\|_F+\|A-M\|_F\overset{\rm (b)}{\leq} 2\|A-M\|_F,  $$
   where (a) follows from \cref{lemma_approx} and 
   $$ \|(A+A^\top)/2 -M\|\leq \|A-M\|_F/2 + \|A^\top - M\|_F/2 = \|A-M\|_F, $$  and (b) follows from 
   \begin{equation}
       \|A-A^\top\|_F\leq \|A-M\|_F+\|M-A^\top\|_F=2\|A-M\|_F.\label{bound_asym_A} 
   \end{equation}
   Then $\|S-M\|_F<\delta$ due to $\|A-M\|_F<\delta/2$. Utilizing the Lipschitz continuity of $g$ on $B_\delta(M)$ with modulus $\gamma$ we can obtain that $|g((A+A^\top)/2)-g(S)|\leq \gamma \|(A+A^\top)/2-S\|_F$. \cref{65d} follows from \cref{lemma_approx}. \cref{65e} follows from (ii) and the fact that $S\in B_\delta(M)$. Next, we consider the following two cases.
   
   \noindent Case 1: $\beta\in [1,2)$. In this case, we have 
       \begin{subequations}
           \begin{align}
             \|A-M\|_F^\beta&\leq (\|A-S\|_F+\|S-M\|_F)^\beta  \label{66a} \\
             &\leq 2^{\beta-1}( \|A-S\|_F^\beta+\|S-M\|_F^\beta) \label{66b} \\
             &\leq 2^{\beta-1}(\|A-A^\top\|_F^\beta+\|S-M\|_F^\beta). \label{66c}
           \end{align}
       \end{subequations}
    Here, \cref{66a} follows from triangle inequality. \cref{66b} follows from H\"older inequality. In \cref{66c}, we have used \cref{lemma_approx} to show that 
    \[   \|A-S\|_F\leq \|A-(A+A^\top)/2\|_F+\|(A+A^\top)/2-S\|_F\leq \|A-A^\top\|_F.   \]
   Reducing $\delta$ such that $\delta\leq 1$. By the fact that $A\in B_{\delta/2}(M)$ and \cref{bound_asym_A} we see that that $\|A-A^\top\|\leq 1$. Due to the condition that $\gamma\geq 2c$, we have 
    \[        c\|S-M\|_F^\beta+\frac{\gamma}{2}\|A-A^\top\|_F\geq c(\|S-M\|_F^\beta+\|A-A^\top\|_F^\beta)\geq c2^{1-\beta}\|A-M\|_F^\beta.           \]
    Consequently, \cref{A3} holds for $\widehat g$. 
    
    Case 2: $\beta=2$. In this case, we can also use \cref{66c} to show that 
    \begin{equation}
        \label{ambound}
        \frac{c}{2}\|A-M\|_F^2\leq  c\|A-A^\top\|_F^2+c\|S-M\|_F^2. 
    \end{equation}
    Reducing $\delta$ if necessary, we may assume that for all $A\in B_\delta(M)$ it holds that 
    \[   c\|A-A^\top\|_F^2 \leq \gamma/2\|A-A^\top\|_F.       \]
    In this case, by \cref{66c} and \cref{ambound} we have 
    \[  \forall A\in B_\delta(M),\quad  \widehat g(A)-\widehat g(M)\geq \frac{c}{2}\|A-M\|_F^2.        \]
    Moreover, for all $\widehat G\in \partial \widehat g(A)$, by the sum rule of convex functions \cite[Theorem 23.8]{rockafellar1970convex} we have $\widehat G=G+W$ with $G\in \partial g((A+A^\top)/2)$ and $W\in \gamma\partial\|(A-A^\top)\|_F\subseteq B_\gamma(0)$, and hence
    \[     \|\widehat G\|_F\leq \|G\|_F+ \|W\|\leq 2\gamma \leq c\sigma_s(M)/16. \]
    This proves \cref{A3}. 

    Then, applying \cref{thm:ms}, we know $\widehat f_{\mathrm{a}}:=\widehat g\circ F_{\mathrm{a}}$ has \KL exponent $1-1/(2\beta)$ at any of its global minimum. By \cref{symmetric_ms}, we see that $g\circ F_{\mathrm{s}}$ has \KL exponent $1-1/(2\beta)$ at any global minimum.
\end{proof}
Finally, we apply \cref{corollary_sym_ms} to symmetric $\ell_1$ matrix factorization and matrix sensing.
\begin{corollary}
   Given $M\in \bS^n_{+,s}$ with $s\leq r$, and assume $M$ is a global minimum of $g_{\mathrm{sms}}$. Moreover, suppose $g_{\mathrm{sms}}$ satisfies the $\delta$-$\RIP_{s+r,s+r}$ condition for some $\delta>0$. Consider the symmetric $\ell_1$ matrix factorization $f_1(X):=\|XX^\top -M\|_1$ and  the symmetric matrix sensing $f_{\mathrm{sms}}$. It holds that $f_{1}$ (resp. $f_{\mathrm{sms}}$) has \KL exponent $1/2$ (resp. $3/4$) at any of its global minimum. 
\end{corollary}

\appendix
\crefalias{section}{appendix}
\section{Appendix}
\subsection{Proof of \cref{fact:conversion_0}}

        Suppose there exists $x_k\to \overline{x}$ such that $f(x_k) < d(x_k,[f=0])/2\leq |x_k-\overline{x}|/2$. Let $f_k:\R^ n \to \overline{\R}$ be defined by $f_k(x) := f(x) + |x-x_k|/2$. Since $f_k$ is coercive and lsc, it admits a minimizer $y_k$. As $f$ is nonnegative, if $f(y_k) = 0$, then $d(x_k,[f=0])/2 \leq |y_k-x_k|/2 \leq f(y_k)+|y_k-x_k|/2\leq f(x_k)$, a contradiction. From the previous chain of inequalities, we see that $\max\{|y_k-x_k|/2,f(y_k)\}\leq f(x_k)\leq |x_k - \overline{x}|/2$ so $(y_k,f(y_k)) \to (\overline{x},f(\overline{x}))$. Since $f$ is lsc, \cite[Corollary 10.9]{rockafellar2009variational} yields $\partial f_k (x) \subseteq \partial f(x) + \partial g_k(x)$ for all $x\in \dom f$ where $g_k(x) := |x-x_k|/2$. Fermat's rule \cite[Theorem 10.1]{rockafellar2009variational} implies that $0 \in \partial f_k (y_k) \subseteq \partial f(y_k) + \partial g_k(y_k)$. Thus $1 \leq d(0,\partial f(y_k)) \leq \max\{|v|:~v\in \partial g_k(y_k)\}= 1/2$, a contradiction.

\subsection{Proof of \cref{lemma:submersion}}
Set $\overline{y}:=F(\overline{x})$. Replacing $g$ by $(g-g(\overline{y}))_+$ and $f$ by $(f-f(\overline{x}))_+$, we may assume that both $g$ and $f$ are nonnegative and $g(\overline{y})=f(\overline{x})=0$. Li and Pong's proof of the \KL exponent is actually for the subdifferential $\partial f$, not the Clarke subdifferential $\overline{\partial}f$. There is no difficulty in handling the Clarke subdifferential of course, but we give the proof for completeness.
%The chain rule for the \KL exponent follows the same pattern as in \cite[Theorem 3.2]{li2018calculus} and we state it here just for completeness. 
Since $dF$ is continuous and surjective at $\overline{x}$, there exists a constant $\gamma>0$ and a neighborhood $U$ of $\overline{x}$ such that 
 $$\forall (x,v)\in U\times \R^m,\quad   |dF_{x}^*v|\geq \gamma|v|.    $$
By the chain rule \cref{fact:change_coordinates}, for all $x\in U$ and $0<f(x)-f(\overline{x})<\ell$, we have 
 \[ d(0,\cp f(x))=d(0,dF_{x}^*\cp g(F(x)))\geq \gamma d(0,\cp g(F(x)))\geq c(g(F(x))^\beta = cf(x)^\beta. \]
The last inequality follows from the \KL exponent $\beta$ of $g$ at $\overline{y}$, after possibly reducing $U$.
 
W next prove the calculus rule for the growth exponent. By \cref{thm:LG}, there exist a constant $\kappa>0$ and a neighborhood $U$ of $(\overline{x},F(\overline{x}))$ such that
\[ \forall (x,y)\in U,~d(x,F^{-1}(y))\leq \kappa|F(x)-y|.     \]
Select a sufficiently small neighborhood $U_0$ of $\overline{x}$ such that $U_0\times P_{[g=0]}(F(U_0))\subseteq U$, where $P_{[g=0]}$ is well-defined near $\overline{y}$ since $[g=0]$ is closed. This is because $g$ is lsc around $\overline{y}$ and $\overline{y}$ is a local minimum of $g$. Utilizing the growth exponent of $g$ at $\overline{y}$, for all $x\in U_0$ and let $y=P_{[g=0]}(F(x))\in P_{[g=0]}(F(U_0))$, then we have $(x,y)\in U$ and 
\begin{align*}
         f(x)=g(F(x))&\geq  \gamma d(F(x),[g=0])^\beta=\gamma|F(x)-y|^\beta \\
         &\geq \frac{\gamma}{\kappa^\beta}d(x,F^{-1}[g=0])^\beta=\frac{\gamma}{\kappa^\beta}d(x,[f=0])^\beta, 
\end{align*}
where we take $U_0$ to be sufficiently small such that for all $z\in F(U_0)$ it holds that 
\[         g(z)-g(\overline{y})\geq  \gamma d(y,[g=0])^\beta.                   \]
This proves that $f$ has growth exponent $\beta$ at $\overline{x}$.
\subsection{Proof of Lemma~\ref{lemma:pham}}
\label{app:weak-to-strong}
Without loss of generality, $f$ is nonnegative,  $\overline{x}=0$ and $f(0)=0$.  Set $\alpha:=1-1/\beta\in[0,1)$. Since $0$ is a strict local minimum of $f$, there exists $r>0$ such that $[f=0]\cap B_r(0)=\{0\}$.
The growth exponent $\beta$ at $0$ therefore yields a constant $\kappa>0$ such that
\begin{equation}\label{eq:growth-bound}
\forall x\in B_r(0),\quad f(x) \ge \kappa |x|^\beta,
\end{equation}
after possibly shrinking $r$.
\textit{Step 1:} \textit{Localization and reduction to the globally subanalytic case.}
Define the truncated function
\[
\widetilde f \;:=\; f + \delta_{\overline{B}_r(0)} + \delta_{\{f\le 1\}}.
\]
Then $\widetilde f$ is lsc and subanalytic, and it coincides with $f$ on the set
\[
A:=\{x:|x|<r,\ f(x)<1\}.
\]
Moreover, for every $x\in A$ the graphs of $f$ and $\widetilde f$ locally coincide around $(x,f(x))$. By \cite[Theorem 8.9]{rockafellar2009variational}, it follows that $\partial f(x)=\partial\widetilde f(x)$, $\partial^\infty f(x)=\partial^\infty \widetilde f(x)$, and hence
\begin{equation}\label{eq:subdiff-coincide}
   \forall x\in A,\quad  \cp f(x)=\overline\co[\partial f(x) + \partial^\infty f(x)] = \overline\co[\partial\widetilde f(x)+\partial^\infty \widetilde f(x)]= \cp \widetilde f(x).
\end{equation}
Since $\gph \widetilde f$ is bounded and subanalytic, it is globally subanalytic.
Because \cref{eq:growth-bound} and \cref{eq:subdiff-coincide} show that the growth (resp. \L{}ojasiewicz) exponent is unaffected by this truncation, replacing $f$ by $\widetilde f$, $f$ is globally subanalytic.

\textit{Step 2:} \textit{Contradiction setup.}
Let $c\in(0,\beta\kappa^{1-\alpha})$.
Assume, for contradiction, that $f$ does not have \KL exponent $\alpha$ at $0$.
Then there exists a sequence $x_k\to 0$ with $f(x_k)\downarrow 0$ such that
\begin{equation}\label{eq:bad-seq}
d(0,\cp f(x_k)) < c f(x_k)^\alpha.
\end{equation}
Consider the globally subanalytic set
\[
S:=\{(x,t)\in \gph f:\ 0<t<1,\ d(0,\cp f(x))<ct^\alpha\}.
\]
By \cref{eq:bad-seq}, the point $(0,0)$ lies in $\overline{S}$.

\textit{Step 3:} \textit{Curve selection on a smooth stratum.}
By \cite[Lemma~8]{bolte2007lojasiewicz}, the set $S$ admits a Whitney stratification $\{X_i\}$ such that each stratum is a $C^2$ submanifold of $\R^{n+1}$ and $f$ is $C^2$ on the projection of each stratum to $\R^n$.
Pick a stratum $X$ whose closure contains $(0,0)$.
By the curve selection lemma, there exists a $C^2$ globally subanalytic curve
\[
\widetilde\gamma:(0,\varepsilon)\to X,\qquad \widetilde\gamma(t)\to (0,0)\ \text{as}\ t\downarrow 0.
\]
Writing $\widetilde\gamma(t)=(\gamma(t),f(\gamma(t)))$, we obtain a $C^2$ globally subanalytic curve $\gamma:(0,\varepsilon)\to \R^n$ such that \cref{eq:bad-seq} holds along $\gamma(t)$ for all small $t>0$:
\begin{equation}\label{eq:bad-curve}
d(0,\cp f(\gamma(t)))<cf(\gamma(t))^\alpha.
\end{equation}

\textit{Step 4:} \textit{Puiseux expansions and order comparison.}
Set $\varphi(t):=f(\gamma(t))$.
Since $\varphi$ and $\gamma$ are globally subanalytic, Fact~\ref{fact:puisuex} and Fact~\ref{fact:puiseux-derivative} yield (after shrinking $\varepsilon$) exponents $a,b>0$ and constants $A>0$, $v\neq 0$ such that, as $t\downarrow 0$,
\[
\varphi(t)=At^{a}+o(t^{a}),
\qquad
\gamma(t)=vt^{b}+o(t^{b}).
\]
In particular,
\[
| \gamma(t)| = | v|t^{b}+o(t^{b}),
\qquad
| \gamma'(t)| = b| v| t^{b-1}+o(t^{b-1}),
\qquad
\varphi'(t)= aA\, t^{a-1}+o(t^{a-1}).
\]
The growth bound \cref{eq:growth-bound} then gives
\begin{equation}
    \label{growth_coef}
    At^{a}+o(t^{a})=\varphi(t) \ge \kappa | \gamma(t)|^\beta
= \kappa | v|^\beta  t^{\beta b}+o(t^{\beta b}),
\end{equation}
hence $a\le \beta b$.
Second, since $f$ is $C^2$ on the projection of the chosen stratum $X$ onto $\R^n$, call it $B$, we have $\varphi'(t)=\langle \nabla_B f(\gamma(t)),\gamma'(t)\rangle$ and $\cp f(\gamma(t))\subseteq \{\nabla_B f(\gamma(t))\}+N_{\gamma(t)}B$ by \cite[Proposition 4]{bolte2007clarke}, where $\nabla_Bf$ is the Riemannian gradient of $f$ on $B$, so
\[
\varphi'(t)
\le | \nabla_B f(\gamma(t))|\,| \gamma'(t)|
\leq d(0,\cp f(\gamma(t)))\,| \gamma'(t)|.
\]
Combining with \cref{eq:bad-curve} yields
\[
\varphi'(t)\ \le\ c\,\varphi(t)^\alpha\, | \gamma'(t)|.
\]
Plugging the expansions into this inequality and comparing leading powers gives
\[
aA\,t^{a-1}+o(t^{a-1})
\le c\,A^\alpha\,t^{\alpha a}\,\bigl(b| v|\,t^{b-1}+o(t^{b-1})\bigr),
\]
so necessarily $a\ge \alpha a+b$, i.e., $a\ge \beta b$. Thus $a=\beta b$.

\textit{Step 5:} \textit{Leading coefficients and contradiction.}
Since $a=\beta b$, the coefficient comparison in \cref{growth_coef} implies $A\ge \kappa| v|^\beta$.
Using again $a=\beta b$ in the leading-term comparison of the derivative inequality gives
\[
aA  \le c A^\alpha  b| v|. \]
This implies that $\beta A^{1-\alpha} \le c| v|$. But $A\ge \sigma| v|^\beta$ implies
$$A^{1-\alpha}\ge \kappa^{1-\alpha} | v|^{\beta(1-\alpha)}=\kappa^{1-\alpha}| v|.$$
Therefore
\[
\beta\kappa^{1-\alpha}| v|  \le\beta A^{1-\alpha}\le c| v|,
\]
and hence $c\ge \beta \kappa^{1-\alpha}$, contradicting the choice of $c$. The contradiction shows that $f$ has \KL exponent $\alpha$ at $0$.

\subsection{Counterexamples}
\label{sec:ce}
\begin{example} 
    Let $f(x,y):=y^2+(x^2y-2y^2)_+$. Then the solution set of $f$ is $\{(x,y):~y=0\}$, which is an embedded submanifold of $\R^2$. Clearly, $f$ has growth exponent $2$ globally. Moreover, the function $f$ is semi-algebraic, locally Lipschitz and Clarke regular by \cite[Example 7.28]{rockafellar2009variational} as a max function of finitely many smooth functions. Moreover, when $x^2>2y$ and $y>0$, we know $f$ locally agrees with $y^2+y(x^2-2y)=x^2y-y^2$, which means that 
    \[   \nabla f(x,y)=\begin{bmatrix}
        2xy \\
        x^2-2y
    \end{bmatrix}.     \]
    Selecting arbitrary $t>0$ such that $x^2=2y+ty^2$ and $y>0$, we have
    \[   |\nabla f(x,y)|=\sqrt{ 4x^2y^2+(x^2-2y)^2}=\sqrt{4(2y+ty^2)y^2+t^2y^4}=O(y^{3/2})=O(f(x,y)^{3/4}),           \]
    which disproves that $f$ has \KL exponent $1/2$.
\end{example}
\begin{example}
    Define 
    \begin{equation*}
     f(x, y):= \begin{cases}|y|-\frac{3 x^2|y|^3}{|y|^3+16 x^6}, & (x, y) \neq(0,0), \\ 0, & (x, y)=(0,0) .\end{cases}
    \end{equation*}
    It is clear that $f$ is semi-algebraic. We next prove that $f$ is lower bounded by $(1-2^{-\frac{2}{3}})|y|$. To show this, notice that when $x=0$, we have 
    \[   f(x,y)=|y|\geq (1-2^{-\frac{2}{3}})|y|.        \]
    When $x\neq 0$, then there exists $z\geq 0$ such that $|y|=zx^2$, and by substituting $y=zx^2$ into the expression of $f$, we get that 
    \[  f(x,y)=|y|\left( 1-\frac{z^2x^{6}}{z^3x^6+16x^6} \right)=|y|\left(1-\frac{3z^2}{z^3+16} \right)       \]
    Using the inequality of arithmetic and geometric means, we have 
    \[  z^3+16=z^3/2+z^3/2+16\geq 3\sqrt[3]{4z^6}=3\sqrt[3]{4}z^2,        \]
    which implies that
    \[ f(x,y)\geq |y|\left(1-\frac{3z^2}{z^3+16}\right)\geq |y|\left( 1- \frac{1}{\sqrt[3]{4}}\right)=(1-2^{-\frac{2}{3}})|y|.        \]
    This proves that $f$ has growth exponent $1$ at $(0,0)$. Next, we verify local Lipschitz continuity of $f$. Let $\Omega$ consist of all the points where $f$ is differentiable. It is clear that $f$ is real analytic at $(x,y)$ when $y\neq 0$. If $f$ is differentiable at some $(x,0)$, then, since $f(x,0)=0$ and $f(x,y)\geq (1-2^{-\frac23})|y|$ holds globally, we know $\nabla f(x,0)=0$ by Fermat's rule. By considering the Taylor's expansion of $f$ at $(x,0)$, we see that the differentiability of $f$ at $(x,0)$ contradicts the sharp growth condition $f(x,y)\geq (1-2^{-\frac23})|y|$. Consequently, the set of points where $f$ is differentiable is exactly $\Omega=\{(x,y):~y\neq 0\}$. When $y\neq 0$, by direct calculation, we have 
    \begin{equation*}
    \partial_x f(x,y)= \frac{6x|y|^3 (32x^6 - |y|^3)}{(|y|^3 + 16x^6)^2}    ,\quad \partial_yf(x,y) = \mathrm{sign}(y) \left( 1 - \frac{144 x^8 y^2}{(|y|^3 + 16 x^6)^2} \right).
    \end{equation*}
    Our first task is to prove boundedness of both partial derivatives. When $x=0$, we see that $\partial_x f(x,y)=0$ and $\partial_yf(x,y)=1$. When $x\neq 0$, then by setting $z:=|y|/x^2$, we have 
    \begin{align*}
       | \partial_xf(x,y)|=\frac{6z^3|x|^7 |32x^6 - z^3x^6|}{(z^3x^6 + 16x^6)^2}=\frac{6z^3|x||32-z^3|}{(z^3+16)^2}\overset{\rm (a)}{\leq} \frac{6|x||32-z^3|}{z^3+16}\overset{\rm (b)}{\leq} 12|x|,
    \end{align*}
    where in (a) we have used the fact that $z\geq 0$ to prove ${z^3}/(z^3+16)\leq 1$, and in (b) to prove that $|32-z^3|\leq z^3+32\leq 2(z^3+16)$. For $\partial_y f(x,y)$ we have 
    \begin{align*}
        |\partial_yf(x,y)|=\left|1-\frac{144z^2x^{12}}{(z^3x^6+16x^6)^2}\right|= \left|1-\frac{144z^2}{(z^3+16)^2}\right|.
    \end{align*}
    Using the inequality of arithmetic and geometric means again, we have 
    \[   z^3+16=z^3+8+8\geq 3\sqrt[3]{64z^3}=12z,                   \]
    where equality holds if and only if $z=2$, which proves that $0\leq 144z^2/(z^3+16)^2\leq 1$, and hence  
    \[   |\partial_yf(x,y)|=  \left|1-\frac{144z^2}{(z^3+16)^2}\right|\leq 1.     \]
    Therefore, $\nabla f$ is bounded on any bounded subset of $\Omega$, and we can conclude that $f$ is Lipschitz continuous on $\{(x,y):~M\geq y>0,~|x|\leq M\}$ and $\{(x,y):~-M\leq y<0,~|x|\leq M\}$ as convex subsets of $\Omega$ for any $M>0$. The global estimation $|y|\geq f(x,y)\geq (1-2^{-\frac23})|y|$ proves that $f$ is continuous at $(x,0)$ for any $x\in \R$. This combined with the Lipschitz continuity of $f$ on $\{(x,y):~M\geq y>0,~|x|\leq M\}$ and $\{(x,y):~-M\leq y<0,~|x|\leq M\}$ proves that $f$ is Lipschitz continuous on the set $\{(x,y):~|x|\leq M,~|y|\leq M\}$. Therefore, we can conclude that $f$ is locally Lipschitz continuous.

   To see that $f$ does not have \KL exponent $0$ at $(0,0)$, let $x>0$ and $y=2x^2$, by direct calculation, we have 
   \[  d(0,\partial f(x,y))=|\nabla f(x,y)|=\sqrt{\partial_xf(x,y)^2+\partial_yf(x,y)^2}=\sqrt{4x^2+0}=2|x|.            \]
    which disproves the claim that $f$ has \KL exponent $0$ at $(0,0)$.
    
    Since at point where $f$ is $C^1$ the regular subdifferential $\widehat\partial f(x,y)=\nabla f(x,y)$ \cite[Exercise 8.8(a)]{rockafellar2009variational}, we then see that $\widehat\partial f=\nabla f$ on $\Omega$. To further determine $\widehat\partial f$ at point $(x,0)$, we calculate the subderivative $df(x,0)$. Since $f$ is locally Lipschitz continuous, we know the subderivative of $f$ agrees with the lower Dini directional derivative \cite[Exercise 9.15]{rockafellar2009variational}:
    \begin{align*}
        df(x,0)(d_x,d_y)&=\liminf_{t\downarrow 0}\frac{f(x+td_x,td_y)-f(x,0)}{t}=\liminf_{t\downarrow 0}\frac{f(x+td_x,td_y)}{t}\\
        &=\begin{cases}
            |d_y|\liminf_{t\downarrow 0}\left( 1-\frac{3(x+td_x)^2t^2d_y^2}{t^3|d_y|^3+16(x+td_x)^6} \right) & \text{if }  (d_x,d_y)\neq 0 \\
            0 & \text{if } (d_x,d_y)=0.
        \end{cases}
    \end{align*}
    Clearly, if $x\neq 0$, then we have $df(x,0)(d_x,d_y)=|d_y|$, and when $x=0$, we have
    \begin{align*}
             df(0,0)(d_x,d_y)=\begin{cases}
            |d_y|\liminf_{t\downarrow 0}\left( 1-\frac{3(td_x)^2t^2d_y^2}{t^3|d_y|^3+16(td_x)^6} \right)=|d_y| & \text{if }  (d_x,d_y)\neq 0 \\
            0 & \text{if } (d_x,d_y)=0.
        \end{cases}
    \end{align*}
    Since $\hat \partial f(x,y)=\{(v_x,v_y):~\forall (d_x,d_y)\in \R^2,~v_xd_x+v_yd_y\leq df(x,y)(d_x,d_y)\}$, we have proved that $\widehat\partial f(x,0)=(0,[-1,1])$ for all $x\in \R$. Next, we aim to prove the Clarke regularity of $f$ on $\R^2$. Since $f$ is $C^1$ near any point in $\Omega$, it is also Clarke regular at any point in $\Omega$ \cite[Exercise 8.20(a)]{rockafellar2009variational}. Therefore, it suffices to prove that $f$ is Clarke regular at $(x,0)$ for any $x\in \R$. In view of \cite[Corollary 8.11]{rockafellar2009variational}, that is to say we need to verify 
    \[  \partial f(x,0)=\widehat\partial f(x,0),~~\partial^\infty f(x,0)=\widehat\partial f(x,0).    \]
    Since we already know that $\widehat\partial f(x,0)=(0,[-1,1])$, we have $[\widehat\partial f(x,0)]^\infty=\{0\}$, where $$[\widehat\partial f(x,0)]^\infty=\{v:~\exists t_k\downarrow 0,~t_k\lambda_k\to v,~\lambda_k\in \hat \partial f(x,0)\},$$
    Moreover, by \cite[Theorem 9.13 (a) and (b)]{rockafellar2009variational}, the local Lipschitz continuity of $f$ implies that $\partial^\infty f(x,0)=\{0\}$. Hence, it suffices to prove that $\widehat\partial f(x,0)=\partial f(x,0)$. Let $(p,q)\in \partial f(x,0)$, then by definition \cite[Definition 8.3(b)]{rockafellar2009variational}, there exists $(p_k,q_k)\in \widehat\partial f(x^k,y^k)$ such that $(p_k,q_k)\to (p,q)$. If $y^k=0$ for infinitely many $k$, then we must have $(p,q)\in \widehat\partial f(x,0)$, since $\widehat\partial f(x,0)=(0,[-1,1])$ is the same for all $x\in \R$. Consequently, we may assume that $y_k\neq 0$ for all $k\in \mathbb N$, and in which case we have $(p_k,q_k)=(\partial_x f(x_k,y_k),\partial_y f(x_k,y_k))$ since $f$ is $C^1$ near $(x_k,y_k)\in \Omega$. The estimation $|\partial_y f(x_k,y_k)|=|q_k|\leq 1$ proves that $|q|\leq 1$. If $x=0$, then the estimation $|\partial_xf(x_k,y_k)|\leq 12|x_k|\to 0$ proves that $p=0$, and hence $(p,q)\in \widehat\partial f(0,0)$. Next, we assume $x\neq 0$. In this case, we have 
    \begin{align*}
        | \partial_x f(x_k,y_k)|= \frac{6|x_k||y_k|^3 |32x_k^6 - |y_k|^3|}{(|y_k|^3 + 16x_k^6)^2}\to 0,
    \end{align*}
    since $x_k\to x\neq0 $ and $y_k\to 0$, which also proves that $p=0$. Therefore, we have $(p,q)\in \widehat\partial f(x,0)$, and $f$ is Clarke regular at $(x,0)$ for all $x\in \R$. Finally, we can conclude that $f$ is Clarke regular on $\R^2$.
\end{example} 

\subsection{Proof of \cref{theorem:landscape}}
\label{subsec:proof_landscape}

    For all $(H,K) \in \R^{m\times r}\times \R^ {r\times n}$,
    \begin{align*}
        f_{\mathrm{a}}(X+H,Y+K) = & ~ \|(X+H)(Y+K)-M\|_F^ 2 \\
        = & ~ \|XY-M+XK+HY+HK\|_F^ 2 \\
        = & ~ \|XY-M\|_F^ 2+2\langle XY-M,XK+HY\rangle + \|XK+HY\|_F^ 2 + \\
        & ~ 2 \langle XY-M, HK \rangle_F + 2\langle XK+HY,HK\rangle  + \|HK\|_F^ 2 \\
        = & ~ f_{\mathrm{a}}(X,Y) +2\langle (XY-M)Y^ T,H\rangle +2\langle X^ T(XY-M),K\rangle_F +\\
        & ~ \|XK+HY\|_F^ 2 + 2 \langle XY-M, HK \rangle + o(\|H\|_F^ 2 +\|K\|_F^ 2) \\        
        = & ~ f_{\mathrm{a}}(X,Y) + \|XK+HY\|_F^ 2 + 2 \langle XY-M, HK \rangle + o(\|H\|_F^ 2 +\|K\|_F^ 2)
    \end{align*}
    where in the last equality, we assume that $(X,Y)$ is critical. With $\Delta  := XY-M$, the first-order optimality condition reads $\Delta Y^ T = 0$ and $X^ T \Delta=0$. Let $u\in \R^ m$ and $v\in \R^ n$ respectively denote left and right maximal singular vectors of $\Delta$. Naturally, $u\in \im \Delta \cap S^{m-1}$, $v\in \im \Delta^ T \cap S^{n-1}$, and $u^ T \Delta v = \|\Delta\|_2$. Since $Y\Delta^ T = 0$, we have $Yv = 0$. Let $(H,K) := (u a^ T, bv^ T)$ where $a,b\in \R^ r$. Compute
    \begin{align*}
        \|XK+HY\|_F^ 2 & = \|Xbv^ T+ua^ TY\|_F^ 2 = \|Xbv^ T\|_F^ 2+2\langle Xbv^ T,ua^ TY\rangle_F +\|ua^ TY\|_F^ 2 \\
        & = |Xb|^ 2|v|^ 2+2\langle Xb,ua^ TYv\rangle_F +|u|^ 2|a^ TY|^ 2 = |Xb|^ 2+|a^ TY|^ 2
    \end{align*}
    and
    $ \langle XY-M, HK \rangle_F  = \langle \Delta , u a^ Tbv^ T \rangle_F = \langle a,b\rangle \langle \Delta , uv^ T \rangle_F = \langle a,b\rangle u^ T \Delta v = \langle a,b\rangle \| \Delta\|_2.$
    
    By Sylvester's formula, $\rk X + \rk Y \leq \rk XY + r$. If $\rk XY<r$, then $\rk X<r$ or $\rk Y<r$. In the former case, by the rank theorem $r = \dim \ker X + \dim \im X$, so that $\dim \ker X \geq 1$. Let $w \in \ker X \cap S^{r-1}$ and $(a,b) := (tw,-w)$ where $t\in \R$. Observe that 
    $$|Xb|^ 2+|a^ TY|^ 2+ 2\langle a,b\rangle \| \Delta\|_2 = t^ 2|Xw|^ 2+ |w^ TY|^ 2 - 2t \langle w,w\rangle \|\Delta\|_2 = |w^ TY|^ 2 - 2t \|\Delta\|_2 < 0$$
    for all $t$ large enough unless $\Delta = 0$, in which case $(X,Y)$ is a global minimum of $f$. In the latter case, take instead $w \in (\im Y)^ \perp \cap S^ {r-1}$ and proceed similarly. 
    
    If $\rk XY = r$, then let $\overline{u}$ and $\overline{v}$ respectively be left and right minimal singular vectors of $XY$, with corresponding singular value $\sigma_r(XY)>0$. Take $(a,b)= (X^ T \overline{u},-Y\overline{v})$ and compute
    \begin{align*}
        |Xb|^ 2+|a^ TY|^ 2+ 2\langle a,b\rangle \| \Delta\|_2 & = |XY\overline{v}|^ 2+|\overline{u}^ TXY|^ 2 - 2\langle X^ T \overline{u} , Y\overline{v} \rangle \| \Delta\|_2 \\
        & = 2\sigma_r(XY)(\sigma_r(XY)-\|\Delta\|_2)< 0
    \end{align*}
    unless $\sigma_r(XY)\geq \|\Delta\|_2$, in which case $(X,Y)$ is a global minimum. 
%     Indeed, since $X^ T\Delta = 0$, $\im \Delta \subseteq \ker X^ T = (\im X)^{\perp} \subseteq (\im XY)^{\perp}$. Likewise, $\Delta Y^ T = 0$ implies $\im \Delta^ T \subseteq (\im (XY)^ T)^ \perp$. %Suppose $XY = U$. Then $\Delta$
% Given a singular value decomposition $X = U\Gamma_1 V^ T$, one thus has a block diagonal structure 
% $$\Sigma = \begin{pmatrix}
%     \Gamma_1 & 0 \\
%     0 & \Gamma_2
% \end{pmatrix}, ~~~ XY = U\begin{pmatrix}
%     \Gamma_1 & 0 \\
%     0 & 0
% \end{pmatrix}V^ T, ~~~\text{and}~~~ \Delta = U\begin{pmatrix}
%     0 & 0 \\
%     0 & \Gamma_2
% \end{pmatrix}V^ T,$$
% for some rectangular diagonal matrix $\Gamma_2$. As $M = XY-\Delta = U\Sigma V^ T$, the set of singular values of $M$ is the union of the set of singular values of $XY$ and $\Delta$. But since $\sigma_r(XY)\geq \|\Delta\|_2$, the singular values of $XY$ must be the maximal $r$ singular values of $M$. (Otherwise, $\Delta$ admits such a singular value and $\sigma_r(XY)< \|\Delta\|_2$.) 
Indeed, notice that 
\begin{align*}
   \det(\lambda^2 I-\lambda M^\top M)&= \det (\lambda^2I -\lambda(XY-\Delta)^\top(XY-\Delta))\\
    &= \det(\lambda^2 I-\lambda(Y^\top X^\top XY+\Delta^\top \Delta))   \\
    &=\det(\lambda^2 I-\lambda(Y^\top X^\top XY+\Delta^\top \Delta)+Y^\top X^\top XY\Delta^\top \Delta) \\
    &=\det(\lambda I - Y^\top X^\top XY)\det(\lambda I-\Delta^\top \Delta),
\end{align*}
which means that the nonzero singular values (counting multiplicities) of $M$ is the union of those of $XY$ and $\Delta$. Since $\sigma_r(XY)\geq \|\Delta\|_2$, we see that $\|\Delta\|_F^ 2 = \sigma_{r+1}^ 2 + \cdots + \sigma_m^ 2$. 

It follows that all the second-order stationary points have the same function value. Also, a global minimum exists since the set of bounded rank matrices $\R^{m\times n}_{\leq r}$ is closed and $\|\cdot-M\|_F^ 2$ is coercive. Since global minima are second-order stationary, $\min f_{\mathrm{a}} = \sigma_{r+1}^ 2 + \cdots + \sigma_m^ 2$ and $(X,Y)$ is globally optimal. This concludes the proof.

\subsection{Proof of \cref{corollary_landscape}}
\label{subsec:cor_landscape}
   Let $\Delta:=XX^\top -M$. First-order stationarity of $f_{\mathrm{s}}$ gives $\Delta X=0$. Then it is clear that $(X,X^\top)$ is a first-order stationary point of $f_{\mathrm{a}}$. Next, we prove that $\Delta\preceq 0$. For any $x\in \R^{n}$, we set $x=x_1+x_2$, where $x_1\in \im(X)$ and $x_2\in \im(X)^\perp=\ker(X^\top)$, and then we have 
   \[  x^\top \Delta x = (x_1+x_2)^\top\Delta (x_1+x_2)=x_2^\top \Delta x_2=x_2^\top (XX^\top -M)x_2=-x_2^\top Mx_2\leq  0.     \]
   Next, we focus on the second-order stationary equation. Similar to the calculation in \cref{theorem:landscape}, for all $H\in \R^{n\times r}$, it holds that 
   \[         \nabla^2f_{\mathrm{s}}(X)(H) =\|XH^\top + HX^\top\|_F^2+2\langle \Delta, HH^\top \rangle.              \]   
   As calculated in \cref{theorem:landscape}, for all $(A,B^\top)\in \R^{n\times r}\times \R^{r\times n}$, it holds that 
    \[    \nabla^2f_{\mathrm{a}}(X,X^\top)(A,B^\top)= \|XB^\top +AX^\top\|_F^2+2\langle \Delta, AB^\top\rangle.      \]
    Let $S:=(A+B)/2$ and $T:=(A-B)/2$. Then, we have $A=S+T$ and $B=S-T$, which further implies that 
    \begin{subequations}
        \begin{align}
             &\nabla^2f_{\mathrm{a}}(X,X^\top)(A,B^\top)\notag\\
             &= 2\langle \Delta, (S+T)(S-T)^\top\rangle+\|X(S-T)^\top +(S+T)X^\top\|_F^2 \label{27a} \\
             &=2\langle \Delta,SS^\top+TS^\top -TS^\top -TT^\top\rangle+\|XS^\top +SX^\top+(XT^\top -TX^\top)\|_F^2 \label{27b} \\ 
             &= 2\langle\Delta,SS^\top -TT^\top\rangle+\|XS^\top +SX^\top\|_F^2+ \|XT^\top -TX^\top\|_F^2 \label{27c} \\
            &\geq 2\langle \Delta,SS^\top\rangle+\|XS^\top +SX^\top\|_F^2\label{27d}
        \end{align}
    \end{subequations}
    Indeed, \cref{27a} and \cref{27b} follow from calculations. \cref{27c} follows from the fact that $\langle C,D \rangle=0 $ for symmetric $C$ and asymmetric $D$. In \cref{27d}, we have used the fact that $\Delta\preceq 0$ to show that $-\langle \Delta, TT^\top \rangle\geq 0$. Therefore, when $X$ is a second-order stationary point of $f_{\mathrm{s}}$, we see that $(X,X^\top)$ is a second-order stationary point of $f_{\mathrm{a}}$. Then, the result follows from \cref{theorem:landscape}. 

\bibliographystyle{abbrv}    
\bibliography{references}
\end{document}